\documentclass[11pt]{article}

\usepackage{amsrefs}
\usepackage{amsmath}
\usepackage{amsthm}
\usepackage{graphicx}
\usepackage{amsfonts}
\usepackage{amssymb}
\usepackage{color}
\usepackage{bm}
\usepackage[colorlinks]{hyperref}
\usepackage{enumerate}
\usepackage[margin=1.0in]{geometry}
\usepackage{enumerate}

\usepackage{xcolor}

\numberwithin{equation}{section}
\newtheorem{theorem}{Theorem}[section]
\newtheorem{lemma}[theorem]{Lemma}
\newtheorem{proposition}[theorem]{Proposition}
\newtheorem{corollary}[theorem]{Corollary}
\theoremstyle{definition} 

\newtheorem{definition}[theorem]{Definition}
\newtheorem{remark}[theorem]{Remark}
\newtheorem{assumption}{Assumption}

\newtheorem{innercustomthm}{Assumption}

  \theoremstyle{remark}
\newtheorem{step}{Step}
\usepackage{etoolbox}
\AtEndEnvironment{theorem}{\setcounter{step}{0}}
\newtheorem{case}{Case}
\usepackage{etoolbox}
\AtEndEnvironment{proof}{\setcounter{case}{0}}

\usepackage{etoolbox}
\AtEndEnvironment{theorem}{\setcounter{claim}{0}}

\newenvironment{customthm}[1]
{\renewcommand\theinnercustomthm{#1}\innercustomthm}
{\endinnercustomthm}

\usepackage{enumitem} 
\setlist{noitemsep}
\usepackage{mathtools} 
\usepackage{mathrsfs}  
\usepackage[mathcal]{euscript}

\renewcommand{\emptyset}{\varnothing}
\newcommand{\1}{\boldsymbol{1}}

\def\E{{\mathbb E}}
\def\N{{\mathbb N}}
\def\P{{\mathbb P}}
\def\R{{\mathbb R}}

\def\cP{{\mathcal P}}

\def\cX{{\mathcal X}}

\def\ra{\rightarrow}
\def\eqdist{\stackrel{(\mathrm{d})}{=}}
\renewcommand{\|}{\Vert}

\def\vk{\mathbb{V}_{n,k}}

\def\ln{L^n}
\newcommand{\norm}[1]{\left\|#1\right\|}
\newcommand{\abs}[1]{\left\lvert#1\right\rvert}
\newcommand{\Bva}[1]{\mathbb{B}^n_{2,V}[#1]}

\begin{document}

\begin{center}

\textbf{\LARGE An asymptotic thin shell condition and large deviations for random multidimensional projections} \\
\vskip 6mm
\textit{\large Steven Soojin Kim, Yin-Ting Liao, Kavita Ramanan\footnote{K. Ramanan and Y.-T. Liao were supported  by the National Science Foundation under grants DMS-1713032 and DMS-1954351}}\\
\vskip 3mm
{\large Brown University}\\
\end{center}

\begin{abstract}
It is well known that fluctuations of marginals of high-dimensional random vectors that satisfy a
  certain concentration estimate called the thin shell condition are approximately Gaussian. 
  In this article we identify  a general condition on  a sequence of high-dimensional random vectors
  under which one can identify the exponential decay rate of large deviation probabilities of the corresponding sequence of marginals.
More precisely,  
consider the projection of an $n$-dimensional random vector onto a random $k_n$-dimensional basis, $k_n \leq n$, drawn uniformly from the Haar measure on the Stiefel manifold of orthonormal $k_n$-frames in $\mathbb{R}^n$, in three different asymptotic regimes as $n \rightarrow \infty$: ``constant" ($k_n=k$), ``sublinear" ($k_n \rightarrow \infty$ but $k_n/n \rightarrow 0$) and ``linear" ($k_n/n \rightarrow \lambda$ with $0 < \lambda \le 1$). When the sequence of random vectors satisfies a certain ``asymptotic thin shell condition", we establish large deviation principles  for the corresponding sequence of random projections in the constant regime, and for the sequence of empirical measures of the coordinates of the random projections in the sublinear and linear regimes. We also establish large deviation principles for  scaled $\ell_q$ norms of the random projections in all three  regimes. Moreover, we show that the
asymptotic thin shell condition holds  for various sequences of random vectors of interest, including the uniform measure on suitably scaled $\ell_p^n$ balls, for $p \in [1,\infty)$, and generalized Orlicz balls defined via a superquadratic function, as  well as a class of  Gibbs measures with superquadratic interaction potential. 
Along the way, we obtain logarithmic asymptotics of volumes of high-dimensional Orlicz balls, which may be of independent interest.
    We also show that the decay rate of large deviation probabilities of  Euclidean norms of multi-dimensional 
    projections of $\ell_p^n$ balls, when $p \in [1,2)$, exhibits an unexpected phase transition in the sublinear regime, 
      thus disproving an earlier conjecture  due to Alonso-Guti\'{e}rrez et al. 
Random projections of high-dimensional random vectors are of interest in a range of fields including asymptotic convex geometry and high-dimensional statistics. 
\end{abstract}


\noindent
{\bf Subject Classification (MSC 2010)}: 60F10; 52A23; 46B06

\section{Introduction}

\subsection{Motivation and context}

The study of high-dimensional probability distributions through their lower-dimensional projections, especially random projections, is a common theme in a wide range of areas, including geometric functional analysis \cite{JohLin84},  statistics and data analysis \cite{FreTuk74,DiaFre84}, information retrieval \cite{IndMot99, PapRagTamVem00},  machine learning \cite{MaiMun09} and asymptotic  geometric analysis \cite{Kla07a,AntBalPer03}. In the latter case, the typical probability measure of interest is the uniform distribution on a high-dimensional convex body (i.e., a compact convex  set with non-empty interior). 
Questions about the geometry of convex bodies in high dimensions often take on a certain probabilistic flavor. A significant result in this direction is the so-called central limit theorem (CLT) for convex sets, which roughly says that most $k$-dimensional projections (equivalently, marginals) of an $n$-dimensional isotropic convex body are close to Gaussian in the total variation distance, when $n$ is sufficiently large and $k$ is of a smaller order than $n^\alpha$ for some universal constant  $\alpha \in (0,1)$. Although  foreshadowed by results of {Sudakov \cite{Sud78}, Diaconis and Freedman \cite{DiaFre84} and von Weizs\"acker \cite{vonWei97}},  {the  conjecture that
most marginals are Gaussian}  was precisely formulated by  Anttila,
Ball and  Perissinaki in \cite{AntBalPer03} (see also \cite{BreVoi00}). Specifically, they  showed  that if the Euclidean norm of a symmetric high-dimensional random vector $X^{(n)}$ satisfies a certain concentration estimate referred to as the ``thin shell" condition, then ``most" of its  marginals are approximately Gaussian. They also verified this condition 
for random vectors uniformly distributed on a certain class of convex sets whose modulus of convexity and diameter satisfy certain assumptions. Subsequently, the thin shell condition was verified for various classes of convex bodies by several authors, with a breakthrough verification due to Klartag \cite{Kla07a,Kla07}  for any isotropic log-concave distribution,
which, in particular, includes the uniform distribution on an isotropic convex body
(see also~\cite{Fre19} for a simplified proof). 
The result of \cite{AntBalPer03} was further extended to general high-dimensional measures by
Meckes \cite{Mec12b},  who showed that whenever an $n$-dimensional  random vector satisfies a quantitative
version of the thin shell condition, then most $k$-dimensional marginals are close to Gaussian
(in the bounded-Lipschitz distance) if $k < 2 \log n /\log \log n$, and that the latter cutoff for ${k}$
is in some sense the best possible.

 One broad aim of studying lower-dimensional projections is to obtain information about less tractable high-dimensional measures. While the central limit theorem for convex sets and related theorems are  beautiful universality results, they imply the somewhat negative result that (fluctuations of) most lower-dimensional projections do not provide much information about the high-dimensional measure. In contrast,  large deviation
 principles (LDPs), which characterize the rate of decay of tail probabilities in an asymptotically exact way,   are typically non-universal and distribution-dependent, and thus may  allow one to distinguish high-dimensional probability measures or convex bodies via their lower-dimensional projections.
Moreover, LDPs  are also useful for  computation of limits of  scaled logarithmic volumes (e.g., of Orlicz balls; see {also~\cite{KabPro21} and Remark \ref{rem-Orlicz}})
   and for the development of computationally efficient
 (importance sampling) algorithms for numerically estimating such  volumes  or other tail probabilities for finite $n$ 
 (see, e.g., \cite{LiaRam20a}). 
   Furthermore, they can also be used to obtain information on the conditional distribution of the high-dimensional measure,
   given that its projections deviate significantly from their means, via the so-called Gibbs conditioning principle
   (as elucidated in~\cite{Csiszar84} or~\cite[Section 3.3]{DemZeiBook}); {see \cite{KimRam18} or the more recent extension in  \cite{JohPro20}, for example,  for demonstrations} in a geometric context.
 
 The focus of this article is to identify general conditions on sequences of high-dimensional random vectors under which one can characterize  asymptotic tail probabilities of their multi-dimensional random projections, and obtain corollaries that may be of interest in convex geometry.
 More precisely, the goal is to establish LDPs for  random projections of general
 sequences of random vectors,  both
 for projections onto fixed lower-dimensional spaces, as well as onto spaces with dimension 
 growing with $n$. In the latter case,  when the dimension of the projected vectors grows with $n$, 
 the space in which one should look for an LDP for the sequence of projected vectors is {\em a priori}  unclear.
 We show that the empirical measure of the coordinates of the projected vector is a convenient object to look at, and allows us to establish LDPs in the space of probability measures on $\R$ in these cases.
 Unlike in the case of the CLT for convex sets, where a transition occurs at a projection
 dimension of $k_n = 2 \log n /\log \log n$ (or $k_n = n^\alpha$ when restricted to isotropic
 logconcave measures),  these LDPs hold for all growing $k_n$ as long as $k_n/n \rightarrow \lambda \in [0,1]$. 
 Additional motivation for studying projections onto growing subspaces
 arises   from the fact that the speeds and rate functions
  of such LDPs for scaled Euclidean norms of sequences of log-concave isotropic random vectors have
  implications for the Kannan-Lov\'{a}sz-Simonovits (KLS) conjecture,
  which is one of the major open problems in convex geometry
  (see \cite[Theorem A]{AloGutProTha20} for details of this connection, and also Remark \ref{rem-KLS}).

    \subsection{Contributions and outline of the paper}
    \label{subs-cont}

     In this article, we introduce a large deviation analogue of
the results in \cite{AntBalPer03}. Whereas the latter
  work shows that fluctuations of (most) random projections of high-dimensional vectors that satisfy a {\em thin shell condition} can be characterized (as almost Gaussian), our work characterizes tail behavior (at the level of annealed LDPs) for projections and their associated norms
  onto (possibly growing) random subspaces of high-dimensional random vectors that satisfy an {\em asymptotic thin shell condition}.  In particular, our work goes beyond the more studied specific setting of  measures on
  $\ell_p^n$ balls  or spheres, and also univariate LDPs, although we also obtain new results in this setting (see Remark \ref{rem-alogut}).  
 Specifically, for any sequence of random vectors $\{X^{(n)}\}_{n \in \mathbb{N}}$ whose
  scaled Euclidean norms satisfy an LDP 
  (see Assumption \ref{ass-normldp} and its specific case Assumption~\ref{ass-normldp}*),
  we characterize the tail behavior of the corresponding sequence of orthogonal projections of $X^{(n)}$ onto a random $k_n$-dimensional basis, $k_n \leq n$, drawn from the Haar measure on the  Stiefel manifold $\vk$ of orthonormal $k_n$-frames in $\mathbb{R}^n$, as the dimension $n$ goes to infinity {with $k_n/n \to \lambda \in [0,1]$.} 
  Assumption \ref{ass-normldp} (or rather, a slight strengthening of it)
    can be viewed as an  ``asymptotic thin shell" condition
   since it implies that for all sufficiently large $n$, the random vector $X^{(n)}$ satisfies the thin shell condition
    (see the discussion at the end of Section \ref{sec-regime}).  
 Note, however, that for growing subspaces,
  in contrast to  CLT results where  approximate Gaussian
  marginals holds only for $k_n < \frac{2\log n}{\log \log n}$ \cite{Mec12b}
  (or $k_n \sim n^\alpha$ \cite{Kla07} if one assumes additional regularity of
  $X^{(n)}$ such as logconcavity), the annealed LDP results
  indicate three crucial regimes for $\{k_n\}_{n\in\N}$, constant, sublinear
  and linear (see also \cite{AloGutProTha18} for $\ell_p^n$ balls).

     A summary of our main results  
 is as follows (the precise definition of LDPs, rate functions, and the Stiefel manifold are given in Section \ref{sec-not} and Section \ref{sec-mainres}):

\begin{enumerate}
\item {\it LDPs in the constant regime (Theorem \ref{th-constant}):  } 
  Given Assumption \ref{ass-normldp}*, or a modification of it stated as  Assumption~\ref{ass-rescaled},
  in the setting where $k_n = k$ for every $n$, we establish an (annealed) LDP  for the sequence of $k$-dimensional random projections of $X^{(n)}$.    
\item {\it LDPs in the sublinear and linear regimes (Theorems \ref{th-on} and \ref{th-simn}):}  
  Given Assumption \ref{ass-normldp}, in the setting where $\{k_n\}_{n\in\N}$ satisfies
  $k_n \rightarrow \infty$, we establish  (annealed) LDPs for the sequence of empirical measures 
  of the coordinates of the $k_n$-dimensional projections of $X^{(n)}$. 
  This LDP is established with respect to the $q$-Wasserstein topology for $q < 2$,
  which is stronger than the weak topology. 
  The rate function 
  is shown to have a different form  in the sublinear 
  ($k_n/n \rightarrow 0$) and  linear 
  ($k_n/n \rightarrow \lambda \in (0,1]$) regimes. 
\item
  {\it LDPs for norms of random projections (Corollary \ref{cor-normcon} and Theorems \ref{sub-mdp},  \ref{sub-ldp} and
  \ref{th-normlinear}):} 
  We  establish LDPs for sequences of  $\ell_q$ norms of the multi-dimensional random projections in
  all regimes, with two different scalings  considered in the sublinear regime.
 \item {\it Illustrative examples (Section \ref{sec-ex}):}
  To show that our theory unites disparate examples under a common framework,
  recovering, 
   and in some cases extending, existing results for $\ell_p^n$ balls,
   while also covering new examples, 
 we verify our assumptions 
  for  many  sequences $\{X^{(n)}\}_{n \in \N}$ of interest, including product measures, the uniform measure
  on certain scaled $\ell_p^n$ balls and  generalized  Orlicz balls,  and Gibbs measures
  (see  Remark \ref{rem-assnormldp}).
Along the way, we obtain several results of potentially independent interest. 
  Specifically, in    Theorem \ref{ple2} we show that when $p \in [1,2)$,  
            the LDP  for Euclidean norms of multi-dimensional 
  projections of $\ell_p^n$ balls  exhibits an interesting phase transition depending
  on whether the ratio  $k_n/n^{2p/(2+p)}$ is asymptotically finite or infinite, thus 
  disproving a  conjecture in  \cite{AloGutProTha18} (see Remark \ref{rem-alogut} for more details).
  In addition, we also obtain the asymptotic  logarithmic volume of Orlicz balls, {as elaborated
    in Remark \ref{rem-Orlicz}.   (See   also recent extensions due to~\cite{KabPro21} and \cite{LiaRam21a}).}  
\end{enumerate}

  The formulation of the correct form of the asymptotic thin shell condition that would also allow one to consider 
   $\ell^n_p$ balls, when $p\in[1,2)$, is somewhat subtle, and involves
    a suitable rescaling argument.   Verification of the asymptotic thin shell condition  for $\ell_p^n$ balls, for all $p \geq 1$,
    makes use  of a probabilistic
  representation for the uniform measure on $\ell_p^n$ balls {(see~\cite{RacRus91,SchZin90} and Section \ref{sec-lp} for details)}.  However, such probabilistic representations are not available 
   for Orlicz balls, and so we develop  a new approach in which 
      the tail probability is expressed as a volume ratio  
   (see~\eqref{vol-ratio} in the proof of Proposition~\ref{lem-orlex}), and  
    the asymptotics of this volume ratio is determined using a tilted measure with respect to Lebesgue measure.
    {Subsequent work that uses similar ideas to obtain the volumetric properties and the thin shell concentration for random vectors in Orlicz balls includes ~\cite{AloPro20, KabPro21}.  
      More recently,  
      detailed estimates of intersections and differences of Orlicz balls have been obtained in \cite{LiaRam20a}.} 
  We leave for future work the identification of more sequences $\{X^{(n)}\}_{n \in \N}$ of random vectors that satisfy the asymptotic thin shell (and related) conditions. 
   Furthermore,  although here we 
  focused on  Euclidean (and more general $\ell_q$) norms of the random
  projections, since they are of special relevance in
  asymptotic convex geometry,  our results could potentially be used to also 
  investigate  other symmetric functionals of lower-dimensional projections that are of interest,
  such as the volume or barycenter of the projected body, or other functionals of interest in high-dimensional statistics.

  \subsection{Summary of Prior Work}

 {First steps towards studying LDPs 
  of sequences of one-dimensional projections were taken in \cite{GanKimRam16}, \cite{skim-thesis} and \cite{GanKimRam18},
  where it was shown that such LDPs capture geometric information about the convex body.} 
However, LDPs for random projections have been first largely restricted to the setting
of high-dimensional product measures \cite{GanKimRam16} or the uniform measure on the (suitably renormalized)
unit ball  or sphere  in the space $\ell_p^n$ for some $p \in [1,\infty)$
  \cite{GanKimRam18,AloGutProTha18,AloGutProTha20,Kabluchko17,ThaKabPro19,LiaRam20a},   
  and secondly, limited to univariate LDPs (i.e., in $\mathbb{R}$) involving either the projection of a high-dimensional random vector onto a random one-dimensional subspace \cite{GanKimRam18} or (annealed) LDPs of the Euclidean norm of an orthogonal projection onto a $k_n$-dimensional subspace, with $k_n$ possibly tending to infinity \cite{AloGutProTha18,AloGutProTha20}. 
 Indeed,  the first paper to consider LDPs for norms of projections of scaled
  $\ell_p^n$ balls onto growing subspaces was \cite{AloGutProTha18},
  with further
  results obtained in subsequent papers (see, e.g., \cite{Kabluchko17,ThaKabPro19,AloGutProTha20}).
   {(Strictly speaking, in \cite{AloGutProTha18}  the authors analyze  norms of   random vectors  uniformly
      distributed on scaled $\ell_p^n$ balls projected onto $k$-dimensional subspaces indexed by the Grassmannian, whereas we analyze   norms of  random projections of more general random $n$-dimensional vectors (including those uniformly distributed
      on scaled $\ell_p^n$ balls) projected onto $k$-dimensional orthogonal bases indexed by  the Stiefel manifold. In the annealed setting considered here,  the Grassmanian and Stiefel manifold perspectives can be seen to be equivalent due to the exchangeability of the coordinates of the  projections, but in quenched settings, as considered in \cite{LiaRam21b},
      it is  more appropriate to consider the Stiefel manifold.)} 
 The only prior example of a multivariate LDP that we know in this context is for the particular case of projections of a random vector sampled from a scaled $\ell_p^n$ ball onto the first $k$ canonical directions \cite[Theorem 3.4]{BarGamLozRou10}.
  Further, in all cases, the analysis for non-product measures has focused on $\ell_p^n$ balls,
  where the analysis is greatly facilitated by a convenient probabilistic representation of the uniform measure on the $\ell_p^n$ ball (see \cite[Lemma 1]{SchZin90} or \cite{RacRus91}), or
  a slightly more general class of measures supported on the $\ell_p^n$ ball
  that admit a similar probabilistic representation
  (see~\cite{BarGueMenNao05} or Section 2.2 of~\cite{AloGutProTha20}).
  Such a representation has also been exploited
  in recent work \cite{LiaRam20a} that obtains  refined or sharp large deviations estimates
  for (quenched) random projections of $\ell_p^n$ balls and spheres.

 \subsection{Basic definitions and background results}
\label{sec-not}

We  set some initial notation and definitions, with a particular emphasis on large deviations terminology.
First, for   $a, b \in \R$, we will use $a \vee b$ and $a \wedge b$ to denote $\max(a,b)$ and $\min(a,b)$, respectively.
Next, for $p\in[1,\infty]$, let $\|\cdot\|_p$ denote the $\ell_p^n$ norm (with some abuse of notation, we use common notation for the $\ell_p^n$ norm on $\R^n$ for any $n\in\N$). We use the notation $\mathcal{N}(\mu,\sigma^2)$ to denote a normal random variable with mean $\mu$ and variance $\sigma^2$.

 Given a topological  space $\cX$ with Borel $\sigma$-algebra $\mathcal{B}$, let $\cP(\cX)$ denote the space of probability measures on $\cX$.  By default, we impose the topology of weak convergence on $\cP(\cX)$: recall that a sequence $\{\mu_n\}_{n\in\N} \subset \cP(\cX)$ is said to converge weakly to $\mu \in \cP(\cX)$, also denoted as $\mu \Rightarrow \mu$, if and only if for every bounded continuous function $f$ on $\cX$, $\int f d\mu_n \rightarrow \int f d\mu$ {as $n\to\infty$}. On occasion, when $\cX = \mathbb{R}^d$, we  will consider  subsets of probability measures that have certain finite moments. For {$q \geq 1$} and $d\in \N$, define   
\begin{equation*} \cP_q(\R^d) := \left\{ \nu \in \cP(\R^d) :   \int_{\R^d} |x|^q \nu(dx) < \infty \right\}.
\end{equation*}
Then, a sequence of probability measure $\{\nu_n\}_{n\in\N}\subset \cP_q(\R^d)$ converges to $\nu \in \cP_q(\R^d)$ with respect to the \emph{$q$-Wasserstein topology} if we have weak convergence $\nu_n\Rightarrow \nu$ and convergence of $q$-th moments $\int_{\R^d} |x|^q \nu_n(dx) \ra \int_{\R^d} |x|^q \nu(dx)$. In fact, as elaborated in \cite[Sec. 6]{Vil08}, the $q$-Wasserstein topology can be metrized by a distance function called the $q$-Wasserstein metric, which we denote $\mathcal{W}_q$. Next, for $q > 0$, let
\begin{equation}\label{secmom} M_q(\nu):= \int_{\R^d} |x|^q \nu(dx), \quad \nu \in \cP(\R^d),
\end{equation}
denote the $q$-th moment map. In our analysis, we will frequently consider the following subset: for $j \in \N$, define $K_{2,j} \subset\cP(\R^d)$ as \begin{equation}\label{k2} K_{2,j} := \left\{ \nu \in \cP(\R^d) : M_2 (\nu) \leq j \right\}. 
\end{equation}

\begin{lemma}\label{lem-compcat}
Fix $j\in \N$. For any {$q\in[1,2)$}, the set $K_{2,j}\subset \cP_2(\R^d)$ is compact with respect to the $q$-Wasserstein topology. In addition, $K_{2,j}$ is convex and non-empty.
\end{lemma}
\begin{proof}
The proof is a simple modification of the proof of the $j=1$ case given in \cite[Lemma 3]{KimRam18}. \end{proof}

\bigskip 

 We refer to \cite{DemZeiBook} for general background on large deviations theory.  In particular, we recall the definition: 

\begin{definition}\label{def-ldp}
Let $\cX$ be a {regular} topological space with Borel $\sigma$-algebra $\mathcal{B}$. A sequence $\{P_n\}_{n\in\N}\subset \cP(\cX)$  is said to satisfy a \emph{large deviation principle (LDP)} at \emph{speed} $s_n$ with rate function $I:\cX \ra [0,\infty]$ if {$I$ is lower-semicontinuous and} for all $\Gamma \in \mathcal{B}$,
\begin{equation*} -\inf_{x\in \Gamma^\circ} I(x) \le \liminf_{n\ra\infty} \tfrac{1}{s_n} \log P_n(\Gamma) \le \limsup_{n\ra\infty} \tfrac{1}{s_n} \log P_n(\Gamma) \le -\inf_{x\in \bar{\Gamma}} I(x),
\end{equation*}
where $\Gamma^\circ$  and $\bar{\Gamma}$ are the interior and closure of $\Gamma$, respectively. We say $I$ is a \emph{good rate function (GRF)} if, in addition,  it has compact level sets. Analogously, a sequence of $\cX$-valued random variables $\{\eta_n\}_{n\in\N}$ is said to satisfy an LDP {at speed $s_n$ and with a rate function $I$} if the sequence of their laws $\{\P \circ \eta_n\}_{n\in\N}$ satisfies an LDP {at the same speed and with the same rate function.} 
\end{definition}

{
\begin{remark}\label{ref-speeds}
Let $\{P_n\}_{n\in \N}$ be a sequence of probability measures satisfying an LDP at speed $s_n$ with GRF $I$ that has a unique minimizer $m$. It is immediate from the definition of the LDP that $I(m) = 0$ and that then, for any speed $t_n$ such that $t_n/s_n \rightarrow 0$,  the sequence $\{P_n\}_{n\in\N}$ also satisfies an LDP at speed $t_n$, but with  a degenerate rate function $\chi_m$, which is defined to be  $0$ at $m$ and $+\infty$ elsewhere. 
\end{remark}
}

{
\begin{remark}
  \label{rem-ldpscale}
  Let $\{P_n\}_{n\in \N}$ be a sequence of probability measures satisfying an LDP at speed $s_n$ with GRF $I$ and
  suppose $b_n$ is another sequence such that $s_n/b_n \rightarrow \lambda \in (0,\infty)$. Then it is immediate
  from the definition that $\{P_n\}_{n\in \N}$ also satisfies  an LDP at speed $b_n$ with GRF $\lambda I$. 
\end{remark}
}

As a useful tool, we recall the following definition.
\begin{definition}\label{def-exp} Let $\cX$ be a metric space with distance $d$,  equipped with its Borel $\sigma$-algebra.   Two sequences of $\cX$-valued random variables $\{\eta_n\}_{n\in\N}$ and $\{\tilde{\eta}_n\}_{n\in\N}$ are \emph{exponentially equivalent} at speed $s_n$ if for all $\delta > 0$,
\begin{equation} \label{expeq} \limsup_{n\ra\infty}\frac{1}{s_n} \log \P(d(\eta_n,\tilde{\eta}_n) > \delta) = -\infty.
\end{equation}
\end{definition}

\begin{remark}
  \label{rem-expeq}
  The notion of exponential equivalence is valuable because if an LDP holds for $\{\eta_n\}_{n\in\N}$, then an LDP holds for an exponentially equivalent sequence $\{\tilde{\eta}_n\}_{n\in\N}$ with the same GRF (see, e.g., Theorem 4.2.13 of \cite{DemZeiBook}).
  \end{remark}

{
\begin{remark}\label{rm-eq-conv}
  Let $\{a_n\}_{n\in\N}$ be a sequence in $\R$ that converges to $a\in\R$ as $n\to\infty$.  Suppose $\{U_n\}_{n\in\N}$ satisfies
  an LDP in $\R$ at speed $b_n$ with a GRF $I$.   Then the  sequences $\{a_nU_n\}_{n\in\N}$ and $\{aU_n\}_{n\in\N}$ are exponentially equivalent at speed $b_n$.  Indeed, for any $M \in (0,\infty)$, since $\{U_n\}_{n\in\N}$ satisfies an LDP with a GRF, there exists $K \in (0,\infty)$ such that $\limsup_{n\to\infty}\frac{1}{b_n}\log\mathbb{P}(|U_n| \geq K)<-M$. Given $\delta, M \in (0, \infty)$, pick $\varepsilon>0$ such that $\delta/\varepsilon >K$. Then
\begin{align*}
\lim_{n\to\infty}\frac{1}{b_n}\log\mathbb{P}\left(\abs{a_nU_n-aU_n}>\delta\right)
&\leq \lim_{n\to\infty}\frac{1}{b_n}\log \left(\mathbb{P}\left(\abs{a_n-a}>\varepsilon\right)+\mathbb{P}\left(\abs{U_n}>\delta/\varepsilon\right)\right)\\
& \leq \lim_{n\to\infty}\frac{1}{b_n}\log\mathbb{P}\left(\abs{U_n}>K\right)\\
&<-M.
\end{align*}
Since $M$ is arbitrary,  we obtain the desired exponential equivalence by sending $M\to\infty$.
\end{remark}
}

For some of our LDPs, the resulting rate functions will be expressed in terms of the following quantities.
For $\nu \in \mathcal{P}(\R)$, define the \emph{entropy} of $\nu$ as
\begin{equation}\label{ent} h(\nu) := -\int_{\R} \log\left(\tfrac{d\nu}{dx}\right) d\nu,
\end{equation}
for $\nu$ with density (with respect to Lebesgue measure $dx$), and $h(\nu):= -\infty$ otherwise. Furthermore, for $\nu, \mu \in \cP(\R)$, define the \emph{relative entropy} of $\nu$ with respect to $\mu$ as
\begin{equation}\label{relent} H(\nu | \mu) := \int_{\R} \log \left( \tfrac{d\nu}{d\mu}\right) d\nu 
\end{equation}
if $\nu$ is absolutely continuous with respect to $\mu$, and $H(\nu | \mu) := +\infty$ otherwise.
Given a function $\Lambda:\R\to\R$, let $\Lambda^*$ be its {Legendre-Fenchel} transform defined by
\begin{align} \label{legendre}
\Lambda^*(x) = \sup_{t\in\R} \{xt-\Lambda(t) \},\quad x\in\R.
\end{align}

The following contraction principle for LDPs is used multiple times throughout the paper.
\begin{lemma}[Contraction principle]\label{lem-con}\cite[Theorem 4.2.1]{DemZeiBook}
Let $\mathcal{X}$ and $\mathcal{Y}$ be Polish spaces and $f:\mathcal{X}\to\mathcal{Y}$ a continuous function. Suppose $\{X_n\}_{n\in\N}$ is a sequence of $\mathcal{X}$-valued random variables and satisfies an LDP in $\mathcal{X}$ at speed $s_n$ with GRF $\mathcal{I}_X$. Then $\{f(X_n)\}_{n\in\N}$ satisfies an LDP in $\mathcal{Y}$ at speed $s_n$ with GRF $\mathcal{I}_Y$ defined by
\[
\mathcal{I}_Y(y):= \inf \{\mathcal{I}_X(x):x\in\mathcal{X}, f(x)=y\}.
\]
\end{lemma}

Finally, we state a simple lemma that is used multiple times in our proofs.

\begin{lemma}\label{ldp-prod}
Suppose $\{U_n\}_{n\in\N}$, $\{V_n\}_{n\in\N}$ and $\{W_n\}_{n\in\N}$ satisfy LDPs in $\R$ at speed $\alpha_n$, $\beta_n$ and $\gamma_n$ with {GRFs} $J_U$, $J_V$ and $J_W$, respectively. Let $\alpha_n=\beta_n\ll \gamma_n$, {(i.e., $\beta_n/\gamma_n \to 0$ as $n\to\infty$)}. Assume $\{U_n\}_{n\in\N}$ is independent of $\{V_n\}_{n\in\N}$ and $J_W$ has a unique minimizer $m$. Then $\{(U_n,V_n,W_n)\}_{n\in\N}$ is exponentially equivalent to $\{(U_n,V_n,m)\}_{n\in\N}$ and satisfies an LDP at speed $\alpha_n$ with GRF {$J:\R^3\to [0,\infty]$ defined by}
\begin{align} \label{rate-J}
J(u,v,w):= 
\begin{cases}
J_U(u)+J_V(v), \quad & w =m,\\
+\infty,\quad & otherwise.
\end{cases}
\end{align}
Moreover, if $m\neq 0$, then $\{V_nW_n\}_{n\in\N}$ satisfies an LDP at speed $\alpha_n$ with GRF $v\mapsto J_V(v/m)$.
\end{lemma}
\begin{proof}
    Define $Y_n:=(U_n,V_n,W_n)$ and $\tilde{Y}_n:=(U_n,V_n,m)$.
By the independence of $\{U_n\}_{n\in\N}$ and $\{V_n\}_{n\in\N}$, $\{\tilde{Y}_n\}_{n\in\N}$
satisfies an LDP at speed $\alpha_n$ with GRF $J$ defined in~\eqref{rate-J} {(this is easily
  deduced from the definition of the LDP, but the fact that the 
  rate function for $(U_n, V_n)$ is $J_U(u) + J_V(v)$ follows from  \cite[Exercise 4.2.7]{DemZeiBook} with
  ${\mathcal X} = \R$, ${\mathcal Y} = \R^2$ and $F$ being the identity mapping; the stated 
  rate function for $(U_n,V_n,m)$ follows as an immediate consequence).}  
Now, for $\epsilon > 0$, we have  
\begin{align*}
\lim_{n\to\infty}\frac{1}{\alpha_n}\log\mathbb{P}\left(\norm{Y_n-\tilde{Y}_n}_2>\epsilon\right) &= \lim_{n\to\infty}\frac{1}{\alpha_n}\log\mathbb{P}(\abs{W_n-m}>\epsilon)
=-\infty,
\end{align*}
{where the last equality follows because  $J_W$ has a unique minimizer at $m$} and $\{W_n\}_{n\in\N}$ satisfies an LDP at speed $\gamma_n\gg \alpha_n$, {as in the observation of Remark \ref{ref-speeds}}. Thus, $\{Y_n\}_{n\in\N}$ is exponentially equivalent to $\{\tilde{Y}_n\}_{n\in\N}$ at speed $\alpha_n$.  By Remark~\ref{rem-expeq},  this implies $\{Y_n\}_{n\in\N}$  satisfies an LDP
with speed $\alpha_n$ with the same GRF as $\{\tilde{Y}_n\}_{n \in \N}$. This proves
the first assertion of the lemma. 

The second assertion follows by applying the contraction principle to the mapping $F:\R^3\to\R$ defined by $F(u,v,w)=vw$.
\end{proof}

\section{Main results}
\label{sec-mainres}

\subsection{Projection Regimes and Assumptions}\label{sec-regime}
For each $n \in \mathbb{N}$, consider a random vector   $X^{(n)}$  
that takes values in  $\R^n$.  For $k \in \N$, let $I_k$ denote the $k\times k$ identity matrix,
and for $n > k$, let $\vk = \{A \in \R^{n\times k} : A^T A = I_k\}$ denote the Stiefel manifold
of orthornomal $k$-frames in $\R^n$. {Here, the constraint $A^TA=I_k$ ensures that the columns of $A$ are orthonormal.}  {As is well known, $\vk$ is a compact subset of $\R^{n \times k}$, being the pre-image of the closed set $\{I_k\}$ under the continuous map   $A \mapsto A^T A$ that is also bounded (by $\sqrt{k}$ with respect to the Frobenius norm).} 
We are interested in the orthogonal projection of $X^{(n)}$ onto a random $k_n$-dimensional subspace, where $1\le k_n < n$. To this end, fix $n\in\N$, $1 \le k_n \leq n$, and let 
\begin{equation*} \mathbf{A}_{n,k_n} = \left[ \mathbf{A}_{n,k_n}(i,j) \right]_{i=1,\dots, n; \, j=1,\dots, k_n}
\end{equation*}
be an $n \times k_n$ random matrix drawn from the Haar measure on the Stiefel manifold $\mathbb{V}_{n,k_n}$ (i.e., the unique probability measure on $\mathbb{V}_{n,k_n}$ that is invariant under orthogonal transformations).
Note that the random matrix  $\mathbf{A}^T_{n,k_n}$    linearly projects a vector from $n$ to  $k_n$ dimensions. 
We assume that for each $n \in \N$,  $\mathbf{A}_{n,k_n}$ is independent of $X^{(n)}$, and 
for simplicity,
we  assume that the sequences 
$\{X^{(n)}\}_{n \in \N}$ and  $\{\mathbf{A}_{n,k_n}\}_{n \in \N}$ are defined on a common  probability space  $(\Omega,\mathcal{F},\P)$, 
although dependencies across $n$ are immaterial for the questions we address.

We aim to analyze the large deviation behavior of the coordinates of random projections $\mathbf{A}_{n, k_n}^T X^{(n)}$ of $X^{(n)}$ in
three regimes, constant, linear and sublinear, depending on how the dimension $k_n$ of the projected vector changes
with $n$:

\begin{definition} \label{def-set} 
Given a sequence $\{k_n\}_{n\in\N}\subset \N$, we say:
\begin{enumerate} 
\item $\{k_n\}_{n\in\N}$ is  \emph{constant}  at $k$, for some $k \in \N$,
  denoted  $k_n\equiv k$, if $k_n = k$ for all $n\in\N$;
\item $\{k_n\}_{n\in\N}$   \emph{grows sublinearly}, denoted  $1\ll k_n \ll n$,  if $k_n \ra \infty$ but $k_n/n \ra 0$;
\item $\{k_n\}_{n\in\N}$ \emph{grows linearly} with rate $\lambda$, for some $\lambda \in (0,1]$, 
  denoted  $k_n \sim \lambda n$, if $k_n/n \ra \lambda$. 
\end{enumerate}
\end{definition}

When $\{k_n\}_{n\in\N}$ is constant  at some $k\in\N$,  
then one can investigate when the sequence of vectors $\{\mathbf{A}_{n, k_n}^T X^{(n)} = \mathbf{A}_{n, k}^T X^{(n)}\}_{n\in\N}$ satisfies an LDP
in the space $\mathbb{R}^k$. In contrast, when $k_n$ increases to infinity, in order to even pose
the question of existence of an LDP,
one must first embed the sequence $\{\mathbf{A}_{n,k_n}^T X^{(n)}\}_{n\in\N}$  of 
 random vectors  of different dimensions 
into a common topological space. Thus, we prove an LDP for the sequence $\{L^n \}_{n\in\N}$ of empirical measures of the coordinates of the $k_n$-dimensional random vectors $\mathbf{A}_{n,k_n}^T X^{(n)}$:
\begin{equation}\label{eq-empproj} \ln := \frac{1}{k_n}\sum_{j=1}^{k_n} \delta_{(\mathbf{A}_{n,k_n}^TX^{(n)})_j} \,, \quad n\in\N.
\end{equation}

\begin{remark}
Note that the law of  $\mathbf{A}_{n,k_n}$ is invariant under permutation of its $k_n$ columns, so the $k_n$ coordinates of $\mathbf{A}_{n,k_n}^T X^{(n)}$ are exchangeable. Thus, the empirical measure  $\ln$ in  \eqref{eq-empproj} encodes the essential distributional properties of the coordinates of the projection, and hence, serves as a natural choice for a common infinite-dimensional embedding of the $k_n$ coordinates of $\mathbf{A}_{n,k_n}^T X^{(n)}$, for all $n\in\N$.
\end{remark}

We now present our main condition on the sequence $\{X^{(n)}\}_{n\in\N}$.
\begin{assumption}\label{ass-normldp}
The sequence  of scaled norms  $\{\|X^{(n)}\|_2/\sqrt{n}\}_{n\in\N}$ satisfies an LDP at speed $s_n$ with GRF $J_X:\mathbb{R}\to[0,\infty]$.   
\end{assumption}
{When a special case of Assumption~\ref{ass-normldp} holds with  speed
$s_n = n$, we  say that Assumption~\ref{ass-normldp}* holds.
\begin{customthm}{A*}
The sequence  of scaled norms  $\{\|X^{(n)}\|_2/\sqrt{n}\}_{n\in\N}$ satisfies an LDP at speed $n$ with GRF $J_X:\mathbb{R}\to[0,\infty]$.
\end{customthm}
}
\begin{remark}
\label{rem-assnormldp}
 The special case of Assumption \ref{ass-normldp}*  is important
 because,  as shown in  Section~\ref{sec-ex}, 
 it is satisfied by several important sequences of measures, including those whose elements are taken
 from a large family of  product measures (see Proposition~\ref{prop-prod}), the 
uniform measure on an  $\ell_p^n$ ball of radius $n^{1/p}$, with $p\geq 2$ (see Proposition \ref{lem-lpeg})
or, in fact, a more general class of measures  that includes the uniform measure on an Orlicz ball
defined via a superquadratic function (see Proposition~\ref{lem-orlex}), and a general class of
Gibbs measures with superquadratic potential and interaction  
functions (see Proposition~\ref{lem-gibbs}). 
However, Assumption~\ref{ass-normldp}* no longer holds  when $X^{(n)}$ is uniformly
chosen from an $\ell_p^n$ ball of radius $n^{1/p}$ with $p \in [1,2)$.
  Indeed, this can be deduced from the sharp large  deviation upper bounds obtained in
  \cite{Pao06} and \cite{GueMil11}.  In addition,    
    Theorem 1.3 of~\cite{Kabluchko17} shows that $\{\|X^{(n)}\|_2/\sqrt{n}\}_{n \in \N}$
    satisfies an LDP with speed  $s_n = n^{p/2}$.  These results  
    motivate the more general condition stated in Assumption \ref{ass-normldp}.
  \end{remark}

 \begin{remark}
   \label{rem-KLS}
   The general form of  Assumption~\ref{ass-normldp} is also of interest because
    of its  connection to the Kannan-Lov\'{a}sz-Simonovits (KLS) conjecture formulated in
    \cite{KLSconj95} (also see \cite{AloGutBas15} for a nice exposition), which is one of the
    major open problems in convex geometry.  Indeed, 
    Theorem A in~\cite{AloGutProTha20} states that the KLS conjecture is false if 
    there exists a sequence
    of isotropic and log-concave random vectors 
    $\{X^{(n)}\}_{n\in\N}$ satisfying Assumption~\ref{ass-normldp} either with $s_n\ll \sqrt{n}$ and nontrivial
    $J_X$ or with $s_n=\sqrt{n}$ and  $\inf_{t>t_0}\{\inf_{x>t}J_X(x)/t\}=0$ for some  $t_0\in(1,\infty)$.
    It is shown in \cite{AloGutProTha20} that when $X^{(n)}$ is uniformly
    distributed on the  $\ell_1^n$ ball of radius $n$, then  $\{X^{(n)}\}$ satisfies
    Assumption~\ref{ass-normldp} with $s_n=\sqrt{n}$,
    but the condition on the rate function $J_X$ is not satisfied (which is consistent with the fact that 
    the KLS conjecture is widely believed to be true).
    In view of this observation, it would be interesting to extend our
    verification of Assumption~\ref{ass-normldp}  to more general measures
      such as Orlicz balls   (respectively, Gibbs measures) with superlinear functions (respectively, potentials),
      extending the corresponding results obtained in this article   for superquadratic functions (respectively, potentials). 
        \end{remark} 
    
When  Assumption~\ref{ass-normldp} 
  is satisfied with a GRF $J_X$ that has a unique minimum at $m\in\R_+$,
go  then for all sufficiently large $n$,  the random variable $X^{(n)}$ satisfies the thin shell condition
  of~\cite[Equation (1)]{AntBalPer03}. To be more precise, note that the observation that the minimum of 
    $J_X$, which  is equal to zero, is achieved  
    uniquely at $m$,  {together with the fact that $s_\ell \ra \infty$, implies that for every $c > 0$, there exists} $\{\delta_\ell\}_{\ell \in \N}$ with 
  $\delta_\ell \downarrow 0$ such that \begin{equation*}\sqrt{s}_\ell \inf\{J_X(x):|x-m| \geq \delta_\ell, x \in \R_+\} \geq c.\end{equation*}
  Setting $\varepsilon_\ell := \max( \delta_\ell, 2 e^{-c \sqrt{s_\ell}})$, 
  it then follows from 
  Assumption~\ref{ass-normldp} and the definition of the LDP (see Definition \ref{def-ldp})
  that  $\varepsilon_\ell \downarrow 0$ as $\ell \rightarrow \infty$, and for every $\ell \in \N$, there exists $N_\ell \in\N$ such that
  \begin{equation}\label{thin-shell-eq} \mathbb{P}\left(\abs{\frac{\norm{X^{(n)}}_2}{\sqrt{n}}-m}\geq \varepsilon_\ell \right )
  \leq \varepsilon_\ell, \qquad \mbox{ for all } n \geq N_\ell.   \end{equation}
  In particular, this implies
   the following weak limit: 
\begin{equation*} \frac{ \|X^{(n)}\|_2 }{ \sqrt{n} } \xrightarrow{\P} m \in \R_+. 
\end{equation*}  
  Thus, we refer to the strengthening of Assumption~\ref{ass-normldp} with $J_X$ having
  a unique minimum as the \emph{asymptotic thin shell condition}.
   Note that the thin shell condition is
usually stated for isotropic random vectors $X^{(n)}$ and with $m = 1$, but since 
we do not restrict our consideration to only isostropic random vectors $X^{(n)}$, we phrased the condition above for arbitrary $m > 0$.
Just as the thin shell condition yields a central limit theorem  in the sense that  
the  projections of $X^{(n)}$ can be shown to be close to Gaussian (see, e.g., \cite{AntBalPer03,Kla07}),  
our results, summarized in the next three sections,  show that
the  asymptotic thin shell condition
implies that the   empirical measures $\{L^n\}_{n \in \N}$
of the coordinates of the projections of $X^{(n)}$ 
satisfy an LDP  in the sublinear and linear regimes (with the weaker Assumption~\ref{ass-normldp} sufficing in the
latter regime).
 
\begin{remark}
  \label{rem-Keith}
  There exist sequences $\{X^{(n)}\}_{n \in \N}$ for which the corresponding random projections
satisfy an LDP, but  the thin shell condition fails to hold.   For example, this can happen when
Assumption \ref{ass-normldp} holds with a rate function $J_X$ that has multiple minima, as can
happen for certain Gibbs measures (see Section \ref{sec-gibbs}) with non-convex potentials $F$ and $G$. 
As another example, let $X^{(n)}$ be distributed according to a mixture of two Gaussian distributions in $\R^n$ both with mean $0$ and covariance matrices $I_n$ and $2I_n$, respectively. {In other words, the density of $X^{(n)}$ takes  the form   $f_{X^{(n)}}=\frac{1}{2}\phi_{I^n}+\frac{1}{2}\phi_{2I^n}$ for a positive definite matrix $C$, where  $\phi_C$ denotes the Gaussian density with mean $0$ and covariance matrix $C$.} Then $\{X^{(n)}\}_{n \in \N}$ does not satisfy the thin shell condition since half its mass is concentrated around the thin shell of radius $\sqrt{n}$ and the other half around the thin shell of radius $\sqrt{2n}$. However, it is easy to show that the sequence $\{X^{(n)}\}_{n\in\N}$ satisfies Assumption~\ref{ass-normldp}* (e.g., since the sequence of norms of each of the Gaussian distributions in the mixture satisfies an LDP trivially, the LDP for the sequence of norms of the mixtures can be deduced from results in \cite{DinZab92}).  
\end{remark}

\subsection{Results in the constant regime}\label{sec-const}

To establish LDPs in the constant regime, when 
$\{k_n\}_{n\in\N}$ is constant at $k$ for some $k\in\N$, 
 we will require either Assumption~\ref{ass-normldp}* or, 
 to cover more general sequences $\{X^{(n)}\}_{n \in \N}$ like the uniform measure
on an $\ell_p^n$ ball with $p \in [1,2)$,  
the following  modification of Assumption \ref{ass-normldp}*.

\begin{assumption}  \label{ass-rescaled}
 There exists a positive sequence $\{s_n\}_{n\in\mathbb{N}}$ with $\lim_{n \rightarrow \infty} s_n  = \lim_{n \rightarrow \infty} n/s_n = \infty$ such that 
the sequence of scaled norms $\{\sqrt{s_n}\|X^{(n)}\|_2/n\}_{n\in\mathbb{N}}$ satisfies an LDP in $\mathbb{R}$
at speed $s_n$ with GRF $J_X:\mathbb{R}\to[0,\infty]$.
\end{assumption}

  \begin{remark}\label{rk-rescaled}
    It is easy to see, by a simple rescaling argument,  that Assumption~\ref{ass-normldp}* is
    equivalent to a modified version of 
    Assumption~\ref{ass-rescaled}, in which one requires that $\{\sqrt{s_n}\|X^{(n)}\|_2/n)\}_{n\in\mathbb{N}}$ satisfies an LDP at  a speed $s_n$ that satisfies 
    $s_n/n \rightarrow r$, with $r \in (0,\infty)$ rather than  $r = 0$. 
    Indeed,  this modified version of Assumption~\ref{ass-rescaled} would hold
with GRF $J_X^{(r)}$  if and only if 
  Assumption~\ref{ass-normldp}* is satisfied with  GRF 
  $J_X (x) := rJ_X^{(r)} (\sqrt{r}x)$.
\end{remark}

 \begin{theorem}[constant, $k_n \equiv k$] \label{th-constant} Suppose $\{k_n\}_{n\in\N}$ is constant at
   $k  \in \N$, and that either Assumption \ref{ass-normldp}* or Assumption \ref{ass-rescaled} holds,
   with  sequence $\{s_n\}_{n\in\N}$ and GRF $J_X$. 
   Then  $\{n^{-1/2}\mathbf{A}_{n,k}^T X^{(n)}\}_{n\in\N}$ satisfies an LDP in $\R^k$ at speed $s_n$, 
   with GRF $I_{\mathbf{A}X,k}:\R^k\ra[0,\infty]$ defined by
   \begin{equation}
     \label{Jan-constant1}
    I_{\mathbf{A}X,k}(x) := 
    \begin{cases}
    \inf_{0 < c < 1} \left\{  J_X\left(\frac{\|x\|_2}{c}\right) -\tfrac{1}{2}
    \log\left(1 - c^2\right)\right\}, \quad &\text{if Assumption~\ref{ass-normldp}* holds},\\
    \inf_{c>0} \left\{J_X\left(\frac{\|x\|_2}{c} \right)+\frac{c^2}{2}\right\}, \quad &\text{if Assumption~\ref{ass-rescaled} holds}.
    \end{cases}
   \end{equation}
 \end{theorem}
 
{
  The proof of Theorem~\ref{th-constant} is given in Section~\ref{sec-constant}, where the role of Assumption \ref{ass-rescaled} in the proof is discussed in detail.
  The rate function takes the form in~\eqref{Jan-constant1} because  rare events for  $n^{-1/2}\mathbf{A}_{n,k}^T X^{(n)}$ occur through a combination of the ``radial" component represented by $J_X$, and the 
  ``angular" component represented by $c\mapsto -\frac{1}{2} \log(1 - c^2)$ (when Assumption~\ref{ass-normldp}* holds) or $c \mapsto c^2/2$ (when Assumption~\ref{ass-rescaled} holds). }

 As an immediate corollary of Theorem \ref{th-constant}, we have the following LDP for the corresponding {scaled $\ell_q^k$} {norms (or in the case $q \in (0, 1)$, quasi-norms) of random projections}.
For  any $q \in (0, \infty)$ and $n \in \N$,
  define { 
  \begin{equation}
    \label{normY2}
	Y_{q,k}^n := \|\mathbf{A}_{n,k}^T X^{(n)}\|_q.
  \end{equation}}

\begin{corollary}[constant, $k_n \equiv k$] \label{cor-normcon}
  Suppose $\{k_n\}_{n \in \N}$,  $\{X^{(n)}\}_{n\in\N}$, 
$\{s_n\}_{n \in \N}$ and $I_{\mathbf{A}X,k}$ 
  are as in Theorem~\ref{th-constant}. 
Then  $\{n^{-1/2}Y_{q,k}^n\}_{n\in \N}$ satisfies an LDP at speed $s_n$ with GRF 
\begin{equation} \label{normgrf-fixed}
 \mathbb{J}_q^{\mathsf{con}} (x) := \inf_{z \in \mathbb{R}^k}  \left\{I_{\mathbf{A}X,k}(z)  : \, \|z\|_q = x\right\}, \quad x\in \R_+.
\end{equation} 
\end{corollary}
\begin{proof}
 This 
 is an immediate consequence of the  LDP for $\{n^{-1/2}\mathbf{A}_{n,k}^T X^{(n)}\}_{n\in\N}$ in Theorem \ref{th-constant} and the contraction principle (Lemma~\ref{lem-con}) applied to the continuous mapping $\R^k \ni x \mapsto \|x\|_q \in \R$.
 \end{proof}

\subsection{Results in the sublinear regime}\label{sec-sublinear}
Recall that if, instead of being constant, the sequence  $k_n$ tends to infinity  as $n\to\infty$,
then our goal is to establish an
LDP for the sequence of empirical measures $\{\ln\}_{n\in\N}$ of \eqref{eq-empproj}. We start in this section  
by analyzing the sublinear regime.  Section \ref{subsub-ldsublin} summarizes our LDP results for the sequences of empirical measures $\{L^n\}_{n\in\N}$ and Euclidean norms of the randomly projected vectors.
 Section \ref{subsub-ldpEunorm} contains  additional results on LDPs for a different
scaling of  Euclidean norms as well as more general $q$-norms, with $q \in [1,2)$, 
  of projected vectors that is suitable for the sublinear regime.

\subsubsection{LDPs for the empirical measures and norms of projected vectors}
\label{subsub-ldsublin}

In what follows, we will write $\gamma_\sigma$ to denote the Gaussian measure on $\R$ with mean 0 and variance $\sigma^2$; that is, for $\sigma > 0$, let 
\begin{equation}\label{gam}
\cP(\R) \ni  \gamma_\sigma \sim \mathcal{N}(0,\sigma^2).
\end{equation}

\begin{theorem}[sublinear, $1 \ll k_n \ll n$] \label{th-on}
  Suppose $\{k_n\}_{n\in\N}$ grows sublinearly and Assumption \ref{ass-normldp} holds
  with associated speed $s_n$ and GRF $J_X$. Also, suppose that $J_X$ has a unique minimum at $m > 0$.  Let $H$ be the relative entropy functional defined in \eqref{relent}. 
  Then, for every $q\in[1,2)$, 
   \begin{enumerate}
 \item If $s_n\gg k_n$,
  $\{\ln\}_{n\in\N}$ satisfies an LDP in $\cP_q(\R)$ at speed $k_n$, with GRF $\mathbb{I}_{L,k_n}:\cP_q(\R)\ra[0,\infty]$, defined by
\begin{equation*} \mathbb{I}_{L,k_n}(\mu) := H(\mu | \gamma_{m}).
\end{equation*}
 \item If $s_n = k_n$,
  $\{\ln\}_{n\in\N}$ satisfies an LDP in $\cP_q(\R)$ at speed $k_n$, with GRF $\mathbb{I}_{L,k_n}:\cP_q(\R)\ra[0,\infty]$, defined by
\begin{equation*} \mathbb{I}_{L,k_n}(\mu) := \inf_{c>0}\left\{H(\mu | \gamma_c)+J_X(c)\right\}.
\end{equation*}
 \item If $s_n\ll k_n$,
  $\{\ln\}_{n\in\N}$ satisfies an LDP in $\cP_q(\R)$ at speed $s_n$, with GRF $\mathbb{I}_{L,k_n}:\cP_q(\R)\ra[0,\infty]$, defined by
\begin{equation*} \mathbb{I}_{L,k_n}(\mu) := 
\begin{cases}
J_X(c),&\quad \mu=\gamma_c,\\
+\infty, &\quad  \text{otherwise}.
\end{cases} 
\end{equation*}
\end{enumerate}
\end{theorem}

{As in Theorem~\ref{th-constant}, the rate functions above can again be decomposed into a
  radial component, represented by $J_X$ (as a consequence of Assumption \ref{ass-normldp}), and
  ``an angular'' component  $\mathbf{A}_{n,k_n}$, which is captured by the relative entropy term.
  Depending on the relative rate of growth of $k_n$ and $s_n$, different parts dominate the rate function,
  with both terms being present only when $s_n = k_n$.}

  {Next, as in the constant regime (see Corollary~\ref{cor-normcon}), we also establish LDPs for the sequence of  scaled  Euclidean norms of the random projections.} 
 To deal with cases when Assumption~\ref{ass-normldp}* is not satisfied, it is
   not sufficient to consider Assumption~\ref{ass-rescaled} as in the constant regime. 
   Instead, we will need to introduce the following refinement of Assumption~\ref{ass-normldp}.

\begin{assumption}\label{ass-sublinear} 
 There exist  $r \in [0,\infty]$, a GRF $J_X^{(r)}:\R \to [0,\infty]$,  and  a positive sequence $\{s_n\}_{n\in\N}$ satisfying 
  $s_n \to \infty$,  $s_n/n \to 0$  and $s_n/k_n \to r$ as $n \to \infty$, such that 
  \begin{enumerate}
\item
  if $r\in[0,\infty)$, then $\{\sqrt{k_n}||X^{(n)}||_2/n\}_{n\in\N}$ satisfies an LDP at speed $s_n$ 
    with GRF $J_X^{(r)}$; 
  \item
     if $r=\infty$, then $\{\sqrt{s_n}||X^{(n)}||_2/n\}_{n\in\N}$ satisfies an LDP at speed $s_n$ 
    with GRF $J_X^{(\infty)}$.
  \end{enumerate}
  \end{assumption}

The following simple observation is analagous to the one made in
Remark \ref{rk-rescaled}.

\begin{remark}\label{req1c}
It is easy to see that 
Assumption~\ref{ass-sublinear} holds with $r\in(0,\infty)$, $\{s_n\}_{n\in\N}$ and GRF $J_X^{(r)}$
if and only if it also holds with $r'=1$,  $\{s_n':=k_n\}_{n\in\N}$, and GRF $J_X^{(1)} (x):= rJ_X^{(r)}(\sqrt{r}x)$,
$x \in [0,\infty)$.  Therefore, in essence, one need only consider the cases $r \in \{0,1,\infty\}$  in  Assumption~\ref{ass-sublinear}.  
\end{remark}

\begin{theorem}\label{sub-ldp}
  Suppose $k_n$ grows sublinearly, and recall the definition of $Y^n_{q,k}$ given in \eqref{normY2}. 
   \begin{enumerate}
   \item If Assumption~\ref{ass-normldp}* holds with GRF $J_X$, then $ \{n^{-1/2}Y^n_{2,k_n}\}_{n\in\N}$
     satisfies an LDP in $\R$ at speed $n$ with GRF $ \mathbb{J}_2^{\mathsf{sub}}:\R\ra[0,\infty]$, defined by
   \[
 \mathbb{J}_2^{\mathsf{sub}}(x):= \inf_{{c\in(0,1)}} \left\{-\frac{1}{2}\log (1-c^2)+J_X\left(\frac{x}{c}\right) \right\}. 
   \]
   \item If Assumption~\ref{ass-sublinear} holds with  $r\in[0,\infty]$, $\{s_n\}_{n\in\N}$ and GRF $J_X^{(r)}$, then
     $ \{n^{-1/2}Y^n_{2,k_n}\}_{n\in\N}$ satisfies an LDP in $\R$, at speed $s_n$ when $r \in \{0,\infty\}$
     and at speed $k_n$ when $r\in(0,\infty)$, with GRF $\mathbb{J}_2^{\mathsf{sub}}:\R\ra[0,\infty]$, where
    \[
 \mathbb{J}_2^{\mathsf{sub}}(x):=
\begin{cases}
J_X^{(0)}(x),\quad& \text{if } r=0,\\
\inf_{c>0}\left\{\frac{c^2-1}{2}-\log c+rJ_X^{(r)}\left(\frac{\sqrt{r}x}{c}\right) \right\},\quad &\text{if } r\in(0,\infty),\\
\inf_{c>0}\left\{\frac{c^2}{2}+ J_X^{(\infty)}\left(\frac{x}{c}\right)\right\},\quad &\text{if } r = \infty.
\end{cases}
\]
  \end{enumerate}
\end{theorem}

Note that there is a transition in the form of the LDP
  depending on the relative growth rates
  of $\{s_n\}_{n\in\N}$ and $\{k_n\}_{n\in\N}$. 
  The implications of these results for the special case
  when $X^{(n)}$ is the uniform measure on an $\ell_p^n$ ball,
  and their relation to the  work of~\cite{AloGutProTha18}, is
  discussed in Section \ref{sec-lp-s}; see Theorem \ref{ple2} and
  Remark \ref{rem-alogut}.

  \subsubsection{LDP for an  alternative scaling of  $q$-norms of projections}
 \label{subsub-ldpEunorm}

 In this section,  we show that we can also establish LDPs for  a different scaling
 of the norm, and in this case we consider not just the Euclidean norm, but
 $q$-norms  for $q \in [1,2]$. 
 {More precisely, we consider the sequence 
 $
 \{k_n^{-1/q}Y^n_{q,k_n}\}_{n\in\N}.
 $}
For $q\in[1,2)$ and $t\in\R$ or $q=2$ and $t<1/2$, define
\begin{equation}\label{lambda_q}
\Lambda_q(t) := \log \int_\R \frac{1}{\sqrt{2\pi}}\exp\left(t\abs{x}^q-\frac{1}{2}x^2\right)dx,\quad t\in\mathbb{R},
\end{equation}
and let $\Lambda_q^*$ be the Legendre transform of $\Lambda_q$. Moreover, set  $\mathcal{M}_q$ to be the $q$-th absolute moment of a standard  Gaussian random variable, 
\begin{equation}\label{q-abs-moment}
\mathcal{M}_q := \int_\R \frac{1}{\sqrt{2\pi}}\abs{x}^q\exp(-x^2/2)dx = \frac{2^{q/2}}{\sqrt{\pi}}\Gamma\left(\frac{q+1}{2} \right).
\end{equation}

\begin{theorem} \label{sub-mdp}
  Fix $q\in[1,2]$, suppose $1\ll k_n\ll n$ and Assumption~\ref{ass-normldp} holds with speed $s_n$
  and GRF $J_X$,  which additionally  has a unique minimum at  $m>0$.  Also,
  recall the definition of $Y^n_{q,k}$ given in \eqref{normY2}.
  Then $\{k_n^{-1/q}Y^n_{q,k_n}\}_{n\in\mathbb{N}}$ satisfies an LDP at speed $s_n\wedge k_n$
    with GRF $\widehat{\mathbb{J}}_q^{\mathsf{sub}}:\R\to[0,\infty]$ defined by 
 \begin{equation}
\widehat{\mathbb{J}}_q^{\mathsf{sub}}(x) := 
    \begin{cases}
     \Lambda^*_q\left(x^q/m^q \right), \quad &\text{if $s_n\gg k_n$},\\
    \inf_{c>0}\left\{\Lambda^*_q(c^q) + J_X\left(x/c\right)\right\}, \quad &\text{if $s_n=k_n$},\\
     J_X(x/\mathcal{M}_q^{1/q}), \quad &\text{if $s_n\ll k_n$},\\
    \end{cases}
   \end{equation} 
   {for $x \geq 0$, and $\widehat{\mathbb{J}}_q^{\mathsf{sub}}(x) := +\infty$ for $x < 0$}.
\end{theorem}

The proof of Theorem~\ref{sub-mdp} is given in Section~\ref{sec-sub-norm}. 
When $q \in [1,2)$, the LDP for $\{k_n^{-1/q}Y^n_{q,k_n}\}_{n\in\mathbb{N}}$ is an immediate consequence of Theorem~\ref{th-on} and the contraction principle. For $q=2$, the contraction principle no longer applies since the moment map $M_2(\cdot)$ is not continuous in $\cP_q(\R)$ for $q<2$. We take a different approach
  to the proof for all cases $q\in[1,2]$, by 
  looking directly at the norm,  instead of using the LDP of empirical measures. 
  In the special case  when $q=2$ and  $X^{(n)}$ is uniformly distributed on the scaled $\ell^n_p$ ball of
  radius $n^{1/p}$, with $p\in[2,\infty]$ (or admits a slightly more general
  representation),
  the  LDP for $\{\widetilde{Y}_{2,k_n}^n\}_{n\in\mathbb{N}}$ was also obtained in~\cite[Theorem B]{AloGutProTha20}.

 \subsection{Results in the linear regime}\label{sec-linear}
 Recall the definition of the second moment map $M_2(\cdot)$ from \eqref{secmom}. 

\begin{definition}
For $\lambda \in (0,1]$, define $\mathcal{H}_\lambda:\mathcal{P}(\R) \ra [0,\infty]$ as
  \begin{equation}\label{eq-hrf} \mathcal{H}_\lambda(\nu) := 
\begin{cases}
-\lambda \, h(\nu) + \tfrac{\lambda}{2} \log (2\pi e) + \tfrac{1-\lambda}{2}\log\left(\tfrac{1-\lambda}{1-\lambda\,M_2(\nu)} \right),\quad &\text{if $M_2(\nu) \leq 1/\lambda$},\\
+\infty,\quad &\text{otherwise,}
\end{cases}
  \end{equation}
  where we adopt the convention that  $0 \log 0 = 0$ and hence, $0 \log (0/0) = 0$. 
\end{definition}

{ Note that if $\lambda M_2 (\nu) = 1$,
then  $\mathcal{H}_\lambda(\nu) = \infty$ when $\lambda < 1$, but (due to our convention) 
$\mathcal{H}_\lambda(\nu) = - h(\nu) + \frac{1}{2} \log (2 \pi e)$  when $\lambda = 1$. }

{
\begin{remark}\label{rm-min-H}
  Since $h$ is strictly concave and $M_2$ is linear, from the definition in \eqref{eq-hrf}, $\mathcal{H}_\lambda$ is strictly convex. A direct verification shows that $\mathcal{H}_\lambda(\gamma_1)=0$, and 
 the strict convexity of  $\mathcal{H}_\lambda$ shows that $\gamma_1$ is the unique minimizer of $\mathcal{H}_\lambda$. 
\end{remark}
}

\begin{theorem}[linear, $k_n\sim \lambda n$]\label{th-simn}
Fix {$q\in[1,2)$}. Suppose $\{k_n\}_{n\in\N}$ grows linearly with rate $\lambda \in (0,1]$ and Assumption \ref{ass-normldp}  holds with sequence $\{s_n\}_{n\in\N}$ and GRF $J_X$. Then $\{\ln\}_{n\in\N}$ satisfies an LDP in $\cP_q(\R)$ at speed $s_n$ with 	GRF $\mathbb{I}_{L,\lambda}: \cP_q(\R)\ra[0,\infty]$, where
\begin{enumerate}
\item If $s_n = n$, then for $\mu\in\cP_q(\R)$, 
 \begin{align}
\mathbb{I}_{L,\lambda}(\mu)=
\label{Jan-linear} 
\begin{aligned}
& {\inf_{ c> \sqrt{\lambda M_2(\mu)}}}
  \left\{J_X(c) - \tfrac{1-\lambda}{2} \log\left( 1 - \tfrac{\lambda M_2(\mu)}{c^2} \right) + \lambda \log(c)  \right\}  \\
  &\quad-\lambda h(\mu) + \tfrac{\lambda}{2} \log(2\pi e) + \tfrac{1-\lambda}{2} \log(1- \lambda),
\end{aligned}
 \end{align}
{where we use the convention that $0\log 0=0$.}  
\item If $s_n \ll n$, then for $\mu\in\cP_q(\R)$
\begin{align*} 
\mathbb{I}_{L,\lambda}(\mu) &:= \begin{cases}
J_X(c),&\quad \mu=\gamma_c\\
+\infty, &\quad  \text{otherwise}.
\end{cases}
\end{align*}
\end{enumerate}
\end{theorem}

The LDP of the sequence of $\ell_q$ norms of the randomly projected vectors is given in the following theorem.

\begin{theorem}\label{th-normlinear}
  Suppose $\{X^{(n)}\}_{n\in\N}$ satisfies Assumption~\ref{ass-normldp}, and
  $k_n \sim \lambda n$ for some $\lambda \in (0,1]$. 
{Then,  for $q \in [1,2]$
 the sequence $\{n^{-1/q}Y^n_{q,k_n}\}_{n\in \N}$  defined
in \eqref{normY2}} satisfies an LDP  at speed $s_n$
 with GRF $\mathbb{J}_{q,\lambda}^{\mathsf{lin}}$, where for $x\in\R_+$,
  \begin{equation}\label{normgrf1}
\mathbb{J}_{q,\lambda}^{\mathsf{lin}}(x):=
\begin{cases}
\inf_{\nu\in\cP(\R), c \in \R_+ }
       \left\{ \mathcal{H}_\lambda(\nu) +
  J_X\left(c \right): \lambda M_q(\nu) = {\left(\frac{x}{c}\right)^q} \right\},\quad &\text{if } s_n=n,\\
 J_X\left(\frac{x}{(\lambda \mathcal{M}_q)^{1/q}}  \right),\quad &\text{if } s_n\ll n,
\end{cases}
\end{equation}
with $M_q$ the $q$-th moment map as in \eqref{secmom} and $\mathcal{M}_q$ the $q$-th absolute moment of a standard Gaussian random variable defined in~\eqref{q-abs-moment}.
\end{theorem}

The proof of Theorem \ref{th-normlinear} is deferred to Section \ref{subs-2norm}. 
It is shown there that when $q\in[1,2)$, the result  is an immediate
consequence of  Theorem \ref{th-simn} and
an application of the contraction principle. 
However, even though the rate function still takes an analogous
form when $q = 2$,  the contraction principle is no longer applicable in {that setting because the LDP of Theorem \ref{th-simn} only holds in the $q$-Wasserstein topology for $q\in [1,2)$.}  
 Despite this apparent gap, using  a different argument in Section \ref{subs-2norm} we show that
 the result nevertheless holds. In the process, we provide an alternative representation for the rate function \eqref{normgrf1}
 for all $q \in [1,2]$  (see Proposition \ref{prop-norm}).    This is a manifestation of a somewhat
 nuanced technical issue, which is elaborated upon in Remark \ref{rmk-technical}.

\section{Examples satisfying the main assumptions}\label{sec-ex}

In this section, we present several examples of sequences of random vectors $\{X^{(n)}\}_{n\in\N}$ that satisfy the  {assumptions introduced  in} Section~\ref{sec-mainres}.

\subsection{Product measures}\label{sec-prod}

\begin{lemma}[i.i.d.\ case]\label{prop-prod}
Let $X_1,X_2,\dots$ be a sequence of i.i.d.\ real-valued random variables, and let $X^{(n)} := (X_1,\dots,X_n)$. Suppose that we have \begin{equation*}
 \Lambda(t) := \log \E[e^{t X_1^2}] < \infty 
\end{equation*} for $t$ in some open ball of non-zero radius about 0, and let
 $\Lambda^*$ be the Legendre transform of $\Lambda$ defined in~\eqref{legendre}.  Then, $\{X^{(n)}\}_{n\in\N}$ satisfies Assumption \ref{ass-normldp}* with $J_X(x) := \Lambda^*(x^2)$ for $x\in\R_+$ {and $J_X(x) := +\infty$ otherwise.  Moreover,  $m = \sqrt{\E[|X_1|^2]}$ is the unique minimizer of $J_X$.} 
\end{lemma}

\begin{proof}
By Cram\'er's theorem for sums of i.i.d.\ random variables \cite[Theorem 2.2.1]{DemZeiBook}, the sequence $\{\|X^{(n)}\|_2^2 / n\}_{n\in\N}$, satisfies an LDP at speed $n$ with GRF $\Lambda^*$.  By the contraction principle applied to the square root function, Assumption~\ref{ass-normldp}* holds. As for the unique minimizer, this follows from the law of large numbers, with limit $\frac{1}{n}\sum_{i=1}^n X_i^2 \xrightarrow{\text{a.s.}} \E[X_1^2] = m^2$.
\end{proof}

\subsection{Scaled $\ell_p^n$ balls}\label{sec-lp}

In this section we consider  random vectors uniformly distributed on  scaled $\ell^n_p$ balls.
More precisely, for $n \in \N$ and $p > 0$,
define $X^{(n,p)}$ to be a random vector uniformly distributed on $n^{1/p}\mathbb{B}_{p}^n$, where 
 $\mathbb{B}_{p}^n$ denotes the unit  $\ell_p^n$-ball: 
\[  \mathbb{B}_{p}^n := \left\{ x \in \R^n:  \sum_{k=1}^n |x_k|^p \leq 1 \right\}. \]
We introduce some preliminaries in Section \ref{subs-prelim}, then verify
our main assumptions for $\{X^{(n,p)}\}_{n \in \N}$ in the  case when
$p \in [1,2)$ in Section \ref{sec-lp-s} and the easier case when $p > 2$ in Section \ref{sec-lp-b}. 
  When combined with the results of Section \ref{sec-mainres}, these yield 
  LDPs for random projections of these $\ell_p^n$ balls.

\subsubsection{Preliminaries}
\label{subs-prelim}

For $p \in [1,\infty)$, let   $f_p$ be the density of the $p$-generalized normal distribution:
  \[  f_p(x) := \frac{1}{2p^{1/p} \Gamma(1+1/p)} e^{-|x|^p/p}, \quad x \in \R,   \] 
  where $\Gamma$ denotes the Gamma function, and let $\xi^{(p)}$ denote a $p$-generalized normal random variable, namely one with density $f_p$.
  We provide here a useful tail estimate:    for $t>0$, 
\begin{equation}\label{est-Y2}
  2\exp\left(-\frac{1}{p}t^{p/2}-\frac{p-1}{2}\log t \right)\geq \mathbb{P}(|\xi^{(p)}|^2>t)\geq\frac{t^{p/2}}{t^{p/2}+1}\exp\left(-\frac{1}{p}t^{p/2}-\frac{p-1}{2}\log t \right).
  \end{equation}
 This estimate was proved in~\cite[Lemma 5.3]{AloGutProTha18} for $p\in[1,2),$ but can easily be extended to include $p=2$.  Now,  
let  
\begin{equation}\label{fpstar}
F_p^*(y) := \sup_{t_1, t_2 \in \mathbb{R}} \left[ t_1 y + t_2 - \log \left(\int_{\mathbb{R}} e^{t_1 x^2 + t_2|x|^p} f_p(x) dx\right)\right],\quad y\in\R, 
\end{equation}
and
\begin{equation}\label{mp}
 m(p) :=  \left( \frac{p^{2/p} \Gamma \left(1 + \frac{3}{p} \right)}{3 \Gamma \left(1+\frac{1}{p}\right)} \right)^{1/2}. 
\end{equation}

Also, let $U$ be a uniform random variable on $[0,1]$ and $\{\xi^{(p)}_i\}_{i\in\N}$ be a sequence of i.i.d. random variables with density $f_p$ independent of $U$. For $n\in\N$ and $p\in[1,\infty)$, denote $\xi^{(n,p)}=(\xi^{(p)}_1,\ldots,\xi^{(p)}_n)$. In this section, we take advantage of the following useful representation of $X^{(n,p)}$ obtained in~\cite[Lemma 1]{SchZin90}:
\begin{equation}\label{lp-repres}
X^{(n,p)}\buildrel (d) \over = n^{1/p} U^{1/n}\frac{\xi^{(n,p)}}{\norm{\xi^{(n,p)}}_p}.
\end{equation}

In view of the representation~\eqref{lp-repres}, we first establish a property of $p$-generalized normal random variables which we strengthen from the result in~\cite{Nagaev69}.
\begin{lemma}\label{LDP-Y2}
Let $p\in[1,2]$,  let $\{k_n\}_{n\in\mathbb{N}}$ satisfy $k_n\to\infty$ as $n\to\infty$. 
Then, given  any $\{b_n\}_{n\in\mathbb{N}}$  such that $k_n/b_n\to 0$ as $n\to\infty$,
the sequence $\{b_n^{-1}\sum_{i=1}^{k_n}(\xi_i^{(p)})^2\}_{n\in\mathbb{N}}$ satisfies an LDP at speed $b_n^{p/2}$ and GRF $\mathcal{J}_{\xi,p}:\mathbb{R}\to[0,\infty]$ defined by
\[
\mathcal{J}_{\xi,p}(t) :=
\begin{cases}
\frac{t^{p/2}}{p},\quad &t\geq 0,\\
+\infty,\quad & t<0.
\end{cases}
\]
\end{lemma}

The proof is deferred to Appendix~\ref{pf-LDP-Y2}.

\subsubsection{Verification of Assumptions~\ref{ass-rescaled} and \ref{ass-sublinear} when $p \in [1,2)$}\label{sec-lp-s}

When $p \in [1,2)$, it should be noted that  annealed LDPs associated with the random vectors
  $\{X^{(n,p)}\}_{n \in \N}$ defined in Proposition~\ref{lem-lpeg} occur at different speeds than in the case $p > 2$; see 
   Theorem 2.3 of \cite{GanKimRam18}.  In particular,  from  Theorem 1.3 of
   \cite{Kabluchko17} it follows that in this case the sequence of
   scaled norms $\{\|X^{(n,p)}\|_{2}/\sqrt{n}\}_{n \in \N}$ satisfies an LDP at speed
   $n^{p/2}$.
    Thus, in this case $\{X^{(n,p)}\}_{n \in \N}$ satisfies Assumption~\ref{ass-normldp} with $s_n=n^{p/2}$ and {unique minimizer} $m=m(p)$ defined in~\eqref{mp}.  In particular, Assumption~\ref{ass-normldp}* is not
   satisfied.   We  show that nevertheless, Assumption~\ref{ass-rescaled} does hold.  
   In what follows, 
   let $U, \{\xi_i^{(p)}\}_{i\in\N}$, and $\xi^{(n,p)}$ be as in the representation of $X^{(n,p)}$
   in~\eqref{lp-repres}, and note that then 
   \begin{equation} 
     \label{lp-rep2}
\frac{\sqrt{s_n}\norm{X^{(n,p)}}_2}{n}\buildrel (d) \over =\frac{U^{1/n}}{\norm{\xi^{(n,p)}}_{p}/n^{1/p}}\frac{\sqrt{s_n}\norm{\xi^{(n,p)}}_2}{n}, \quad n\in\N.
   \end{equation}

\begin{proposition}\label{lp-thin-shell}
For $p\in[1,2)$, $\{X^{(n,p)}\}_{n\in\N}$ satisfies Assumption~\ref{ass-rescaled} with speed 
$s_n:= n^{2p/(2+p)}$
and GRF
\begin{align}\label{rate-normY}
J_{X,p}(x):=
\begin{cases}
\displaystyle \frac{x^p}{p}, \quad &x\geq 0,\\
+\infty, \quad &x<0.
\end{cases}
\end{align}
\end{proposition}
\begin{proof}
  Fix $p \in [1,2)$.
    It follows from \cite[Lemma 3.3]{GanKimRam18} that $\{U^{1/n}\}_{n\in\N}$ satisfies an LDP at speed $n$,
    and from {Cram\'er's theorem and the contraction principle} that 
$ \{\|\xi^{(n,p)}\|_p/n^{1/p}\}_{n\in\N}$  also satisfies an  LDP at speed $n$.
Moreover, it is easy to see that both sequences converge almost surely to $1$. Since, in addition, $\xi^{(n,p)}$ and $U$ are independent, 
appealing again to  the contraction principle it follows that  their ratio $W_n := U^{1/n}n^{1/p}/\|\xi^{(n,p)}\|_p$ also 
satisfies an LDP at speed $n$. {Furthermore, since $W_n\to 1$ almost surely as $n\to\infty$, the LDP rate function for $\{W_n\}_{n\in\N}$ has the unique minimizer $1$.}
Given $s_n \ll n$ and the representation in~\eqref{lp-rep2}, an application of (the last assertion of)
Lemma~\ref{ldp-prod}, with $W_n$ as above, {$m = 1$, $\gamma_n = n$}, $V_n := \sqrt{s_n}\|\xi^{(n,p)}\|_2/n$,   { and $\beta_n = s_n$}  
shows that to establish  the proposition it suffices to prove that $\{V_n\}_{n \in \N}$ satisfies an LDP with speed $s_n$ and 
GRF  $J_{X,p}$.
To this end, using the relation
$\frac{\sqrt{s_n}\norm{\xi^{(n,p)}}_2}{n} = \left(\frac{s_n}{n^2}\sum_{i=1}^n( \xi_i^{(p)})^2 \right)^{1/2},$  
 the contraction principle, Lemma~\ref{LDP-Y2} {with $k_n=n$ and $b_n=n^2/s_n$ therein, and the property that} $s_n\ll n$, we can conclude that $\{\sqrt{s_n}\|\xi^{(n,p)}\|_2/n\}_{n\in\N}$ satisfies an LDP at speed $n^{p}s_n^{-p/2}$ with GRF $J_{X,p}$ given in~\eqref{rate-normY}. Finally, since $s_n=n^{2p/(2+p)}$, we obtain $n^{p}s_n^{-p/2}=s_n$. 
\end{proof}

Next, we turn to verifying Assumption~\ref{ass-sublinear}. In the following lemma, we show that for $1\leq p <2$, different conditions in Assumption~\ref{ass-sublinear} are satisfied according to the growth speed of $k_n$.

\begin{proposition}\label{sub-assumption}
  When $p\in[1,2)$, let $J_{X,p}$ be defined as in \eqref{rate-normY}. Then {$\{X^{(n,p)}\}_{n\in\N}$} satisfies Assumption~\ref{ass-sublinear} with GRF {$J_{X}^{(r)} = J_{X,p}$ for all $r \in [0,\infty]$,} 
 where $s_n$ is defined by
\[
s_n:=
\begin{cases}
n^pk_n^{-p/2}, \quad &\text{if}\quad k_n \gg n^{2p/(2+p)},\\
n^{2p/(2+p)}, \quad &\text{if}\quad k_n = n^{2p/(2+p)} \quad\text{or} \quad k_n\ll n^{2p/(2+p)}.
\end{cases}
\]
\end{proposition}

\begin{proof}
  Let $s_n$ be as defined in the proposition. 
  As  in Assumption~\ref{ass-sublinear}, let $r \in [0,\infty]$ be the limit of $s_n/k_n$ as $n \rightarrow \infty$.  
        It is easy to see that $r = 0$ if $k_n \gg n^{2p/(2+p)}$, $r= 1$ 
        if  $k_n = n^{2p/(2+p)}$,  and  $r = \infty$  when $k_n \ll n^{2p/(2+p)}$.
      Now,  set $c_n:=\sqrt{n/k_n}$ if $r\in[0,\infty)$  and $c_n:=\sqrt{n/s_n}$ if $r=\infty$ and  note that $c_n\to \infty$ as $n\to\infty$.  
        By~\eqref{lp-repres}, we have the following representation
        \[
        \frac{\norm{X^{(n,p)}}_2}{c_n\sqrt{n}}\buildrel (d) \over =\frac{U^{1/n}}{\norm{\xi^{(n,p)}}_{p}/n^{1/p}}\frac{\norm{\xi^{(n,p)}}_2}{c_n\sqrt{n}}, \quad n\in\N.
        \]
        We first claim that $\{\|\xi^{(n,p)}\|_2/(c_n\sqrt{n})\}_{n\in\N}$ satisfies an LDP at speed $s_n\ll n$. Indeed, by Lemma~\ref{LDP-Y2} {with $k_n=n$ and $b_n=c_n^2n$ there}, the fact that
        $b_n\to\infty$
        as $n\to\infty$ and the contraction principle with the
        mapping $x\mapsto\sqrt{x}$,  
$\left\{\|\xi^{(n,p)}\|_2/(c_n\sqrt{n})\right\}_{n\in\N}$
satisfies an LDP at speed $n^{p/2}c_n^{p}$ with GRF $J_{X,p}$ given in~\eqref{rate-normY}. By~\cite[Lemma 3.3]{GanKimRam18}, {Cram\'{e}r's theorem and the contraction principle}, the independence of $\xi^{(n,p)}$ and $U$, and the contraction principle, we see that $W_n := U^{1/n}n^{1/p}/\|\xi^{(n,p)}\|_p$ satisfies an LDP at speed $n$ and further, it also converges almost surely to $1$.
The lemma then follows  upon applying Lemma~\ref{ldp-prod} with $V_n:=\|\xi^{(n,p)}\|_2/(c_n\sqrt{n})$ and $W_n := U^{1/n}n^{1/p}/\|\xi^{(n,p)}\|_p$, {$m=1$}, $\gamma_n = n$
and $\beta_n := s_n\ll n$. 
\end{proof}

When combined with the results of Section \ref{subsub-ldsublin}, the last two propositions
imply the following 
  LDPs for norms of projections in {the constant and sublinear regimes}.
  Here, we   focus {on the norm ${Y}^{(n,p)}_{2,k_n}$, 
    which is defined as in~\eqref{normY2} but with $X^{(n)}$
    replaced with $X^{(n,p)}$.}
  The corresponding results for the differently
  scaled norms $\{k_n^{-1/2}Y^{(n,p)}_{2,k_n}\}_{n \in \N}$
  were obtained in Theorem B of~\cite{AloGutProTha20}.
\begin{theorem}\label{ple2}
Fix $p\in[1,2)$. 
\begin{enumerate}
\item Suppose $k_n$ is constant at $k\in\N$. Then $\{n^{-1/2}\mathbf{A}_{n,k}^T X^{(n,p)}\}_{n\in\N}$ satisfies an LDP in $\R^k$ at speed $s_n=n^{2p/(2+p)}$ with GRF $I_{\mathbf{A}X^{(p)},k}:\R^k\to[0,\infty]$ defined by
\[
I_{\mathbf{A}X^{(p)},k}(x):=\frac{p+2}{2p}\norm{x}_2^{2p/(p+2)},\quad x\in\R^k.
\]
Moreover, $\{n^{-1/2}Y_{2,k}^{(n,p)}\}_{n \in \N}$ satisfies an LDP with GRF{
\begin{equation}
  \label{jy-constant}
\mathbb{J}_{Y^{(p)}_{2,k}}(x):=
\begin{cases}
\frac{p+2}{2p}x^{2p/(p+2)},\quad & x\geq0,\\
+\infty,\quad & x<0.
\end{cases}
\end{equation}}
\item Suppose $k_n$ grows sublinearly.
\begin{enumerate}
\item If $k_n \ll n^{2p/(2+p)}$. Then $\{n^{-1/2}Y^{(n,p)}_{2,k}\}_{n\in\N}$ satisfies an LDP in $\R$ at speed $n^{2p/(2+p)}$ with GRF  $\mathbb{J}_{Y^{(p)}_{2,k_n}}:\R\ra[0,\infty]$, which is equal to $\mathbb{J}_{Y^{(p)}_{2,k}}$ defined in \eqref{jy-constant}.  
\item If $k_n = n^{2p/(2+p)}$. Then $\{n^{-1/2}Y^{(n,p)}_{2,k}\}_{n\in\N}$ satisfies an LDP in $\R$ at speed $n^{2p/(2+p)}$ with GRF  $\mathbb{J}_{Y^{(p)}_{2,k_n}}:\R\ra[0,\infty]$ defined by {
\[
\mathbb{J}_{Y^{(p)}_{2,k_n}}(x):=
\begin{cases}
\frac{p+2}{2p} \frac{x^p}{\bar{c}(x)^p} - \log (\bar{c}(x)),\quad &x\geq 0,\\
+\infty,\quad & x<0.
\end{cases}
\]}
where $\bar{c}(x) \in [1+x^{p/(p+2)}, \infty)$
    is the unique positive solution to $c^{p+2} - c^p - x^p = 0$. 
\item If $k_n \gg n^{2p/(2+p)}$. Then $\{n^{-1/2}Y^{(n,p)}_{2,k}\}_{n\in\N}$ satisfies an LDP in $\R$ at speed $n^pk_n^{-p/2}$ with GRF  $\mathbb{J}_{Y^{(p)}_{2,k_n}}:\R\ra[0,\infty]$ defined by{
\[
\mathbb{J}_{Y^{(p)}_{2,k_n}}(x):=
\begin{cases}
\frac{x^p}{p},\quad& x\geq0,\\
+\infty,\quad & x<0,
\end{cases}
\]
which is equal to $J_{X,p}$ in \eqref{rate-normY} of Proposition~\ref{lp-thin-shell}.}
\end{enumerate}
\end{enumerate}
\end{theorem}
\begin{proof}
  In  the constant regime,  Proposition~\ref{lp-thin-shell} shows that $\{X^{(n,p})\}$ satisfies
  Assumption \ref{ass-rescaled} with rate function $J_{X,p}$ taking the explicit form given in \eqref{rate-normY}.  Subsituting this into  Theorem~\ref{th-constant}, we see that 
  $\{n^{-1/2}\mathbf{A}_{n,k}^T X^{(n,p)}\}_{n\in\N}$ satisfies an LDP in $\R^k$ at speed $s_n=n^{2p/(2+p)}$ with GRF
  \begin{align}
    \label{temp-calc}
I_{\mathbf{A}X^{(p)},k}(x) &=\inf_{c>0}\left\{\frac{\norm{x}^p_2}{pc^p}+\frac{c^2}{2} \right\}=\frac{p+2}{2p}\norm{x}_2^{2p/(p+2)}, \quad  x\in\R^k. 
\end{align}
The LDP for the norm follows from the contraction principle applied to $x \mapsto \|x\|_2$.

In the sublinear regime, Proposition~\ref{sub-assumption} shows that $\{X^{(n,p})\}$ satisfies
Assumption \ref{ass-sublinear} with GRF $J_{X,p}$  (for all values of $r \in [0,\infty]$)
and at  speed $s_n = n^{2p/(2+p)}$ in
case (a), which corresponds to $r = \infty$, and in case (b), which corresponds to $r = 1$,
and speed $s_n = n^p k_n^{-p/2}$ in case (c), which corresponds to $r = 0$.  
Together with Theorem~\ref{sub-ldp}(ii), this immediately implies case (c) and case (a)
follows from the second equality in \eqref{temp-calc} above.   
On the other hand, if case (b) holds, Theorem~\ref{sub-ldp}(ii) and the expression for
$J_{X,p}$ in \eqref{rate-normY} show that for $x\geq 0$,
\begin{align*}
\mathbb{J}_{Y^{(p)}_{2,k_n}}(x)&:=
\inf_{c>0}\left\{\frac{c^2-1}{2}-\log c+\frac{x^p}{pc^p} \right\}\\
&= \frac{p+2}{2p} \frac{x^p}{\bar{c}(x)^p} - \log (\bar{c}(x)),
\end{align*}
where $\bar{c}(x)$ is the unique positive solution to $c^{p+2} - c^p - x^p = 0$, whose existence and uniqueness is guaranteed by  Descartes' rule of signs (see e.g.~\cite{AndJacSit98}). Since $x\geq 0$, the unique positive solution satisfies $\bar{c}(x)\geq 1$. Furthermore,  $\bar{c}(x)=(\bar{c}(x)^p+x^p)^{1/(p+2)}\geq(1+x^p)^{1/(p+2)}\geq 1+x^{p/(p+2)}$. 
\end{proof}

In a similar fashion, {applying Theorem~\ref{th-on} and Theorem \ref{th-simn},  
  one can also deduce LDPs for the empirical measures of the coordinates of the sequences of projections of  
   $X^{(n)} = X^{(n,p)}$ in the sublinear and linear regimes. 
   Furthermore,  in the linear regime,} 
        similar LDP results for the norms of these projections can be deduced from Theorem~\ref{th-normlinear}
        on noting that $\{X^{(n,p)}\}_{n\in\N}$ satisfies Assumption~\ref{ass-normldp}.
        But, we do not state these results since they were already obtained 
         in Theorem 1.2 of~\cite{AloGutProTha18}.

 \begin{remark}
   \label{rem-alogut}
           Theorem \ref{ple2} shows that when $p \in [1,2)$,  
            the speed of the LDP  for Euclidean norms of $k_n$-dimensional 
  projections of $\ell_p^n$ balls  exhibits an interesting phase transition depending
  on whether the ratio  $k_n/n^{2p/(2+p)}$ is asymptotically finite or infinite, and the form of the rate function {also differs depending on 
    whether  the limit is zero, strictly positive or infinite.
    Indeed, note that
$
x^p\geq x^p/\bar{c}(x)^p=\bar{c}(x)^2-1\geq x^{2p/(p+2)}.
$
Hence, a comparison of cases (b) and (c) of the last theorem above shows that} the faster the growth of the speed $k_n$, the faster the growth  {of the rate function $\mathbb{J}_{Y^{(p)}_{2,k_n}}(x)$, as $x\to\infty$.} 
    This  disproves a  conjecture in  \cite{AloGutProTha18} (see the statement after Theorem 1.2 therein),
  which  states that the LDP (that is, speed and rate function) of these Euclidean norms should
  be the same in the constant and sublinear cases.
{This behavior is in contrast to the case of $\ell_p$ balls, with $p > 2$, where the rate functions
of the (Euclidean) norms of the projections in the sublinear and constant cases coincide
(see Proposition \ref{lem-lpeg} and Remark \ref{rem-lpeg}). }
   \end{remark}

  \subsubsection{Verification of Assumption \ref{ass-normldp}*  when  $p \in [2,\infty)$}\label{sec-lp-b} 
  The verification of the asymptotic thin shell condition is much easier in the case $p\in[2,\infty)$ than when $p\in[1,2)$.
\begin{proposition}
  \label{lem-lpeg}
  For $n \in \N$ and $p\in[2,\infty)$,  
  $\{X^{(n,p)}\}_{n \in \N}$ satisfies Assumption~\ref{ass-normldp}* with $s_n=n$ and GRF
  $J_{X,p}$ where
    \begin{equation*}
    J_{X,2} (x) :=  \left\{     \begin{array}{ll}      - \log x &  \mbox{ if } x \in (0,1], \\
+\infty & \mbox{ otherwise, }
 \end{array}     \right.
    \end{equation*}
and for $p > 2$,  
\begin{equation*}
      J_{X,p} (x) :=
    \left\{     \begin{array}{ll}       \inf_{y \geq x} [ \frac{1}{2} \log \frac{y^2}{x^2} + F_p^*(y)]          & \mbox{ if } x \in (0,1), \\       +\infty & \mbox{ otherwise. }     \end{array}     \right.
    \end{equation*}
    Moreover, $J_{X,p}$ has a unique minimizer at $m:=m(p)$ defined in~\eqref{mp}. 
\end{proposition}

\begin{proof} 
  In the  case $p = 2$, by~\eqref{lp-repres}, clearly $\|X^{(n,p)}\|_{2}/\sqrt{n} \stackrel{(d)}{=} U^{1/n}$.
  Thus, Lemma 3.3 of \cite{GanKimRam18} shows that
  $\{X^{(n,p)}\}_{n \in \N}$ satisfies  Assumption~\ref{ass-normldp}*  with GRF $J_{X,2}$, and the 
  explicit form of  $J_{X,2}$ implies it  has a unique minimum at $m(2) = 1$.
  On the other hand, when $p > 2$,  {the LDP follows from 
  \cite[Theorem 1.2]{Kabluchko17} and uniqueness of the minimizer
  from \cite[Theorem  1.1 (a)]{Kabluchko17},  
   with $d =1$ and $q_i=2$  (note that the $\ell^n_p$ ball is scaled differently there but the results nevertheless follow).}
  \end{proof}

The last  result can also be deduced from Proposition~\ref{lem-orlex} on Orlicz balls in Section~\ref{sec-orl} by choosing $V(x) = \abs{x}^p$. However, we presented the self-contained proof above since the special case of $\ell_p$ balls with $p\geq 2$ is easier due to the representation~\eqref{lp-repres}.  Finally,   Proposition~\ref{lem-lpeg} and Theorem~\ref{th-constant} together yield the following corollary.

\begin{corollary}
Fix $p \in [2,\infty)$, let $k_n$ be constant at $k\in\mathbb{N}$ and let $J_{X,p}$ be defined as in Proposition~\ref{lem-lpeg}. Then  $\{n^{-1/2}\mathbf{A}_{n,k}^T X^{(n,p)}\}_{n\in\N}$ satisfies an LDP in $\R^k$ at speed $n$ with GRF $I_{\mathbf{A}X^{(p)},k}:\R^k\to[0,\infty]$ defined by
\[ 
I_{\mathbf{A}X^{(p)},k}(x):=\inf_{0 < c < 1} \left\{  J_{X,p}\left(\frac{\|x\|_2}{c}\right) -\tfrac{1}{2}
    \log\left(1 - c^2\right)\right\},\quad x\in\R^k.
\]
\end{corollary}

The particular case when $k_n\equiv 1$, for all range of $p\in[1,\infty)$ was studied in~\cite{GanKimRam18}; see Theorem 2.2 therein. 
  Our method is different in that we take advantage of the equivalent representation in Lemma~\ref{lem-toprow} and hence a different form of the rate function is obtained here. 

  \begin{remark}
    \label{rem-lpeg}
    One may combine Proposition~\ref{lem-lpeg}, Theorem~\ref{sub-ldp} and Theorem~\ref{th-normlinear} to obtain  LDPs for $\{n^{-1/2}\|\mathbf{A}_{n,k_n}^T X^{(n,p)}\|_2\}_{n\in\N}$ in the sublinear and linear regimes when
    $p \in [2,\infty)$.
   We omit these results, since they  were already obtained in Theorem 1.1 of~\cite{AloGutProTha18}. 
   LDPs for empirical measures of the coordinates of the projections of the $\ell_p^n$ balls
   with $p \in [2,\infty)$ in the sublinear and linear regimes can also be deduced on combining Proposition~\ref{lem-lpeg} with  Theorem~\ref{th-on} or Theorem~\ref{th-simn}. 
\end{remark}

  \subsection{Orlicz and generalized Orlicz balls} \label{sec-orl} 
 
  We now consider the class of
  \emph{Orlicz and generalized Orlicz balls}. 
 These form a natural class of examples because, 
like the uniform measure on any convex set, the uniform measure on a generalized Orlicz ball is logconcave,  and much like $\ell_p^n$ balls, the coordinates of a uniformly distributed random vector satisfy a certain ``subindependence" property known as ``negative association" \cite{PilWoj08},
and  (under suitable conditions) the Orlicz ball 
satisfies the KLS conjecture~\cite{KolMil18}.

  \begin{definition}\label{def-orl}
    We say $V$ is an \emph{Orlicz function} if $V:\R\to\R_+\cup \{\infty\}$ is convex and
    satisfies $V(0)=0$ and $V(x) = V(-x)$ for $x \in \R$. Let the domain of $V$ be $D_V:=\{x\in\R:V(x)<\infty\}$. We say $V$ is \emph{superquadratic} if
\begin{equation} \label{vquad}
V(x)/x^2 \ra \infty, \quad \text{ as } x \ra \infty. 
\end{equation} 
    \end{definition}

  \noindent 
  {In particular, $2$-convex Orlicz functions  are superquadratic in the sense of Definition~\ref{def-orl},
and this includes the case $V(x) = |x|^p$, with $p > 2$. 
  }
    Fix a  superquadratic Orlicz function $V$, and denote the associated symmetric Orlicz ball by 
\begin{equation}\label{orlball} 
\mathbb{B}_V^n := \left\{x\in \R^n : \sum_{i=1}^n V(x_i) \le n\right\}.  
\end{equation}

  We also discuss so-called \emph{generalized} Orlicz balls, associated with a sequence of
  superquadratic Orlicz functions
$V_i: \R \ra \R_+$, $i \in \N$. 
{These generalized Orlicz balls, {which are induced by a norm defined by functions that vary with dimension, are also} known in the literature as \emph{Musielak-Orlicz balls}~\cite{Mus83}.
} For  $n\in \N$, define the set
\begin{equation}
  \label{gen-orlball}
  \mathbb{B}_{V_1,\dots,V_n}^n := \left\{x\in \R^n : \sum_{i=1}^n V_i(x_i) \le n  \right\}.
\end{equation}
In contrast with the usual (symmetric) Orlicz ball, the  generalized Orlicz ball $B_{V_1,\dots,V_n}^n$ need 
not be invariant to permutations of the  coordinates.

We verify  Assumption~\ref{ass-normldp}* {for symmetric  and generalized Orlicz balls in Sections \ref{sec-sorl} 
  and~\ref{sec-gen-orl}, respectively.}   
As mentioned above, the special case  $\mathbb{B}_V^n = n^{1/p}\mathbb{B}_p^n$,  with $p > 2$, 
 was analyzed  in Section \ref{sec-lp-b}. 
  However, unlike   $\mathbb{B}_p^n$, whose study is facilitated by the probabilistic representation
  \eqref{lp-repres}, there is no analogous representation for  general $\mathbb{B}_V^n$, making its analysis
  significantly more complicated.  {Instead, our proof shows that one can represent $\mathbb{B}_V^n$ in terms
  of an exponential Gibbs measure;  this idea has been subsequently  further exploited in~\cite{KabPro21}. 
  Conditions like $2$-convexity (and $2$-concavity) often arise in the local theory of Banach spaces.
     In our case this condition  arises naturally because we are also considering $2$-norms in  this section.
 However, it would be  interesting to investigate 
  whether  an analogous analysis can be carried out for  
  Orlicz balls defined via   superlinear but subquadratic Orlicz functions $V$
(it is worth  noting  that we have been able to cover the  special case of $\ell_p^n$ balls with $p \in [1,2)$) 
  or whether there is a more fundamental obstruction.}
In the following, by a slight abuse of notation, we use $\abs{A}$ to denote the volume of a measurable set $A\subset \R^n$ and use $dx$ to denote the integral over Lebesgue measure in $\R$ or $\R^n$.

\subsubsection{Symmetric Orlicz balls}\label{sec-sorl}

In analogy with \eqref{secmom}, for $\nu\in \cP(\R)$,
  define $M_V:\cP(\R) \ra \R_+$ as 
\begin{equation} \label{def-mv}
M_V(\nu) := \int_{\R} V(x) \nu(dx), 
\end{equation}
for $\nu$ such that the integral in the last display is finite. 
Also, for $b > 0$, define  $\mu_{V,b} \in \cP(\R)$  as follows:  for $x\in D_V$,
where $D_V$ is the domain of $V$ as specified in Definition \ref{def-orl},
\begin{align} \label{muv} 
\mu_{V,b}(dx) &:= \frac{1}{Z_{V,b}} e^{-b\,V(x)} dx,
\end{align}
with $Z_{V,b}$ being the normalizing constant $Z_{V,b}:=\int_{D_V} e^{-bV(x)}dx.$  Moreover, 
 define 
\begin{align}
\mathcal{J}(u,v)
&:=\sup_{s<0,t\in \R}\left\{su+tv-\log\left( \int_{D_V} e^{sV(x)+tx^2}dx\right)\right\} \quad\text{for}\quad u,v\in\R_+,\label{J_uv}
\end{align}

We now state the main result of this section.

\begin{proposition}[Annealed example, Orlicz] \label{lem-orlex} 
  {Suppose $V$ is a superquadratic Orlicz function} and for $n\in\N$, let
  $X^{(n)} \sim \textnormal{Unif}(\mathbb{B}_V^n)$.
  Then $\{X^{(n)}\}_{n\in\N}$ satisfies Assumption \ref{ass-normldp}* with $J_X = J_{X,V}$,  where  
  \begin{equation}
    \label{def-JV}
J_{X,V}(z) := \mathcal{J}(1,z^2)-\sup_{s<0}\left\{s- \log\left(\int_{D_V} e^{s V(x)}dx\right)\right\},\quad z\in\R_+. 
  \end{equation}
  Moreover,  there exists a unique  $b^* > 0$ such that $M_V(\mu_{V,b^*}) = 1$ and
  $J_{X,V}$ has a unique minimizer $m:=\sqrt{M_2(\mu_{V,b^*})}$, with $M_2$ as defined in \eqref{secmom}. 
\end{proposition}

The proof of Proposition~\ref{lem-orlex}, which is given at the end of the section,
relies on certain properties of the function $\mathcal{J}$ specified in~\eqref{J_uv}, which we state in the lemma
below, whose proof is relegated to Appendix \ref{ap-orlicz}. 
 For $s<0$, $t\in\R$, define $\nu_{s,t}\in\mathcal{P}(\R)$ as 
 \[
 \nu_{s,t}(dx):=\frac{1}{Z_{s,t}}e^{sV(x)+tx^2}dx, \quad   x\in {D_V},
 \]
 where
  $Z_{s,t}$ is the normalizing constant $Z_{s,t}:=\int_{D_V} e^{sV(x)+tx^2}dx$,
 which is finite since $V$ is superquadratic and $s < 0$.

\begin{lemma} \label{lem-prop-J}
Suppose $V$ is a  superquadratic Orlicz function
  and let $\mathcal{J}$ be defined in~\eqref{J_uv}. 
\begin{enumerate}
\item For $u,v\in\R_+$ and $(u,v)$ lying in the interior  of the domain of ${\mathcal J}$,
  there exists a unique $(s,t)\in\R_-\times\R$ such that the supremum in the definition of~\eqref{J_uv} of $J(u,v)$ is attained. Moreover, this  $(s,t)$ satisfies $M_V(\nu_{s,t})\leq u$ and $M_2(\nu_{s,t})=v$. 
\item There exists a unique constant $b^*>0$ such that $M_V(\mu_{V,b^*}) = 1$. Moreover,  $v\mapsto\mathcal{J}(1,v)$ is a convex function with minimizer $m=M_2(\mu_{V,b^*})$, and furthermore, 
\begin{align}\label{JV-min}
{ \min_{x \in \R_+} \mathcal{J}(1,x)  = } \mathcal{J}(1,m)=\sup_{s<0}\left\{s- \log\left(\int_{D_V} e^{s V(x)}dx\right)\right\}.
\end{align}
\item For $v>m$, the supremum in the definition~\eqref{J_uv} of $\mathcal{J}(1,v)$ is attained at $(s,t)\in\R_-\times\R_+$ while for $0<v<m$, the supremum  is attained at $(s,t)\in\R_-\times\R_-$. 
\item For  $u,v\in\R_+$ such that ${\mathcal J}(u,v) < \infty$,  $\partial_u\mathcal{J}(u,v)\leq 0$. 
\end{enumerate}
\end{lemma}

{
\begin{remark}
Let $V, \tilde{V}$ be  Orlicz functions,  and suppose now that the definition of $\mathcal{J}$ in \eqref{J_uv} is replaced with 
\begin{align}\label{new-J}
\tilde{\mathcal{J}}(u,v):=\sup_{s<0,t\in\R}\left\{su+tv-\log\left(\int_{D_V}e^{sV(x)+t\tilde{V}(x)}dx \right) \right\} \quad \text{for} \quad u,v\in\R_+.
\end{align}
Then, as can be seen from the proof  in Appendix \ref{ap-orlicz}, the conclusions of Lemma \ref{lem-prop-J}
continues to hold, with  $M_{\tilde{V}}$ and $\tilde{\mathcal{J}}$ in place of  $M_2$ and $\mathcal{J}$, respectively, as long as 
 $V(x)/\tilde{V}(x) \to\infty$ as $x\to\infty$. 
\end{remark}
}

\begin{proof}[Proof of Proposition~\ref{lem-orlex}]
  Fix a superquadratic Orlicz function $V$ (see Definition~\ref{def-orl}).
  For a measurable set $A\subset \R$, define $\Bva{A}:=\mathbb{B}^n_V\cap\{x\in\R^n:\sum_{i=1}^nx_i^2\in nA \}$, and note that
\begin{equation}\label{vol-ratio}
\mathbb{P}\left( \frac{1}{n}\sum_{i=1}^n(X^n_i)^2\in A\right) = \frac{\abs{\Bva{A}}}{\abs{\mathbb{B}^n_V}},
\end{equation}
which expresses a tail probability in terms of {the ratio of  volumes of two convex bodies.  We now proceed in three steps to characterize the asymptotics of these volumes in order
to estimate  the tail probabilities.} 
\begin{step}
For any closed set $F\subset\R_+$, we will establish an upper bound on the numerator in~\eqref{vol-ratio},
\begin{equation}\label{num-upper}
\limsup_{n\to\infty}\frac{1}{n}\log\abs{\Bva{F}} \leq -\inf_{x\in F}\mathcal{J}(1,x).
\end{equation}
\end{step}
{With $b_* > 0$ as  defined in property 2 of} Lemma \ref{lem-prop-J},
and $M_2$ as defined in \eqref{secmom},  set  $m := M_2 (\mu_{V,b^*})$. 
Next, set  $\alpha_+:=\min\{x\in[m,+\infty)\cap F\}$ and $\alpha_-:=\max\{x\in[0,m]\cap F\}$. 
Assume $\alpha_-<m<\alpha_+$. Then 
\begin{align}
  \notag
  \abs{\Bva{F}}& \leq \abs{\mathbb{B}^n_V\cap\left\{x\in\R^n:\sum_{i=1}^nx_i^2\geq n\alpha_+ \right\}}+\abs{\mathbb{B}^n_V\cap\left\{x\in\R^n:\sum_{i=1}^nx_i^2\leq n\alpha_- \right\}}\\
   \label{sum}
&=\abs{\Bva{[\alpha_+,\infty)}}+\abs{\Bva{[0,\alpha_-]}}.
\end{align}
Fix $s<0$ and $t>0$, and note that then 
\[
\Bva{[\alpha_+,\infty)}=\left\{x\in\R^n:\exp\left(s\sum_{i=1}^nV(x_i)\right)\geq e^{ns},\exp\left(t\sum_{i=1}^nx_i^2\right)\geq e^{nt\alpha_+} \right\}, 
  \]
  and therefore,  it follows that 
\begin{align*}
\abs{\Bva{[\alpha_+,\infty)}}
& \leq \int_{\Bva{[\alpha_+,\infty)}}\exp\left(s\sum_{i=1}^nV(x_i)-ns+t\sum_{i=1}^nx_i^2-nt\alpha_+\right)dx\\
& \leq e^{-ns-nt\alpha_+}\int_{{(D_V)}^n}\exp\left(s\sum_{i=1}^nV(x_i)+t\sum_{i=1}^nx_i^2\right)dx. 
\end{align*}
Hence, for every $s<0$ and $t>0$, we see that  
\[
\limsup_{n\to\infty}\frac{1}{n}\log\abs{\Bva{[\alpha_+,\infty)}}
\leq -s-t\alpha_++\log\left( \int_{D_V} e^{sV(x)+tx^2}dx\right).
\]
Taking the infimum over $s<0$ and $t>0$, we obtain
\begin{align*}
\limsup_{n\to\infty}\frac{1}{n}\log\abs{\Bva{[\alpha_+,\infty)}}
&\leq \inf_{s<0,t>0}\left\{-s-t\alpha_++\log\left( \int_{D_V} e^{sV(x)+tx^2}dx\right)\right\}\\
&=-\sup_{s<0,t>0}\left\{s+t\alpha_+-\log\left( \int_{D_V} e^{sV(x)+tx^2}dx\right)\right\}\\
&=-\mathcal{J}(1,\alpha_+),
\end{align*}
where, since $\alpha_+ > m$,  the last equality follows by property 3 of Lemma~\ref{lem-prop-J}. 
  
Similarly, once again fixing   $s<0$ but now also choosing $t < 0$, we  have
\begin{align*}
\limsup_{n\to\infty}\frac{1}{n}\log\abs{\Bva{[0,\alpha_-)}}
&\leq-\sup_{s<0,t<0}\left\{s+t\alpha_--\log\left( \int_{D_V} e^{sV(x)+tx^2}dx\right)\right\}\\
&=-\mathcal{J}(1,\alpha_-),
\end{align*}
where, noting that $\alpha_- < m$,  the last inequality  once again follows by property 3
of  Lemma~\ref{lem-prop-J}.
{Together with \eqref{sum},  the last two displays show that
\begin{align*}
  \limsup_{n\to\infty}\frac{1}{n} \log | B_{2,V}^n[F] | \leq
-\min \left( {\mathcal J}(1,\alpha_+), {\mathcal J}(1,\alpha_-) \right) =
- \inf_{x \in F} {\mathcal J}(1,x), 
\end{align*}
where the last equality follows from the convexity of $x \rightarrow {\mathcal J}(1,x)$ and 
the fact that $m$ is the unique minimizer of ${\mathcal J}(1,\cdot)$, as  established in
property 2 of Lemma \ref{lem-prop-J}, and the definitions of $\alpha_-$ and $\alpha_+$. 
This proves the claim \eqref{num-upper} in this case. 
  }

On the other hand,  if $\alpha_+=m$ or $\alpha_-=m$, then by property 2 of 
Lemma~\ref{lem-prop-J}, $\inf_{x\in F}\mathcal{J}(1,x)=\mathcal{J}(1,m)$.  
We first obtain an estimate of the {volume of the} Orlicz ball $\mathbb{B}^n_V$. For $s<0$,
\begin{align*}
\frac{1}{n}\log\abs{\mathbb{B}^n_V}&\leq\frac{1}{n}\log\int_{\sum_{i=1}^nV(x_i)\leq n}\exp\left(s\sum_{i=1}^nV(x_i)-ns \right)dx\\
&\leq -s+\log\left(\int_{D_V} e^{sV(x)}dx \right).
\end{align*}
The claim  \eqref{num-upper} then follows in this case as well since
\begin{align*}
\limsup_{n\to\infty}\frac{1}{n}\log\abs{\Bva{F}}& \leq \limsup_{n\to\infty}\frac{1}{n}\log\abs{\mathbb{B}^n_V}\\
& \leq -\sup_{s<0}\left\{s-\log\left( \int_{D_V} e^{sV(x)}dx\right) \right\}\\
&=-\min_{x \in \R_+} \mathcal{J}(1,x) \\
& \leq -\min_{x \in F} \mathcal{J}(1,x), 
\end{align*}
where the second inequality follows on taking the infimum over $s<0$ in the last display
and the equality on the third line follows {from \eqref{JV-min}}.

\begin{step} 
For any open set $U\subset\R_+$, we will show  the lower bound
\begin{equation}\label{num-lower}
\liminf_{n\to\infty}\frac{1}{n}\log\abs{\Bva{U}} \geq -\inf_{x\in U}\mathcal{J}(1,x).
\end{equation}
\end{step}

Since  $\inf_{x\in U}\mathcal{J}(1,x) = \inf_{x \in U \cap D_{{\mathcal J}(1,\cdot)}} \mathcal{J}(1,x)$, with the latter
being infinity if the intersection is empty, it 
suffices to only consider $U$ that  lies in the interior of the domain of $\mathcal{J}(1,\cdot)$.
By property 4 of Lemma~\ref{lem-prop-J}, $\partial_u\mathcal{J}(u,v)\leq0$, and so for each ${v \in\R_+}$, $u\mapsto\mathcal{J}(u,v)$ is decreasing.  By the convexity, and hence continuity, of $(u,v)\mapsto\mathcal{J}(u,v)$ in $\R_+\times\R_+$, for every $\varepsilon >0$, there exist $0<y<1$ and $z\in U$ such that 
\begin{equation} \label{infyx} \inf_{x\in U}\mathcal{J}(1,x)>\mathcal{J}(y,z)-\varepsilon.\end{equation} 
Pick ${\delta>0}$ small such that $(z-\delta,z+\delta)\subset U$. Let $(s_y,t_z)$ be the value that attains the supremum in the expression \eqref{J_uv} for $\mathcal{J}(y,z)$, which exists by property 1 of Lemma~\ref{lem-prop-J} {since
$U$ lies in the interior of the domain of $\mathcal{J}(1, \cdot)$.}  
Define 
$$A^n_\delta:=\left\{x\in\mathbb{R}^n:0 \leq \frac{1}{n}\sum_{i=1}^nV(x_i)< 1, z-\delta<\frac{1}{n}\sum_{i=1}^nx_i^2 < z+\delta\right\}$$
and $\bar{Z}_{y,z}:= Z_{s_y,t_z} = \int_{D_V} e^{s_yV(x)+t_zx^2}dx$,   which is finite by the definition of $(s_y, t_z)$, the definition of $\mathcal{J}(y,z)$ in \eqref{J_uv} and the finiteness of $\mathcal{J}(y,z)$.
Suppose $t_z<0$. We then obtain the following estimate for the lower bound:  
\begin{align*}
\abs{\Bva{U}}
& \geq \int_{A^n_\delta}dx \\
& = \int_{A^n_\delta} \left(\bar{Z}_{y,z} \right)^ne^{-s_y\sum_{i=1}^nV(x_i)-t_zx_i^2}\prod_{i=1}^n\frac{1}{\bar{Z}_{y,z}}e^{s_yV(x_i)+t_zx_i^2}dx\\
& \geq \exp\left(n(\log \bar{Z}_{y,z}-s_yy-t_z(z-\delta))\right)\int_{A^n_\delta}\prod_{i=1}^n\frac{1}{\bar{Z}_{y,z}}e^{s_yV(x_i)+t_zx_i^2}dx.
\end{align*}
The case $t_z>0$ can be handled analogously, simply by replacing $z-\delta$ with $z+\delta$ on the right-hand side.

Let $\{\Xi_i\}_{i=1}^\infty$ be a sequence of i.i.d continuous random variables with density $\frac{1}{\bar{Z}_{y,z}}e^{s_yV(x)+t_zx^2}$. Since $(s_y,t_z)$ attains the supremum in~\eqref{J_uv} of $\mathcal{J}(y,z)$, by property 1 of Lemma~\ref{lem-prop-J},
$\mathbb{E}[V(\Xi_i)]\leq y<1$ and $\mathbb{E}[\Xi_i^2]=z$. Then the integral in the last display can be rewritten as 
\[
\int_{A^n_\delta}\prod_{i=1}^n\frac{1}{\bar{Z}_{y,z}}e^{s_yV(x_i)+t_zx_i^2}dx=\mathbb{P}\left(0 \leq\frac{1}{n}\sum_{i=1}^nV(\Xi_i)< 1,z-\delta <\frac{1}{n}\sum_{i=1}^n\Xi_i^2< z+\delta\right),
\]
which converges to $1$ by the weak law of large numbers and finiteness of the respective moments. 
Thus, combining the last two displays with  
 the expression for ${\mathcal J}(y,z)$ in  \eqref{J_uv} and  the definition of $(s_y,t_z)$, 
we obtain  
\begin{align*}
  \liminf_{n\to\infty}\frac{1}{n}\log\abs{\Bva{U}}
  &\geq
  \log \bar{Z}_{y,z}-s_y(y-\delta)-t_z(z-\delta)\\
&=-\mathcal{J}(y,z) +s_y\delta+t_z\delta\\
&\geq {-\inf_{x\in U}\mathcal{J}(1,x)}+\varepsilon+s_y\delta+t_z\delta, 
\end{align*}
where the last inequality invokes \eqref{infyx}. 
Since $\varepsilon$ and $\delta$ are arbitrary, this implies the desired lower bound in~\eqref{num-lower}.

\begin{step}
We claim that the denominator of~\eqref{vol-ratio} satisfies
\begin{equation}\label{deno}
\lim_{n\to\infty}\frac{1}{n}\log\abs{\mathbb{B}^n_V}= -\sup_{s<0}\left\{s- \log\left(\int_{D_V} e^{s V(x)}dx\right)\right\}.
\end{equation}
\end{step}
Since $\mathbb{B}^n_V=\Bva{\R_+}$, 
applying \eqref{num-upper} with $F = \R_+$ and \eqref{num-lower} with $U = (0,\infty)$, we obtain
\[ \lim_{n\to\infty}\frac{1}{n} \log |\mathbb{B}^n_V| = -\inf_{x\in\R_+}\mathcal{J}(1,x). \]
The claim then follows from
{\eqref{JV-min}} of Lemma~\ref{lem-prop-J}.

Finally, the combination of~\eqref{deno},~\eqref{num-upper} and~\eqref{num-lower} with ~\eqref{vol-ratio} yields an LDP for $\{\frac{1}{n}\sum_{i=1}^n(X^n_i)^2\}_{n\in\N}$ at speed $n$ with rate function
\[
\mathcal{J}(1,x)-\sup_{s<0}\left\{s- \log\left(\int_{D_V} e^{s V(x)}dx\right)\right\}.
\]
The proposition then follows by an application of the contraction principle to the mapping {$x\mapsto\sqrt{x}$}.
\end{proof}

\begin{remark}
    The relation \eqref{deno} provides the asymptotic  logarithmic volume of an  Orlicz ball. 
      A careful inspection of the argument shows that this limit result does not require the super-quadratic
    restriction on the Orlicz function; it is the LDP for the norms that uses this restriction.
    {Indeed, let $V$ and $\tilde{V}$ be Orlicz functions such that
      $V(x)/\tilde{V}(x) \to\infty$ as $x\to\infty$,  let $\tilde{\mathcal{J}}$ be defined as in \eqref{new-J}.
      and for any measurable set $A\subset \R$ define $\mathbb{B}^n_{\tilde{V}}[A]:=\{x\in \R^n:\sum_{i=1}^n \tilde{V}(x_i)\in nA\}$. Then, we see from Step 1 in the proof of Proposition \ref{lem-orlex} that the following more general inequality holds: 
\[
\limsup_{n\to\infty}\frac{1}{n}\log \abs{\mathbb{B}^n_V[A] \cap \mathbb{B}^n_{\tilde{V}}[B]}\leq -\inf_{x\in A,y\in B}\tilde{\mathcal{J}}(x,y), \qquad \mbox{ for closed sets } A, B \subset \R_+. 
\]}
\end{remark} 

\begin{remark}
  \label{rem-Orlicz}
  {After the first version of this paper was posted on arXiv, 
  the asymptotic logarithmic volume of an Orlicz ball was also obtained in~\cite{KabPro21},
  where they also established further refined estimates  using  ideas from sharp large deviations estimates 
  (see, e.g., \cite{BahRao60,LiaRam20a}).} 
  \end{remark}

\subsubsection{Generalized Orlicz balls}\label{sec-gen-orl}

Suppose that $\mathbb{B}_{V_1,\dots,V_n}^n$ is close to a symmetric Orlicz ball in the following sense: there exists an
Orlicz function $\bar{V}:\R_+\ra\R_+$ such that as {$n\to\infty$},
\begin{equation}\label{orldec} \frac{1}{n} \log \frac{ \abs{\mathbb{B}_{\bar{V}}^n \, \Delta \,  \mathbb{B}_{V_1,\dots,V_n}^n}}{ \abs{\mathbb{B}_{\bar{V}}^n \, \cup \,  \mathbb{B}_{V_1,\dots,V_n}^n} } \ra -\infty, 
\end{equation}
where $\mathbb{B}_{\bar{V}}^n$ is the Orlicz ball as defined in \eqref{orlball}, and for any sets
  $A, B \subset \R^n$, $A \Delta B  = [A \setminus B] \cup [B \setminus A]$ represents the
  symmetric difference of $A$ and $B$.

\begin{lemma}[Generalized Orlicz]
  Fix Orlicz functions $\{V_i\}_{i \in \N}$ and $\bar{V}$ that satisfy \eqref{orldec},
and are uniformly superquadratic in the sense that there exists a  constant $c_V \in {(0,\infty)}$ such that for $x>c_V$, $V_i(x)>x^2$ for $i\in\N$ and $\bar{V}(x)>x^2$. 
  Suppose $X^{(n)} \sim \textnormal{Unif}( \mathbb{B}_{V_1,\dots,V_n}^n)$. 
    Then, $\{X^{(n)}\}_{n\in\N}$ satisfies Assumption \ref{ass-normldp}* with $J_{X,\bar{V}}$, defined as in \eqref{def-JV}, but with
      $V$ replaced with $\bar{V}$.
\end{lemma}
\begin{proof} 
  Let $\bar{X}^{(n)}\sim \textnormal{Unif}(\mathbb{B}_{\bar{V}}^n)$,  $U^{(n)}\sim \text{Unif}(\mathbb{B}_{\bar{V}}^n \cup \mathbb{B}_{V_1,\dots,V_n}^n)$ and $X^{(n)} \sim \textnormal{Unif}( \mathbb{B}_{V_1,\dots,V_n}^n)$ be independent random vectors
  defined on a common probability space. Define 
\begin{align*}
{W}^{(n)} &:= U^{(n)} \1_{\{U^{(n)} \in \mathbb{B}_{V_1,\dots,V_n}^n\}} + {X}^{(n)} \1_{\{U^{(n)} \not\in \mathbb{B}_{V_1,\dots,V_n}^n\}}, \\
\bar{W}^{(n)} &:= U^{(n)} \1_{\{U^{(n)} \in \mathbb{B}_{\bar{V}}^n\}} + \bar{X}^{(n)} \1_{\{U^{(n)} \not\in \mathbb{B}_{\bar{V}}^n\}}.
\end{align*}
First note that conditioned on $\{U^{(n)} \in \mathbb{B}_{V_1,\dots,V_n}^n\}$, $U^{(n)}$ 
{has the same distribution as} $X^{(n)}$. By definition, {$X^{(n)}$} is supported on $\mathbb{B}_{V_1,\dots,V_n}^n$. For a measurable set $A\subset\mathbb{B}_{V_1,\dots,V_n}^n$, we have
\begin{align*}
\mathbb{P}\left(W^{(n)}\in A \right) &= \mathbb{P}\left(W^{(n)}\in A \middle | U^{(n)}\in\mathbb{B}_{V_1,\dots,V_n}^n\right)\mathbb{P}\left( U^{(n)}\in\mathbb{B}_{V_1,\dots,V_n}^n\right)\\
&\qquad+\mathbb{P}\left(W^{(n)}\in A \middle | U^{(n)}\not\in\mathbb{B}_{V_1,\dots,V_n}^n\right)\mathbb{P}\left( U^{(n)}\not\in\mathbb{B}_{V_1,\dots,V_n}^n\right)\\
& = \mathbb{P}\left(U^{(n)}\in A \middle | U^{(n)}\in\mathbb{B}_{V_1,\dots,V_n}^n\right)\mathbb{P}\left( U^{(n)}\in\mathbb{B}_{V_1,\dots,V_n}^n\right)\\
&\qquad+\mathbb{P}\left(X^{(n)}\in A \middle | U^{(n)}\not\in\mathbb{B}_{V_1,\dots,V_n}^n\right)\mathbb{P}\left( U^{(n)}\not\in\mathbb{B}_{V_1,\dots,V_n}^n\right)\\
& = \mathbb{P}\left(X^{(n)}\in A\right)\mathbb{P}\left( U^{(n)}\in\mathbb{B}_{V_1,\dots,V_n}^n\right)+\mathbb{P}\left(X^{(n)}\in A\right)\mathbb{P}\left( U^{(n)}\not\in\mathbb{B}_{V_1,\dots,V_n}^n\right)\\
&=\mathbb{P}\left(X^{(n)}\in A \right).
\end{align*}
Hence, $W^{(n)} \eqdist X^{(n)}$.  A similar argument can be used to show
that $\bar{W}^{(n)} \eqdist \bar{X}^{(n)}$, thus 
providing a useful coupling between the uniform measures on $\mathbb{B}_{V_1,\dots,V_n}^n$ and $\mathbb{B}_{\bar{V}}^n$.
{Due to the contraction principle applied to $x\mapsto\sqrt{x}$, Proposition \ref{lem-orlex} and Remark \ref{rem-expeq},}
it suffices to show that $\{\|W^{(n)}\|_2^2/n\}_{n\in\N}$ is exponentially equivalent (see Definition \ref{def-exp}) to $\{\|\bar{W}^{(n)}\|_2^2/n\}_{n\in\N}$.

{To this end, 
  let $\kappa:=\sqrt{c_V^2+1}$, where $c_V$ is the positive finite constant stated in the lemma, 
  and note that $\mathbb{B}_{V_1,\dots,V_n}^n\subset \kappa\sqrt{n}\mathbb{B}^n_2$ and $\mathbb{B}_{\bar{V}}^n\subset \kappa\sqrt{n}\mathbb{B}^n_2$ since for $x\in\mathbb{B}_{V_1,\dots,V_n}^n$, 
  $\sum_{i=1}^nx_i^2\leq\sum_{i=1}^n (c_V^2+V_i(x_i)) \leq (c_V^2+1)n=\kappa^2 n,$ where the last inequality uses
  \eqref{gen-orlball}, and  for $x\in\mathbb{B}_{\bar{V}}^n$,
  $\sum_{i=1}^nx_i^2\leq\sum_{i=1}^n (c_V^2+\bar{V}(x_i))\leq (c_V^2+1)n=\kappa^2 n,$ where
  the last inequality uses
  \eqref{orlball}. 
Thus,
\begin{align*}
\left|  \tfrac{\|\bar{W}^{(n)}\|^2_2}{n} - \tfrac{\|W^{(n)} \|^2_2}{n} \right|  &= \frac{1}{n} \left| \| \bar{W}^{(n)}\|^2_2  - \|W^{(n)} \|^2_2 \right| \le  \kappa \1_{\{U^{(n)} \not\in (\mathbb{B}_{\bar{V}}^{n} \cap \mathbb{B}_{V_1,\dots, V_n}^n) \}}.
 \end{align*}
Therefore, for every $\delta > 0$, 
\begin{align*}
  \frac{1}{n} \log  \P\left( \left| \| \bar{W}^{(n)}\|^2_2/n - \|W^{(n)} \|^2_2/n \right| > \delta \right)
  &\le \frac{1}{n} \log \P\left(\kappa \1_{\{U^{(n)} \not\in (\mathbb{B}_{\bar{V}}^{n} \cap \mathbb{B}_{V_1,\dots, V_n}^n) \}} > 0\right) \\
  &= \frac{1}{n} \log \P\left(  U^{(n)} \in  (\mathbb{B}_{\bar{V}}^{n} \,  \Delta \,  \mathbb{B}_{V_1,\dots, V_n}^n) \right) \\
  &=  \frac{1}{n} \log \frac{\abs{\mathbb{B}_{\bar{V}}^n \, \Delta \,  \mathbb{B}_{V_1,\dots,V_n}^n}}{ \abs{\mathbb{B}_{\bar{V}}^n \, \cup \,  \mathbb{B}_{V_1,\dots,V_n}^n}},  
\end{align*}}
which converges to $-\infty$  as $n \rightarrow \infty$ by  \eqref{orldec}.  This establishes the desired
exponential equivalence and thus completes the proof. 
\end{proof}

\subsection{Gibbs measures}\label{sec-gibbs}

 We now consider the case when the random vector $X^{(n)}$ is drawn from a \emph{Gibbs measure} on configurations of $n$ interacting particles. To be precise, let $F: \R \ra (-\infty, \infty]$ be a ``confining" potential, $G:\R\times \R \ra (-\infty,\infty]$ an ``interaction" potential, and for $n\in \N$, define a Hamiltonian $\mathbf{H}_n:\R^{n} \ra (-\infty,\infty]$ given by
\begin{equation*} \mathbf{H}_n(x) := \frac{1}{n}\sum_{i=1}^n F(x_i) + \frac{1}{n^2} \sum_{i=1}^n \sum_{j=1, j \neq i}^n G(x_i, x_j), \quad x\in\R^n.
\end{equation*}
As in \cite{DupLasRam15}, we assume that there exists a  non-atomic, {$\sigma$-finite measure $\ell$} on $\R$ which, 
along with 
the potentials $F$ and $G$,  satisfies the following conditions: 
\begin{enumerate}
\item $F$ and $G$ are lower semicontinuous on the respective sets on which they are finite; 
\item there exists $a \in [0, 1)$ and $c \in \R$ such that $F$ satisfies $\int_{\R} e^{-(1 - a) F(x)}\ell(dx) < \infty$, $\inf_x F(x) > c$, and $\inf_{(x,y)\in \R\times \R}\left[ G(x,y) + a(F(x) + F(y))\right] > c$;
\item there exists a Borel measurable set $A\subset \R$ with $\ell(A) > 0$ such that
 \begin{equation*} \int_{A\times A} \left[ F(x) + F(y) + G(x,y)\right] \ell(dx) \ell(dy) < \infty;
\end{equation*}
\item for all $\lambda \in \R$, we have 
\begin{equation*} \int_{\R\times\R} \exp\left[ \lambda (x^2+y^2) - F(x) - F(y) - G(x,y)\right] \ell(dx) \ell(dy) < \infty.
\end{equation*}
\end{enumerate}

{Under the above conditions, it is straightforward to verify that for $n\in \N$,
  $Z_n := \int_{\R^n} e^{-n\mathbf{H}_n(x)} \ell^{\otimes n}(dx)$  
  is finite  and so we can 
define  $P_n \in \cP(\R)$ as follows:} 
\begin{equation}\label{gibbs} P_n(dx) := \frac{1}{Z_{n}} e^{-n\mathbf{H}_n(x)} \ell^{\otimes n}(dx), \quad x \in \R^n. 
\end{equation}
Further, {let $Q_n \in \cP(\R)$ be the empirical measure of  the coordinates of $X^{(n)}$, when the latter
  is drawn from $P_n$; more precisely, set  $Q_n:=\frac{1}{n}\sum_{i=1}^n \delta_{X^{(n)}_i}$, which is a random $\cP(\R)$-valued
element.}

\begin{theorem}[Theorem 2.7 and Lemma 2.6 of \cite{DupLasRam15}] \label{th-gibbs}
Under the conditions stated above, 
 $\{Q_n\}_{n\in\N}$ satisfies an LDP in $\cP_2(\R)$ equipped with the $2$-Wasserstein topology, at speed $n$, with GRF $\mathcal{I}_*:\cP_2(\R) \ra [0,\infty]$ defined by
\begin{align}
  \label{Gibbsrfn}
  \mathcal{I}_*(\mu) &:=  \mathcal{I}(\mu) - \inf_{\mu \in \cP_2(\R)} \mathcal{I}(\mu), \\
  \notag 
  \mathcal{I}(\mu) &:= H(\mu |\ell) + \frac{1}{2} \int_{\R\times \R} G(x,y) \mu(dx) \mu(dy) + \int_{\R} F(x)  \mu(dx), 
\end{align}
with $H$ being the relative entropy functional defined in \eqref{relent}. 
\end{theorem}

\begin{proposition}[Gibbs measures]\label{lem-gibbs}
  {Suppose $F, G$ and $\ell$ satisfy the conditions stated above, and 
    for $n\in \N$, suppose $X^{(n)}$ is drawn from $P_n$ of \eqref{gibbs}.}
  Then $\{X^{(n)}\}_{n\in\N}$ satisfies Assumption \ref{ass-normldp}* with
  GRF
  \[
  J_X(x):= \inf\left\{{\mathcal{I}_*}(\mu):\mu\in\cP_2(\R),x = \sqrt{M_2(\mu)}\right\},\quad x\geq 0,
  \]
  {with $\mathcal{I}_*$ as defined in \eqref{Gibbsrfn}.} 
\end{proposition}
\begin{proof}
  The LDP for $\{\|X^{(n)}\|_2^2/n\}_{n\in \N}$, follows from an application of the contraction principle to  the empirical measure LDP of Theorem \ref{th-gibbs} and the second moment map $M_2$ of \eqref{secmom}, which is continuous with respect to the $2$-Wasserstein topology. Assumption \ref{ass-normldp}* then follows on applying the contraction principle with the map
    $x \mapsto \sqrt{x}$.
\end{proof}

\begin{remark}
  In a similar fashion, the large deviation results  in the  recent work of \cite[Theorem 2.8]{LiuWu20}
  can be combined with  the contraction principle to show that Assumption \ref{ass-normldp}*
  is also satisfied by sequences of Gibbs measures associated with a class of  interaction potentials
  $G:\R^k \to (-\infty,\infty]$ that capture $k$-tuple interactions, for fixed $k \in \N$.
\end{remark}

\section{Proofs of  LDPs for the random projections}\label{sec-anpr}

In this section, we prove the main large deviation results, namely, Theorems \ref{th-constant}, \ref{th-on} and \ref{th-simn}, stated in Section \ref{sec-mainres}, which deal with the constant, sublinear and linear regimes, respectively.  At many points, we will refer to certain properties of the top row of $\mathbf{A}_{n,k_n}$, which we first establish in
Section \ref{sec-top}.   Sections \ref{sec-constant}, \ref{sec-on} and \ref{sec-simn} consider the regimes of $\{k_n\}$ constant, sublinear and linear, respectively.  Throughout, let $\zeta_1,\zeta_2,\dots,$ denote a sequence of i.i.d.\ standard Gaussian random variables,  let $\zeta^{(n)}:=(\zeta_1,\dots,\zeta_n) \in \R^n$, let
$e_1 := (1,0,\dots,0)\in \R^n$. 

\subsection{The top row of $\mathbf{A}_{n,k_n}$}\label{sec-top}

\begin{lemma}\label{lem-row1}
Fix $k, n\in\N$ such that $k \leq n$.   Then the following relation holds: 
\begin{equation}\label{eq-eqd}
  \mathbf{A}_{n,k}(1, \cdot\, )=\left(  \mathbf{A}_{n,k}(1,1 ),\ldots,  \mathbf{A}_{n,k}(1, k ) \right) \eqdist \frac{(\zeta_1,\dots,\zeta_k)}{\|\zeta^{(n)}\|_2}.
\end{equation}
\end{lemma}
\begin{proof} 
  Let $\mathbf{O}_n$ be a random $n\times n$ orthogonal matrix (i.e., sampled from the normalized Haar measure on the group of $n\times n$ orthogonal matrices). Let $I_{n,k}$ be the $n\times k$ matrix of ones on the diagonal and zeros elsewhere. Note that
  \[ \mathbf{A}^T_{n,k} \eqdist I_{n,k}^T \mathbf{O}_n, \]
  which implies that $\mathbf{A}_{n,k}(1,\cdot\,)$ is equal in distribution to the vector of  the first $k$ elements in the top row of $\mathbf{O}_n$. The marginal distribution of the top row of $\mathbf{O}_n$ is the uniform measure on the unit sphere of $\R^n$, which establishes the identity \eqref{eq-eqd}, due to the classical fact that $\zeta^{(n)} / \|\zeta^{(n)}\|_2$ is uniformly distributed on the unit sphere in $\R^n$ (see, e.g., \cite{Mul59}).
\end{proof}

\begin{lemma}\label{lem-toprow}
Fix $n\in \N$ and $k \leq n$.  Suppose $X^{(n)}$ is an $n$-dimensional random vector independent of $\mathbf{A}_{n,k}$. Then the following relation holds:  
\begin{equation*}
\left( \mathbf{A}_{n,k}^T \tfrac{X^{(n)}}{\|X^{(n)}\|_2}, \, \|X^{(n)}\|_2 \right) \eqdist  \left( \mathbf{A}_{n,k}^T e_1, \, \|X^{(n)}\|_2\right). 
\end{equation*}
\end{lemma}
\begin{proof}
  { As in  Lemma 6.3 of \cite{GanKimRam18}, which considered 
    the case $k=1$, this result  
 is   easily deduced from the fact that the distribution of 
$\mathbf{A}_{n,k}$ is invariant under orthogonal transformations and independent of $X^{(n)}$.
In particular, fix $n \in \N$ and given  any $n \times n$ orthogonal matrix ${O}_n$
and $x \in \R^n$,  we have 
\[  \left( \mathbf{A}_{n,k}^T  \frac{x}{\|x\|_2},  \|x\|_2 \right) \eqdist \left( \mathbf{A}_{n,k}^T {O}_n^{-1} \frac{x}{\|x\|_2}, \|x\|_2 \right). 
\]
Now, given  $e_1 \in S^{n-1}$,  for any $y \in S^{n-1}$ let ${\mathcal R}(y) \in \mathbb{V}_{n,n}$ be the unique  orthogonal matrix
${O}_n$ such that ${O}_n^{-1} y = e_1$.  It is easy to see that ${\mathcal R}:S^{n-1} \mapsto  \mathbb{V}_{n,n}$ is a measurable map.
Then, for any $x \in \R^n$, substituting ${O}_n = {\mathcal R}(x/\|x\|_2)$ in the last display it follows that
\[ \left( \mathbf{A}_{n,k}^T  \frac{x}{\|x\|_2},  \|x\|_2 \right) \eqdist \left( \mathbf{A}_{n,k}^T \left({\mathcal R}\left(\frac{x}{\|x\|_2}\right)\right)^{-1} \frac{x}{\|x\|_2}, \|x\|_2 \right) =  \left( \mathbf{A}_{n,k}^T e_1, \|x\|_2 \right).
\]
Since $\mathbf{A}_{n,k}^T$ is independent of $X^{(n)}$,  the above relation holds when $x$ is replaced with $X^{(n)}$. 
} 	
\end{proof}

In the settings in which   $\{k_n\}_{n\in\N}$ grows, we will first analyze the empirical measure of the
$k_n$ elements in the top row of $\sqrt{n}\mathbf{A}_{n,k_n}$. That is, let 
\begin{equation}\label{eq-muna}
  \hat{\mu}^n_{\mathbf{A}} := \frac{1}{k_n} \sum_{j=1}^{k_n} \delta_{\sqrt{n}\mathbf{A}_{n,k_n}(1,j)}, \quad n\in\N.
\end{equation}
Recall that ${\mathcal P}(\R)$ is always equipped with the weak topology unless otherwise stated.

\begin{lemma} \label{lem-repn}
For $n\in\N$,  let $X^{(n)}$ be independent of $\mathbf{A}_{n,k_n}$ and recall the definition of $\ln$ given in 
\eqref{eq-empproj}.    Then we have 
\begin{equation*}
  \ln(\cdot) \eqdist \hat{\mu}^n_{\mathbf{A}} (\,\cdot \times \sqrt{n} / \|X^{(n)}\|_2).
\end{equation*}
Moreover, the map ${\mathcal P} (\R) \times (0,\infty) \ni (\nu,c) \mapsto \nu(\,\cdot \times c^{-1}) \in {\mathcal P}(\R)$ 
is continuous.  
\end{lemma}

\begin{proof} 
For each $n \in \N$, using the definition of $\ln$ from \eqref{eq-empproj} and Lemma \ref{lem-toprow},  we have 
\begin{align*}
L^{n} (\cdot) = \frac{1}{k_n}\sum_{j=1}^{k_n} \delta_{\|X^{(n)}\|_2 \left(\mathbf{A}_{n,k_n}^T \tfrac{X^{(n)}}{\|X^{(n)}\|_2}\right)_j}(\cdot) 
  &\eqdist 	\frac{1}{k_n}\sum_{j=1}^{k_n}\delta_{\frac{\|X^{(n)}\|_2}{\sqrt{n}} \sqrt{n}  \left(\mathbf{A}_{n,k_n}^T e_1\right)_j} (\cdot) \\
  &= 	\frac{1}{k_n}\sum_{j=1}^{k_n}\delta_{\sqrt{n}\mathbf{A}_{n,k_n}(1,j)} (\, \cdot \, \times \sqrt{n} / \|X^{(n)}\|_2 ), 
\end{align*}
from which the first assertion  of the lemma follows by \eqref{eq-muna}. 
{Continuity of the map $(\nu,c) \mapsto \nu(\,\cdot \times c^{-1})$ follows if and only if as $n\to\infty$ a sequence of random variables $\{B_n\}_{n\in\N}$ converging in distribution to $B$ and a sequence of constants $\{c_n\}$ converging to $c$ implies that $c_nB_n$ converges to $cB$ in distribution as $n\to\infty$, which follows from Slutsky's theorem~\cite[Theorem 13.18]{Kle13}. }
\end{proof}

\begin{lemma}\label{sub-aux}
Suppose $\{\hat{\mu}^n_{\mathbf{A}}\}_{n\in\N}$ and $\{\|X^{(n)}\|_2/\sqrt{n}\}_{n\in\N}$ satisfy LDPs, both at speed $s_n$, and  
with GRFs $\mathbb{I}_1:\cP_q(\R)\to[0,\infty]$ and $\mathbb{I}_2:\R_+\to[0,\infty]$, respectively. Then $\{L^n\}_{n\in\N}$ satisfies an LDP at speed $s_n$ with GRF defined by
\[
\mathbb{I}(\mu) = \inf_{\nu\in \cP_q(\R), c\in \R_+} \left\{\mathbb{I}_1(\nu)+ \mathbb{I}_2(c) : \mu = \nu(\,\cdot \times c^{-1})\right\},\quad\mu\in\cP_q(\R).
\]
\end{lemma}
\begin{proof}
By the independence of $\mathbf{A}_{n,k_n}$ and $X^{(n)}$, $\{ \hat{\mu}^n_{\mathbf{A}} , \|X^{(n)}\|_2/\sqrt{n}\}_{n\in\N}$
satisfies an LDP at speed $s_n$ with GRF 
\begin{equation*}
\cP_q(\R) \times \R_+ \ni (\nu, c) \mapsto  \mathbb{I}_1(\nu)+ \mathbb{I}_2(c) \in [0,\infty].
\end{equation*}
The claim follows on applying the contraction principle to the mapping $F:\cP_q(\R) \times \R_+\to\cP_q(\R)$ defined by $F(\nu,c)=\nu(\cdot\times c^{-1})$\end{proof}

\begin{lemma} \label{lem-nu_n}
Let $\{\zeta_j\}_{j \in \N}$ be i.i.d. $\mathcal{N}(0,1)$ random variables, $\zeta^{(n)} = (\zeta_1, \ldots, \zeta_n)$ and 
consider the sequence
\begin{equation}
  \label{def-nun2} 
\nu_n :=  \frac{1}{k_n} \sum_{j=1}^{k_n} \delta_{\zeta_j}, \quad n\in\N. 
\end{equation}
Then ${\{\nu_n\}_{n \in \N}}$ satisfies an LDP in $\cP(\R)$ with respect to the weak topology, at speed $k_n$, with GRF $H(\cdot|\gamma_1)$.
Moreover, for a sequence $\{s_n\}_{n\in\N}$ such that $s_n\ll k_n$, $\{\nu_n\}_{n\in\N}$ satisfies an LDP in $\cP(\R)$ with respect to the weak topology at speed $s_n$ with GRF 
\begin{equation}
  \label{chi-gamma1}
\chi_{\gamma_1}(\nu):=
\left\{\begin{array}{ll}
0 & \text{ if } \nu = \gamma_1\\
+\infty & \text{ else}	
\end{array}\right., \quad \nu\in\cP(\R).
\end{equation}
\end{lemma}
\begin{proof}
The first claim is a direct conclusion of Sanov's theorem~\cite[Theorem 2.1.10]{DemZeiBook} while the second claim follows since $H(\cdot |\gamma_1) : \cP(\R) \ra [0,\infty]$ is convex with a unique minimizer at $\gamma_1$, as in the observation of Remark \ref{ref-speeds}.
\end{proof}

{We now document an  elementary observation since it  will be used multiple times in the sequel.
\begin{remark}
  \label{lem-chi-2}
  Let $\{\zeta_j\}_{j \in \N}$ be i.i.d. $\mathcal{N}(0,1)$ random variables, let $\{k_n\}_{n\in\N}$ be a sequence of positive integers that converges to infinity and let $C_n :=  \frac{1}{k_n} \sum_{j=1}^{k_n} \zeta^2_j$,  $n\in\N$.  It is easy to see that 
  the log moment generating function of $\zeta_1^2$ is given by $-\frac{1}{2}\log (1-2t)$ for $t<1/2$ and infinity, otherwise. 
  Hence,  it is an immediate consequence of  Cram\'er's theorem \cite[Theorem 2.2.1]{DemZeiBook}  that
$\{C_n\}_{n \in \N}$ satisfies an LDP in $\R$  at speed $k_n$ with GRF $J_{\zeta^2}:\R\to[0,\infty]$ 
  of the form
  \begin{equation}
    \label{rate-chi-2}
J_{\zeta^2}(t):=\sup_{s\in\R}\left\{ st+ \frac{1}{2}\log (1-2s)\right\}=
\begin{cases}
\frac{t-1}{2}-\frac{1}{2}\log t, &\quad  t>0,\\
+\infty, &\quad t\leq 0.
\end{cases}
\end{equation}
\end{remark}
  }

\subsection{Proof of the LDP for  random projections in the constant regime}\label{sec-constant}

This section is devoted to the proof of Theorem \ref{th-constant}.  
Throughout this section, fix $k \in \N$ and suppose $k_n = k$ for each $n \in \N$.  
We first give the main idea behind the proof. 
Due to Lemma \ref{lem-toprow}
we have 
\begin{align}
\label{axn1}
\frac{1}{\sqrt{n}} \mathbf{A}_{n,k}^T X^{(n)} \eqdist \frac{1}{\sqrt{n}} \mathbf{A}_{n,k}^T e_1 \| X^{(n)} \|_2 = \mathbf{A}_{n,k}(1, \cdot) \frac{\|X^{(n)}\|_2}{\sqrt{n}}
\end{align}
where $\{\mathbf{A}_{n,k}(1, \cdot)\}_{n \in \N}$ is independent of 
$\{ X^{(n)}\}_{n \in \N}$. 
Thus,
the question essentially reduces to understanding the following: 
suppose two independent sequences of  random vectors 
$\{V_n\}_{n\in\mathbb{N}}$ and $\{W_n\}_{n\in\mathbb{N}}$ (with at least one-being real-valued) 
satisfy LDPs at speeds $\{\beta_n\}_{n\in\N}$, and $\{\gamma_n\}_{n\in\N}$ with GRFs $J_V$  and $J_W$, respectively.
Then we want to understand when an LDP with a non-trivial rate function
can be deduced for the product  sequence $Z_n:=V_n W_n$, $n\in\N$.
When $\beta_n = \gamma_n$, then this is a simple application of the contraction principle.
We will see that this is the case when Assumption \ref{ass-normldp}* holds, that is, when 
$\{\|X^{(n)}\|_2/\sqrt{n}\}_{n \in \N}$ satisfies an LDP at speed $n$, 
by showing that then $\{\mathbf{A}_{n,k}(1, \cdot)\}_{n \in \N}$ 
also satisfies an LDP at speed $n$.
On the other hand, if $\beta_n \ll \gamma_n$ or $\beta_n \gg \gamma_n$, 
then the idea is to find sequences $\{b_n\}_{n \in \N}$ and $\{s_n\}_{n \in \N}$ such that
$\{b_n V_n\}_{n \in \N}$ and $\{b_n^{-1} W_n\}_{n \in \N}$ both satisfy non-trivial LDPs at the same 
speed $s_n$, and once again apply the contraction principle.   This falls into the case of  
Assumption \ref{ass-rescaled},
which states that  $\{b_n\|X^{(n)}\|_2/\sqrt{n}\}_{n \in \N}$, with  $b_n = \sqrt{s_n/n}$,
  satisfies a non-trivial LDP at speed $s_n$, and thus
  the proof in this case follows by showing that $\{b_n^{-1}  \mathbf{A}_{n,k}(1, \cdot)\}_{n \in \N}$ also satisfies
  a non-trivial LDP at speed $s_n$.

\begin{proof}[Proof of Theorem \ref{th-constant}]
We consider two cases.
\begin{case}
Suppose Assumption~\ref{ass-normldp}* holds with GRF $J_X$. 
\end{case}
First, note that by Lemma \ref{lem-row1}, we can apply the LDP in \cite[Theorem 3.4]{BarGamLozRou10}
to find that $\{\mathbf{A}_{n,k}(1,\cdot)\}_{n\in\N}$ satisfies an LDP in $\R^k$ at speed $n$ with GRF 
$\mathcal{J}_k(y) := -\tfrac{1}{2}\log(1-\|y\|_2^2)$,
when $\|y\|_2\le 1$, and $\mathcal{J}_k(y) := +\infty$ otherwise. 
Since $\{X^{(n)}\}_{n \in \N}$ is independent of $\{\mathbf{A}_{n,k}\}_{n \in \N}$,   and the
case assumption shows that $\{\|X^{(n)}\|_2/\sqrt{n}\}_{n \in \N}$ satisfies an LDP with speed $n$ and GRF $J_X$, {Lemma~\ref{ldp-prod} implies that $\{\mathbf{A}_{n,k}(1,\cdot),\|X^{(n)}\|_2/\sqrt{n}\}_{n \in \N}$ satisfies an LDP at speed $n$ with GRF $\mathcal{J}_k(y)+J_X(\alpha)$ for $y\in\R^k$ and $\alpha\in\R_+$.
The contraction principle applied to the mapping $\R^k \times \R_+ \ni (y,\alpha)\mapsto \alpha y \in \R^k$ implies that }
$\{n^{-1/2}\mathbf{A}_{n,k}^T X^{(n)}\}_{n\in\N}$ satisfies an LDP with GRF
\begin{equation*}
  I_{\mathbf{A}X,k}(x) := \inf_{y \in \R^k, z\in \R} \left\{ -\tfrac{1}{2} \log(1- \|y\|_2^2) +  J_X(z) : x = yz,\, \|y\|_2 \le 1,\, z \ge 0\right\}, \quad x\in \R^k. 
\end{equation*}
We can without loss of generality restrict the range of $z$ in the infimum  to $z > 0$; then,  
substituting $y = x/z$ and noting that the constraint  $\|y\|_2 \le 1$ is  equivalent to
$\|x\|_2 \le z$, it follows that
\[  I_{\mathbf{A}X,k}(x) = \inf_{z\ge \|x\|_2} \left\{ -\tfrac{1}{2} \log\left(1- \frac{\|x\|_2^2}{z^2}\right) +  J_X(z) \right\} = \inf_{z >  \|x\|_2} \left\{ -\tfrac{1}{2} \log\left(1- \frac{\|x\|_2^2}{z^2}\right) +  J_X(z) \right\}, \,  x \in \R^k.   
\]
On rewriting the above in terms of $ c= \|x\|_2/z$, we obtain the form 
 \eqref{Jan-constant1} for the rate function $I_{\mathbf{A}X,k}$.
 
\begin{case}
Suppose Assumption~\ref{ass-rescaled} holds with sequence $\{s_n\}_{n \in \N}$ and GRF $J_X$. 
\end{case} 
From Lemma~\ref{lem-toprow} and~\eqref{eq-eqd}, denoting $\zeta^{(n)}:=(\zeta_1,\ldots,\zeta_n)$ with $\{\zeta_i\}_{i\in\mathbb{N}}$ i.i.d. $\mathcal{N}(0,1)$ random variables, 
we have
\begin{equation}\label{reform-bn}
\frac{1}{\sqrt{n}}\mathbf{A}^T_{n,k}X^{(n)}\buildrel (d) \over = {\mathbf{A}_{n,k}(1,\cdot)}\frac{\norm{X^{(n)}}_2}{\sqrt{n}} \buildrel (d) \over = \frac{\zeta^{(k)}/\sqrt{s_n}}{\norm{\zeta^{(n)}}_2/\sqrt{n}}\frac{\sqrt{s_n}\norm{X^{(n)}}_2}{n},
\end{equation}
Now, consider the sequence of vectors
\[
R_n :=\left (\frac{\zeta^{(k)}}{\sqrt{s_n}},\frac{\sqrt{n}}{\norm{\zeta^{(n)}}_2},\frac{\sqrt{s_n}\norm{X^{(n)}}_2}{n} \right),\quad n\in\N, 
\]
which are almost surely well-defined.  
Then $\zeta_i/\sqrt{s_n}$ is a $\mathcal{N}(0,1/s_n)$ random variable with $s_n\to \infty$ as $n\to\infty$. Hence, {by the G\"artner-Ellis theorem, the sequence $\{\zeta^{(k)}/\sqrt{s_n}\}_{n\in\mathbb{N}}$ satisfies an LDP in $\R^k$ at speed $s_n$ with GRF $x\mapsto \|x\|^2/2$}. Assumption~\ref{ass-rescaled} implies that $\{\sqrt{s_n}\|X^{(n)}\|_2/n\}_{n\in\N}$  satisfies an LDP, also  at speed $s_n$, and with GRF $J_X$. Finally, by Remark~\ref{lem-chi-2},
{the contraction principle applied to $x \mapsto 1/\sqrt{x}$} and the strong law of large numbers, $\{\sqrt{n}/\|\zeta^{(n)}\|_2\}_{n\in\N}$ satisfies an LDP at speed $n$ and converges almost surely to $1$, which implies that the associated GRF has a unique minimum at $1$. 

Together with the independence of $\{\zeta_i\}_{i=1,\ldots,k}$ and $X^{(n)}$, and Lemma~\ref{ldp-prod} {(with $U_n = \zeta^{(k)}/\sqrt{s_n}$, $V_n = \sqrt{s_n}\|X^{(n)}\|_2/n$,  $W_n = \sqrt{n}/\|\zeta^{(n)}\|_2$ and $m = 1$), this implies
  that $R_n$ satisfies an LDP with GRF given  by
$\R^k \times \R_+ \times \R_+ \ni (r_1,r_2,r_3) \mapsto \frac{\norm{r_1}^2_2}{2}+J_X(r_3)$ if $r_2 = 1$ (and $+\infty$ otherwise). 
   Combining this with \eqref{reform-bn} and 
  the contraction principle for the continuous mapping $\R^k \times \R_+^2  \ni (r_1,r_2,r_3) \mapsto r_1 r_2 r_3  \in \R^k$, 
it follows that} the sequence of $k$-dimensional vectors, $\{n^{-1/2}\mathbf{A}_{n,k}^T X^{(n)}\}_{n\in\N}$, satisfies an LDP at speed $s_n$ with GRF
\begin{align*}
I_{\mathbf{A}X,k}(x)&=\inf\left\{\frac{\norm{r_1}^2_2}{2}+J_X(r_3):r_1\in\mathbb{R}^k,r_3>0 ,x=r_1r_3\right\}\\
&=\inf_{c>0} \left\{J_X\left(\frac{\norm{x}_2}{c} \right)+\frac{c^2}{2}\right\}, 
\end{align*}
which coincides with the expression given in  \eqref{Jan-constant1}. This completes the proof of Theorem \ref{th-constant}.
\end{proof}

\subsection{Proof of the empirical measure LDP in the sublinear regime $1\ll k_n \ll n$}\label{sec-on}

We now present the  proof of Theorem \ref{th-on}. We first start with an  auxiliary result. 

\begin{proposition}\label{lem-aux2}
Fix $q\in[1,2)$. Suppose $\{k_n\}_{n\in\N}$ grows sublinearly. Then, $\{\hat{\mu}^n_{\mathbf{A}}\}_{n\in\N}$ defined in~\eqref{eq-muna} satisfies an LDP in $\cP_q(\R)$ at speed $k_n$ with GRF $H(\cdot |\gamma_1) : \cP(\R) \ra [0,\infty]$.
\end{proposition}

\begin{proof}
Recall that $\{\zeta_j\}_{j \in \N}$ are i.i.d. $\mathcal{N}(0,1)$ random variables, $\zeta^{(n)} = (\zeta_1, \ldots, \zeta_n)$. Due to \eqref{eq-muna} and  Lemma \ref{lem-row1}, $\{\hat{\mu}^n_{\mathbf{A}}\}_{n\in\N}\buildrel (d) \over =\{\tilde{\nu}_n\}_{n\in\N}$, where 
\begin{equation*} 
 \tilde{\nu}_n := \frac{1}{k_n} \sum_{j=1}^{k_n} \delta_{\sqrt{n} \zeta_j / \|\zeta^{(n)}\|_2}, \quad n\in\N.
\end{equation*} 
We now claim (and justify below) that $\{\nu_n\}_{n\in\N}$ defined in Lemma~\ref{lem-nu_n} is exponentially equivalent (at speed $k_n$) to the sequence $\{\tilde{\nu}_n\}_{n\in\N}$ defined in \eqref{def-nun2}.

To prove the claim, let $z_n := \sqrt{n} / \|\zeta^{(n)}\|_2$ and let $d_{\text{BL}}$ denote the bounded-Lipschitz metric (which metrizes weak convergence). Then, letting $\text{BL}(\R)$ denote the space of bounded Lipschitz functions $f:\R\ra\R$ with Lipschitz constant 1,
\begin{align*}
d_{\text{BL}}(\nu_n, \tilde{\nu}_n) &\le \sup_{f\in\text{BL}(\R)} \frac{1}{k_n} \sum_{j=1}^{k_n} |f(\zeta_j) - f(\zeta_j z_n)| \le \frac{1}{k_n} \sum_{j=1}^{k_n} |\zeta_j - \zeta_j z_n| = |z_n - 1| \cdot \frac{1}{k_n} \sum_{j=1}^{k_n} |\zeta_j|.
\end{align*}
Hence, for  $\delta,\epsilon > 0$, we have 
\begin{align*}
\P(d_{\text{BL}}(\nu_n,\tilde{\nu}_n) > \delta) &\le \P\left( |z_n - 1| \cdot \frac{1}{k_n} \sum_{j=1}^{k_n} |\zeta_j| > \delta \right) \le \P\left( \frac{1}{k_n} \sum_{j=1}^{k_n} |\zeta_j| > \frac{\delta}{\epsilon} \right)	 + \P(|z_n -1| > \epsilon).
\end{align*}
Since $\zeta_1$ has a finite exponential moment,  Cram\'{e}r's theorem \cite[Theorem 2.2.1]{DemZeiBook} implies  
\begin{equation*}
  \lim_{n\ra\infty}\frac{1}{k_n} \log \P\left( \frac{1}{k_n} \sum_{j=1}^{k_n} |\zeta_j| > \delta/\epsilon \right) = -\mathcal{I}(\delta/\epsilon),
\end{equation*}
for some convex and superlinear rate function $\mathcal{I}$. Also, by Cram\'er's theorem for $\sum_{i=1}^n \zeta_i^2/n$, 
 the continuity of $x \mapsto 1/\sqrt{x}$ on $(0,\infty)$ and the contraction principle, 
the sequence $\{z_n\}_{n\in\N}$ satisfies an LDP at speed $n$. Hence, due to the sublinear growth of $k_n$, we have 
\begin{equation*}
  \lim_{n\ra\infty}\frac{1}{k_n} \log \P(|z_n -1| > \epsilon) = -\infty .
\end{equation*}
Combining the last three displays, we find that 
\begin{equation*}
  \lim_{n\ra\infty} \frac{1}{k_n} \log \P(d_{\text{BL}}(\nu_n,\tilde{\nu}_n) > \delta) \le -\mathcal{I}(\delta/\epsilon).
\end{equation*}
The claim of exponential equivalence follows on sending $\epsilon\to0$, due to the superlinearity of ${\mathcal{I}}$.

{In light of Remark \ref{rem-expeq}, the last observation taken together with Lemma~\ref{lem-nu_n}  implies that}  
  $\{\hat{\mu}^n_{\mathbf{A}}\}_{n\in\N}$  satisfies an LDP in $\cP(\R)$ at speed $k_n$ with GRF $H(\cdot|\gamma_1)$. 
In order to strengthen the LDP for $\{\hat{\mu}^n_{\mathbf{A}}\}_{n\in\N}$, by \cite[Corollary 4.2.6]{DemZeiBook}, {it suffices to show that $\{\hat{\mu}^n_{\mathbf{A}}\}_{n\in\N}$ is exponentially tight in $\cP_q(\R)$.  Since $\{\hat{\mu}^n_{\mathbf{A}}\}_{n\in\N}\buildrel (d) \over =\{\tilde{\nu}_n\}_{n\in\N}$ and for each $j > 0$, the set  $K_{2,j}$ defined in \eqref{k2} is compact with respect to $\cP_q(\R)$ due to Lemma \ref{lem-compcat}, if suffices to show that for $\varepsilon >0$, there exists $M < \infty$ such that 
  \begin{equation}
    \label{exptight}
    \lim_{n\to\infty}  \frac{1}{k_n}\log\P(\tilde{\nu}_n \in K_{2,M}^c) = \lim_{n\to\infty} \frac{1}{k_n}\log\P(M_2(\tilde{\nu}_n)>M)
    <-\varepsilon. 
  \end{equation}
  Now note that $\{M_2(\tilde{\nu}_n) = \frac{n}{k_n} \sum_{j=1}^{k_n}\frac{\zeta^2_j}{ \|\zeta^{(n)}\|_2}\}_{n\in\N}$ and 
  by Remark~\ref{lem-chi-2}, $\{\|\zeta^{(k_n)}\|^2_2/k_n\}_{n\in\N}$ and $\{\|\zeta^{(n)}\|^2_2/n\}_{n\in\N}$ satisfy LDPs at speed $k_n$ and $n$, respectively, both with the same GRF $J_{\zeta^2}(x)$ defined in \eqref{rate-chi-2}.
  Thus, $\{n/\|\zeta^{(n)}\|^2_2\}_{n\in\N}$ also satisfies an LDP at speed $n$ by the contraction principle.  Together with Lemma \ref{ldp-prod} and the fact that $J_{\zeta^2}$ has a unique minimizer at $1$, this implies the sequence $\{M_2(\tilde{\nu}_n)\}_{n\in\N}$ satisfies an LDP at speed $k_n$ with GRF $J_{\zeta^2}$. Now, for $\varepsilon>0$, pick $M$ such that $J_{\zeta^2}(M)>\varepsilon$. We then obtain 
\begin{align*}
  \lim_{n\to\infty} \frac{1}{k_n}\log\P(M_2(\tilde{\nu}_n)>M)
&\leq-J_{\zeta^2}(M) < - \varepsilon, 
\end{align*}
which   proves \eqref{exptight}, and thus concludes the proof of the lemma. 
}
\end{proof}

\begin{proof}[Proof of Theorem \ref{th-on}]
By Assumption~\ref{ass-normldp}, $\|X^{(n)}\|_2/\sqrt{n}$ satisfies an LDP at speed $s_n$ with GRF $J_X$
    and  when $s_n \gg k_n$ by Remark \ref{ref-speeds} (and the additional assumption that $J_X$ has a unique minimizer $m$),
    $\{\|X^{(n)}\|_2/\sqrt{n}\}_{n\in\N}$ also satisfies an LDP at speed $k_n$ with (degenerate) rate function $\chi_m$.
       Also, by Proposition~\ref{lem-aux2} and Remark \ref{ref-speeds},  $\{\hat{\mu}^n_{\mathbf{A}}\}_{n\in\N}$ satisfies an LDP at speed $k_n$ with  GRF $H(\cdot|\gamma_1)$ and, when $s_n \ll k_n$,   $\{\hat{\mu}^n_{\mathbf{A}}\}_{n\in\N}$ also
  satisfies an LDP at speed $s_n$ with (degenerate) GRF  $\chi_{\gamma_1}$ defined in Remark \ref{ref-speeds}
  (see also Lemma~\ref{lem-nu_n}).

Combining the above observations with Lemma~\ref{sub-aux}, we see that $\{L^n\}_{n\in\N}$ satisfies an LDP at speed
$s_n\wedge k_n$ with GRF $\mathbb{I}_{L,k_n}$, where when $s_n\gg k_n$,
\begin{align*}
\mathbb{I}_{L,k_n}(\mu) &= \inf_{\nu\in \cP_q(\R), c\in \R_+} \left\{ H(\nu | \gamma_1) + \chi_m(c) : \mu = \nu(\,\cdot \times c^{-1})\right\}\\
 &= H(\mu(\,\cdot \times m) | \gamma_1)\\
     &= H(\mu | \gamma_1(\,\cdot \times m^{-1}))\\
     &= H(\mu | \gamma_m),
\end{align*}
when $s_n=k_n$,
\begin{align*}
\mathbb{I}_{L,k_n}(\mu) &= \inf_{\nu\in \cP_q(\R), c\in \R_+} \left\{ H(\nu | \gamma_1) + J_X(c) : \mu = \nu(\,\cdot \times c^{-1})\right\}\\
 &= \inf_{ c\in \R_+} \left\{ H(\mu | \gamma_c) + J_X(c) \right\},
\end{align*}
and when $s_n\ll k_n$,
\begin{align*}
\mathbb{I}_{L,k_n}(\mu) &= \inf_{\nu\in \cP_q(\R), c\in \R_+} \left\{ \chi_{\gamma_1}(\nu)+ J_X(c) : \mu = \nu(\,\cdot \times c^{-1})\right\}\\
 &=\begin{cases}
J_X(c),&\quad \mu=\gamma_c\\
+\infty, &\quad  \text{otherwise}.
\end{cases}
\end{align*}
This proves Theorem \ref{th-on}. 
\end{proof}

\subsection{Proof of the empirical measure LDP in the linear regime $k_n \sim \lambda n$}
\label{sec-simn}

Throughout this  section, suppose $\{k_n\}_{n\in\N}$ grows linearly with rate $\lambda \in (0,1]$. As in the sublinear regime, 
we first analyze the sequence of empirical measures $\{\hat{\mu}^n_{\mathbf{A}}\}_{n\in\N}$ of \eqref{eq-muna} as a precursor to the analysis of $\{\ln\}_{n\in\N}$ of \eqref{eq-empproj}.

\begin{proposition}\label{lem-aux}
Fix $q\in[1,2)$. Suppose $k_n$ grows linearly with rate $\lambda \in (0,1]$. Then, the sequence $\{\hat{\mu}^n_{\mathbf{A}}\}_{n\in\N}$ satisfies an LDP in $\cP_q(\R)$ at speed $n$ with GRF $\mathcal{H}_\lambda$ of \eqref{eq-hrf}.
\end{proposition}

The proof of the above result is deferred to the end of this section. Taking it as given, we now prove the main LDP in the linear regime.

\begin{proof}[Proof of Theorem \ref{th-simn}]
Suppose Assumption~\ref{ass-normldp} is satisfied with corresponding sequence $\{s_n\}_{n\in\N}$.
First, suppose $s_n=n$. Then, {along with   Proposition~\ref{lem-aux} and Lemma~\ref{sub-aux}}, this implies that 
$\{\ln\}_{n\in\N}$ satisfies an LDP at speed $n$ with GRF $\hat{\mathbb{I}}_{L,\lambda}$, defined by 
 \begin{align*}\hat{\mathbb{I}}_{L,\lambda}(\mu) & := \inf_{\nu \in \mathcal{P}(\R), c\in \R_+}
   \left\{ \mathcal{H}_\lambda(\nu) + J_X(c) : \mu (\cdot) = \nu(\,\cdot \times c^{-1})  \right\}, \\
& = { \inf_{c\in \R_+}
   \left\{ \mathcal{H}_\lambda(\mu (\cdot \times c)) + J_X(c)  \right\}, 
   \quad \mu \in \cP(\R). }
 \end{align*}
 {In turn, since $\mathcal{H}_\lambda (\nu) = \infty$ if $M_2(\nu) > 1/\lambda$ and $M_2(\mu(\cdot \times c)) =
 M_2(\mu)/c^2$, this implies
 \begin{equation}
\label{temp-GRF-linear}
   \hat{\mathbb{I}}_{L,\lambda}(\mu) = \inf_{c > \sqrt{\lambda M_2(\mu)}} 
 \left\{ \mathcal{H}_\lambda(\mu (\cdot \times c)) + J_X(c)  \right\}, \quad  \mu \in \cP(\R). 
  \end{equation}
 Now for $\mu \in \cP(\R)$ and $c > \sqrt{\lambda M_2(\mu)}$,  the definitions of $\mathcal{H}_\lambda$ and
 $h$ in \eqref{eq-hrf} and \eqref{ent}, respectively, imply 
  \begin{align*}
    \mathcal{H}_\lambda(\mu(\,\cdot \times c))
   &= -\lambda \, h(\mu(\,\cdot \times c)) + \tfrac{\lambda}{2} \log (2\pi e) + \tfrac{1-\lambda}{2}\log\left(\tfrac{1-\lambda}{1-\lambda\,M_2(\mu(\,\cdot \times c))} \right)\\
   &= -\lambda \, h(\mu) + \lambda \log(c) + \tfrac{\lambda}{2} \log (2\pi e) + \tfrac{1-\lambda}{2}\log\left(\tfrac{1-\lambda}{1-(\lambda / c^2)\,M_2(\mu)} \right) \\
   &= \lambda \log(c)- \tfrac{1-\lambda}{2} \log\left( 1 - \tfrac{\lambda M_2(\mu)}{c^2} \right) -\lambda h(\mu) + \tfrac{\lambda}{2} \log(2\pi e) + \tfrac{1-\lambda}{2} \log(1- \lambda).
  \end{align*}
  When substituted back into \eqref{temp-GRF-linear},  this shows that for every
  $\mu \in \cP(\R)$, 
  $\hat{\mathbb{I}}_{L,\lambda}(\mu) = {\mathbb{I}}_{L,\lambda}(\mu)$, with the latter defined as in \eqref{Jan-linear}. } 
 
 Next, consider the case $s_n\ll n$. {By Lemma~\ref{lem-aux},
 $\{\hat{\mu}^n_{\mathbf{A}}\}_{n\in\N}$ satisfies an LDP at speed $s_n$ with GRF $\mathcal{H}_\lambda$. Therefore, by Remark \ref{ref-speeds} and Remark \ref{rm-min-H}, $\{\hat{\mu}^n_{\mathbf{A}}\}_{n\in\N}$ also satisfies an LDP at speed $s_n$ with GRF $\chi_{\gamma_1}$.} 
  Then, by Assumption~\ref{ass-normldp} and Lemma~\ref{sub-aux}, $\{\ln\}_{n\in\N}$ satisfies an LDP at speed $s_n$ with GRF
\begin{align*}
&   \inf_{\nu\in \cP(\R), c\in \R_+} \left\{ \chi_{\gamma_1}(\nu)+ J_X(c) : \mu = \nu(\,\cdot \times c^{-1})\right\}\\
 &\qquad \qquad  =\begin{cases}
J_X(c),&\quad \mu=\gamma_c,\\
+\infty, &\quad  \text{otherwise},
\end{cases}
\end{align*}
which coincides with the expression for $I_{L,\lambda}(\mu)$ given in Theorem \ref{th-simn} for the case $s_n \ll n$. 
\end{proof}
{
\begin{remark}
  \label{rem-linsnn}
  Note that the above  proof in particular shows that, under Assumption \ref{ass-normldp} the rate function $\mathbb{I}_{L,\lambda}(\mu)$
  from Theorem \ref{th-simn} in the linear regime in the case $s_n = n$, satisfies, for $\lambda > 0$, 
  \[   \mathbb{I}_{L,\lambda}(\mu)  =   \inf_{c\in \R_+}
  \left\{ \mathcal{H}_\lambda(\mu (\cdot \times c)) + J_X(c)  \right\},
  \]
  with $\mathcal{H}_\lambda$ defined as in \eqref{eq-hrf}.
  \end{remark}
 }

The rest of this section is devoted to the proof of Proposition \ref{lem-aux}, which is broken down into the intermediate steps given by Lemmas \ref{lem-i1} and \ref{lem-i2} below.

\begin{lemma}\label{lem-i1}
  {Suppose $k_n/n \to \lambda \in (0,1]$ as $n \to \infty$,
  and  $\{\zeta_j\}_{j \in \N}$ is an i.i.d. sequence with common law $\gamma_1$ (the standard normal distribution). 
 If $\lambda\in(0,1)$, then 
the sequence 	
\begin{equation}\label{eq-i2seq}
 \left( \frac{1}{k_n} \sum_{j=1}^{k_n} \delta_{\zeta_j}\, , \quad\frac{1}{n-k_n}\sum_{j=k_n+1}^n \delta_{\zeta_j} \,, \quad \frac{1}{n}\sum_{j=1}^n \zeta_j^2  \right) , \quad n\in\N,
\end{equation}
satisfies an LDP in $[\cP(\R)]^2 \times \R_+$ at speed $n$ with GRF $I_{1,\lambda}$ defined by
\begin{equation*}
I_{1,\lambda}  (\mu, \nu, s) := \lambda \, H( \mu | \gamma_1) + (1-\lambda) \, H(\nu  | \gamma_1) + \tfrac{1}{2} \left[s - \lambda\,M_2(\mu) - (1-\lambda)\,M_2(\nu)\right] ,
\end{equation*}
if $\lambda\,M_2(\mu) + (1-\lambda)\,M_2(\nu) \le s$, and $I_{1,\lambda}(\mu,\nu,s) := +\infty$ otherwise. 
 On the other hand, if $\lambda = 1$  then the sequence
 \begin{equation*}
 \left( \frac{1}{k_n} \sum_{j=1}^{k_n} \delta_{\zeta_j}\, \,, \quad \frac{1}{n}\sum_{j=1}^n \zeta_j^2  \right) , \quad n\in\N,
\end{equation*}
satisfies an LDP in $\cP(\R) \times \R_+$ at speed $n$ with GRF $I_{1,1}$ defined by
\begin{equation*}
I_{1,1}  (\mu, s) :=  H( \mu | \gamma_1)  + \tfrac{1}{2} \left[s - M_2(\mu) \right] ,
\end{equation*}
if $M_2(\mu)  \le s$, and $I_{1,1}(\mu,s) := +\infty$ otherwise.}
\end{lemma}
\begin{proof}
Our approach is to apply the approximate contraction principle of Proposition \ref{prop-bdgrestate} and Corollary \ref{cor-iid}  of the appendix, with the following parameters:
\begin{itemize}
\item $\Sigma := \R$;
\item $\cX := \R = :\cX^*$;
\item $\mathbf{c}(x) = x^2$, for $x\in \R$;
\item for $n\in \N$, let $\mathscr{L}_n^{(1)} := \frac{1}{k_n} \sum_{j=1}^{k_n}\delta_{\zeta_j} \in \cP(\R)$, and $\mathscr{L}_n^{(2)} := \frac{1}{n-k_n} \sum_{j=k_n+1}^n\delta_{\zeta_j} \in \cP(\R)$; 
\item for $n \in \N$, let $\mathscr{C}_n^{(i)} := \int \mathbf{c}  \mathscr{L}_n^{(i)} = \int_{\R} \mathbf{c}(x) \mathscr{L}_n^{(i)}(dx)$, for $i = 1, 2$. 
 \end{itemize}
First, { let $\lambda\in(0,1)$}, consider $(\mathscr{L}_n, \mathscr{C}_n) = (\mathscr{L}_n^{(1)}, \mathscr{C}_n^{(1)})$.  Then  
the domain $\mathcal{D}$ of \eqref{eq-dom} takes the form   $\mathcal{D} = \{\alpha \in \R : \log \E[ e^{\lambda^{-1} \alpha\, \zeta_1^2}] < \infty\} = (-\infty, \frac{\lambda}{2})$,  and  so $0 \in \mathcal{D}^\circ = \mathcal{D}$. 
Thus, $F(x) = \sup_{\alpha <\lambda/2} \alpha x$, which is equal to $\lambda x/2$ if $x \geq 0$, and equal to 
$\infty$ otherwise, and $\int {\mathbf{c}} d \mu = M_2 (\mu)$. 
Thus, by Corollary \ref{cor-iid},   the sequence 
 $\{\mathscr{L}_n^{(1)}, \mathscr{C}_n^{(1)}\}_{n\in\N}$ 
satisfies an LDP {at speed $n$ and} with GRF $J^{(1)} = J^{(1)}_\lambda$, defined by 
\begin{align*}
	J^{(1)}(\mu,t) &:= \left\{ \begin{array}{ll} \lambda H(\mu | \gamma_1) + \tfrac{\lambda}{2} [ t -  M_2(\mu) ] & \text{ if }  M_2(\mu) \le t;	\\
	+\infty & \text{ else. }
 \end{array}\right.
\end{align*}
Similarly, another application of Corollary \ref{cor-iid} shows that the sequence 
$\{\mathscr{L}_n^{(2)}, \mathscr{C}_n^{(2)}\}_{n\in\N}$ satisfies an LDP at speed $n$ with GRF $J^{(2)} = J^{(2)}_\lambda$, defined by
\begin{align*}
 	J^{(2)}(\nu,u) &:= \left\{ \begin{array}{ll} (1-\lambda) H(\nu | \gamma_1) + \tfrac{1-\lambda}{2} [ u -  M_2(\nu) ] & \text{ if }  M_2(\nu) \le u;	\\
	+\infty & \text{ else. }
 \end{array}\right.
\end{align*}
Due to the independence of $\zeta_j, j\in\N,$ and the contraction principle {applied to the mapping $\cP(\R)\times\R\times\cP(\R)\times\R\ni(\mu,t,\nu,u)\mapsto(\mu,\nu,\lambda t+ (1-\lambda)u)$}, the sequence 
\begin{equation}\label{eq-i2seqalt}
  \left\{\mathscr{L}_n^{(1)}, \, \mathscr{L}_n^{(2)}, \, \lambda\int {\mathbf{c}}\, d\mathscr{L}_n^{(1)} + (1-\lambda)\int {\mathbf{c}}\, d\mathscr{L}_n^{(2)}\right\}_{n\in\N}
\end{equation}
satisfies an LDP {at speed $n$} with GRF 
\begin{equation*}
  (\mu, \nu, s) \mapsto \inf_{t,u \in \R} \left\{ J^{(1)}(\mu,t) + J^{(2)}(\nu, u) : \lambda t + (1-\lambda) u = s\right\} = I_{1,\lambda}(\mu,\nu,s).
\end{equation*}
To complete the proof, note that because $k_n/n\ra\lambda$ deterministically, the sequences
\begin{equation*}
  \left\{\frac{\lambda}{k_n} \sum_{j=1}^{k_n} \zeta_j^2 + \frac{1-\lambda}{n-k_n}\sum_{j=k_n+1}^n \zeta_j^2\right\}_{n\in\N} \quad \mbox{ and } \quad \left\{\frac{1}{n} \sum_{j=1}^{n} \zeta_j^2\right\}_{n\in\N} 
\end{equation*}
are exponentially equivalent (recall Definition \ref{def-exp}).
Hence, the sequence in \eqref{eq-i2seqalt} is  also exponentially equivalent to \eqref{eq-i2seq},
and the latter thus satisfies an LDP at speed $n$ with GRF $I_{1,\lambda}$. 

{Now, let $\lambda=1$. We see from the derivation above that the sequence 
 $\{\mathscr{L}_n^{(1)}, \mathscr{C}_n^{(1)}\}_{n\in\N}$ 
satisfies an LDP at speed $n$  with GRF $J^{(1)}$, and that the two sequences $\{\frac{1}{k_n} \sum_{j=1}^{k_n} \zeta_j^2\}_{n\in\N}$ and $\{\frac{1}{n} \sum_{j=1}^{n} \zeta_j^2\}_{n\in\N}$ are exponentially equivalent. Hence, from Remark \ref{rem-expeq}, we obtain the desired LDP.}
\end{proof}

\begin{lemma}\label{lem-i2}
 { Suppose  $k_n/n \rightarrow \lambda \in (0,1]$ and let $\{\zeta_j\}_{j \in \N}$ be as in Lemma \ref{lem-i1}.  If $\lambda\in(0,1)$,}  the sequence	 of pairs of measures 
\begin{equation}\label{eq-pairi3}
 \left( \frac{1}{k_n} \sum_{j=1}^{k_n} \delta_{\sqrt{n}\zeta_j / \|\zeta^{(n)}\|_2 }\, , \quad\frac{1}{n-k_n}\sum_{j=k_n+1}^n \delta_{\sqrt{n}\zeta_j / \|\zeta^{(n)}\|_2}  \right) , \quad n\in\N,
\end{equation}
satisfies an LDP in $[\cP(\R)]^2$ at speed $n$ with GRF $I_{2,\lambda}$ defined by
\begin{align*}
I_{2,\lambda}(\mu, \nu) &:=  I_{1,\lambda}(\mu,\nu,1) \\
 &=  \lambda \, H(\mu | \gamma_1) + (1-\lambda) \, H(\nu| \gamma_1) + \tfrac{1}{2} \left(1 - \lambda \, M_2(\mu) - (1-\lambda)\,M_2(\nu)\right) ,
\end{align*}
if $\lambda\, M_2({\mu}) + (1-\lambda)\,M_2({\nu}) \le 1$, and $I_{2,\lambda}({\mu},{\nu}) := \infty$ otherwise. {On the other hand, if $\lambda = 1$, then the sequence 
$\{ \frac{1}{k_n} \sum_{j=1}^{k_n} \delta_{\sqrt{n}\zeta_j / \|\zeta^{(n)}\|_2 }\}_{n\in\N}$
satisfies an LDP in $\cP(\R)$ at speed $n$ with GRF $I_{2,1}$ defined by
\begin{align*}
I_{2,1}(\mu)=    H(\mu | \gamma_1) + \tfrac{1}{2} \left(1 -  M_2(\mu) \right) ,
\end{align*}
if $ M_2({\mu}) \le 1$, and $I_{2,1}({\mu},{\nu}) := +\infty$ otherwise.
}
\end{lemma}
\begin{proof}
  {Below we present the proof only for the case  $\lambda\in(0,1)$ since the
    case $\lambda=1$ can be argued in an exactly analogous fashion.} Due to Slutsky's theorem, the map 
\begin{equation*}
[\cP(R)]^2\times\R_+ \ni  (\bar{\mu},\bar{\nu},s) \mapsto \left(\bar{\mu}(\,\cdot \times s^{1/2}), \,\bar{\nu}(\,\cdot\times s^{1/2})\right),
\end{equation*}
is continuous. Then, applying the contraction principle with this map to the LDP of Lemma \ref{lem-i1}, we find that the sequence in \eqref{eq-pairi3} satisfies an LDP with GRF:
\begin{align*}
(\mu,\nu) \mapsto & \inf_{\bar{\mu},\bar{\nu}\in \cP(\R), s\in \R_+} \left\{I_{1,\lambda}(\bar{\mu},\bar{\nu},s) : \mu = \bar{\mu}(\,\cdot\times s^{1/2}) , \nu = \bar{\nu}(\,\cdot\times s^{1/2}) \right\}	 \\
   &= \inf_{s\in \R_+} \bigl\{  \lambda \, H( \mu(\,\cdot \times s^{-1/2}) | \gamma_1) + (1-\lambda) \, H(\nu(\,\cdot \times s^{-1/2})  | \gamma_1) \\
   &\quad\quad + \tfrac{1}{2} \left[s - \lambda s\,M_2(\mu) - (1-\lambda)s\,M_2(\nu)\right] \, :\,   s \ge \lambda s\,M_2(\mu) +(1-\lambda)s\,M_2(\nu) \bigr\} \\
   &= \inf_{s\in \R_+} \bigl\{  \lambda \, H( \mu | \gamma_1(\,\cdot \times s^{1/2})) + (1-\lambda) \, H(\nu | \gamma_1(\,\cdot \times s^{1/2}) ) \\
   &\quad\quad + \tfrac{s}{2} \left[1 - \lambda \,M_2(\mu) - (1-\lambda)\,M_2(\nu)\right] \, :\,   1 \ge \lambda \,M_2(\mu) +(1-\lambda)\, M_2(\nu) \bigr\}
\end{align*}
Assuming $  1 \ge \lambda \,M_2(\mu) +(1-\lambda)\, M_2(\nu) $,  noting that
\[ {\log} \left( \frac{d\gamma_1}{d\gamma_1(\cdot \times s^{1/2})}(x) \right)  =   - \frac{1}{2} \log s  - (1-s) \frac{x^2}{2}, 
\]
{and using the chain rule for relative entropy,   the right-hand side of the previous display}  is equal to 
\begin{align*}
& \quad \lambda \, H( \mu | \gamma_1)  + (1-\lambda) \, H(\nu | \gamma_1 )  \\
   &\quad\quad + \inf_{s\in \R_+} \bigl\{ -\lambda\left(\tfrac{1}{2}\log s + \tfrac{1-s}{2}M_2(\mu)\right) - (1-\lambda)\left(\tfrac{1}{2}\log s + \tfrac{1-s}{2} M_2(\nu)\right)\\
   &\quad\quad + \tfrac{s}{2} \left[1 - \lambda \,M_2(\mu) - (1-\lambda)\,M_2(\nu)\right]  \bigr\} \\
&= I_{2,\lambda}(\mu,\nu) -\tfrac{1}{2} + \tfrac{1}{2} \inf_{s\in \R_+} \left\{s-  \log s  \right\} \\
&= I_{2,\lambda}(\mu,\nu).
\end{align*}
{On the other hand, if $  1 < \lambda \,M_2(\mu) +(1-\lambda)\, M_2(\nu) $ is violated, then
  (by the convention that the infimum over an empty set is infinite)  is infinite, which is again equal to $I_{2,\lambda}(\mu, \nu)$.  This completes the proof. } 
\end{proof}

{ 
We now turn to the proof of Proposition \ref{lem-aux}.  First, for $\lambda \in (0,1]$, define 
\begin{equation}\label{j2}
  \mathbf{J}_{2,\lambda}(z):=
  \left\{
  \begin{array}{ll}
    \tfrac{\lambda}{2} \log\left(\tfrac{\lambda}{ z^2}\right)   +  \tfrac{1-\lambda}{2}\log\left(\tfrac{1-\lambda}{1- z^2}\right), &  0 < z < 1,  \\
    \infty,  & \mbox{ otherwise, }
  \end{array}
  \right. 
\end{equation}
where we use the convention $0 \log 0 = 0 \log (0/0) = 0$. 
}

\begin{proof}[Proof of Proposition \ref{lem-aux}]
We first prove that $\{\hat{\mu}^n_{\mathbf{A}}\}_{n\in\N}$ satisfies an LDP in 
${\mathcal P}(\R)$ with respect to the weak topology, with GRF $\mathcal{H}_\lambda$ of \eqref{eq-hrf}.  
Due to the representation {for $\mathbf{A}_{n,k_n}(1,\cdot)$ established in}  Lemma \ref{lem-row1},
it suffices to show that the sequence 
\begin{equation}\label{equivzeta}
\hat{\nu}_n:= \frac{1}{k_n} \sum_{j=1}^{k_n} \delta_{\sqrt{n}\zeta_j / \|\zeta^{(n)}\|_2 }, \quad n\in\N,
\end{equation}
satisfies an LDP in $\cP(\R)$ at speed $n$ with GRF $\mathcal{H}_\lambda$. 

{{We deduce in the following the case for $\lambda\in(0,1)$. The case for $\lambda=1$ can be shown using a similar calculation and is hence omitted.} Combining the LDP established in Lemma \ref{lem-i2} with the contraction principle applied to the
projection mapping: $(\cP(\R))^2 \ni (\pi_1, \pi_2) \mapsto \pi_1 \in \cP(\R)$, it follows that
$\{\hat{\nu}_n\}_{n \in \N}$ satisfies an LDP at speed $n$ with GRF given by $\mu\mapsto \inf_{\nu\in\cP(\R)} I_{2,\lambda}(\mu, \nu)$.} 
Now, note that for $\mu,\nu \in \cP(\R)$ such that $\lambda M_2(\mu) + (1-\lambda) M_2(\nu) \le 1$, we have 
\begin{align*}
I_{2,\lambda}(\mu,\nu) & = -\lambda h(\mu) + \tfrac{\lambda}{2} \log(2\pi) + \tfrac{\lambda}{2} M_2(\mu) \\
 & \quad\quad   -(1-\lambda) h(\nu) + \tfrac{1-\lambda}{2} \log(2\pi) + \tfrac{1-\lambda}{2} M_2(\nu) \\
 & \quad \quad + \tfrac{1}{2} (1- \lambda\,M_2({\mu}) - (1-\lambda)\,M_2({\nu}))\\
 & = -\lambda h(\mu ) -(1-\lambda) h(\nu)  + \tfrac{1}{2} \log(2\pi e),
\end{align*}
where $h$ is the entropy functional defined in \eqref{ent}.
Thus, 
\begin{align*}
 \inf_{\nu\in\cP(\R)} I_{2,\lambda}(\mu, \nu) &= -\lambda h(\mu) + \tfrac{1}{2}\log(2\pi e) + \inf_{{\nu}\in\mathcal{P}(\R)}\left \{  -(1-\lambda) h(\nu)  : (1-\lambda)\,M_2({\nu}) \le 1- \lambda \,M_2({\mu}) \right\} \\
&= -\lambda h(\mu) + \tfrac{1}{2}\log(2\pi e) + (1-\lambda) \inf_{\nu\in\mathcal{P}(\R)}\left\{   -h(\nu)  : M_2(\nu) \le \frac{1- \lambda\,M_2(\mu)}{1-\lambda} \right\}. 
\end{align*}
Since $M_2(\nu) \geq 0$, the right-hand side above is equal to  infinity if ${1 < \lambda M_2(\mu)}$.  On the other hand, if
$\lambda M_2(\mu) \leq 1$, then  
recalling that the maximum entropy probability measure under a second moment upper bound of $z$ is the Gaussian measure with mean zero and variance $z$ (see, e.g., Section 12   of \cite{CovTho01}), we have 
\begin{align*}
\inf_{{\nu}\in\cP(\R)} I_{2,\lambda}(\mu,\nu) &= -\lambda h(\mu) + \tfrac{1}{2}\log(2\pi e) - \tfrac{1-\lambda}{2}\log\left( 2\pi e  \tfrac{1- \lambda\,M_2({\mu})}{1-\lambda}  \right) \\ 
   &= -\lambda h(\mu)   + \tfrac{\lambda}{2}\log(2\pi e) + \tfrac{1-\lambda}{2}  \log\left(  \tfrac{1-\lambda}{1- \lambda\, M_2({\mu})}  \right) \\
   &= \mathcal{H}_\lambda({\mu}).
\end{align*}
Thus,  we have shown that the rate function is $\mathcal{H}_\lambda({\mu})$, as desired.  

{To strengthen the LDP on $\cP(\R)$ 
to an LDP on  $\cP_q(\R)$ (i.e., with respect to the $q$-Wasserstein topology),  via an appeal to 
\cite[Corollary 4.2.6]{DemZeiBook} it suffices to show that $\{\hat{\mu}^n_{\mathbf{A}}\}_{n\in\N}$ is exponentially tight in $\cP_q(\R)$. 
 Since for each  $j>0$, the set $K_{2,j}$  defined in \eqref{k2} is compact in the 
 $q$-Wasserstein topology by Lemma \ref{lem-compcat}.  By the definition of $K_{2,j}$, the definition of $M_2$  in \eqref{secmom},  
 and the identity {$\hat{\mu}^n_A \eqdist \hat{\nu}_n$ established in  \eqref{eq-muna} and Lemma \ref{lem-row1}}, and the
 definition of $\hat{\nu}_n$ in \eqref{equivzeta}, 
 \begin{equation}
   \label{eq-obs}
   \{ \hat{\mu}^n_{\mathbf{A}} \in K_{2,M}^c \} = \{ \hat{\nu}_n \in K_{2,M}^c\} = \{ M_2(\hat{\nu}_n) > M \}
=\left\{  \sqrt{\frac{n}{k_n}} \frac{\|\zeta^{(k_n)}\|_2}{ \|\zeta^{(n)}\|_2}  > \sqrt{M} \right\}. 
 \end{equation}
 When $k_n/n \rightarrow \lambda \in (0,1]$, by~\cite[Lemma 4.2]{AloGutProTha18}, 
$\{\|\zeta^{(k_n)}\|_2/\|\zeta^{(n)}\|_2\}_{n\in\N}$ satisfies an LDP at speed $n$ with GRF $\mathbf{J}_{2,\lambda}$ defined in \eqref{j2},
 and $\{\sqrt{M_2(\hat{\nu}_n})\}_{n\in\N}$ is exponentially equivalent to $\{\|\zeta^{(k_n)}\|_2/(\sqrt{\lambda}\|\zeta^{(n)}\|_2)\}_{n \in \N}$ at speed $n$
 by Remark \ref{rm-eq-conv}.     Combining this with \eqref{eq-obs} and  Remark \ref{rem-expeq},  when  $M=4/\lambda$,  we obtain
\begin{align*}
  \lim_{n\to\infty} \frac{1}{n}\log \mathbb{P}(\hat{\mu}^n_{\mathbf{A}} \in K_{2,M}^c)
  =\lim_{n\to\infty} \frac{1}{n} \log \P \left(\frac{\|\zeta^{(k_n)}\|_2}{\sqrt{\lambda}\|\zeta^{(n)}\|_2} > \sqrt{M} \right)  
\leq-\mathbf{J}_{2,\lambda}(2) 
=-\infty.
\end{align*}
This proves the desired exponential tightness of $\{\hat{\mu}^n_{\mathbf{A}}\}_{n\in\N}$ in $\cP_q(\R)$.
 }
\end{proof}

\section{Proofs of $q$-norm LDPs  in the sublinear regime, $1 \ll k_n \ll n$} 
\label{sec-sub-norm}
In this section we prove Theorems~\ref{sub-ldp} and~\ref{sub-mdp}.
{Recall from  \eqref{normY2} that $Y^n_{2,k_n} = \|\mathbf{A}_{n,k_n}^{\top}X^{(n)}\|_2$.}

\begin{proof}[Proof of Theorem~\ref{sub-ldp}]
Fix $\{k_n\}_{n\in\N}$ that grows sublinearly. From Lemma~\ref{lem-toprow} and \eqref{eq-eqd}, we have the following distributional identity,
\begin{equation}\label{n-scaling}
n^{-1/2}Y^n_{2,k_n} \buildrel (d) \over =\frac{\norm{\zeta^{(k_n)}}_2}{\norm{\zeta^{(n)}}_2}\frac{\norm{X^{(n)}}_2}{\sqrt{n}}
= \frac{\norm{\zeta^{(k_n)}}_2/\sqrt{k_n}}{\norm{\zeta^{(n)}}_2/\sqrt{n}}\frac{\sqrt{k_n}\norm{X^{(n)}}_2}{n}.
\end{equation}

Suppose Assumption~\ref{ass-normldp}* holds, in which case  $\{\norm{X^{(n)}}_2/\sqrt{n}\}_{n \in \N}$ satisfies an LDP at speed $n$ with GRF $J_X$.  On the other hand,  by~\cite[Lemma 4.2]{AloGutProTha18},  $\{\|\zeta^{(k_n)}\|_2/\|\zeta^{(n)}\|_2\}_{n\in\N}$ satisfies an LDP at speed $n$ with GRF 
\begin{equation}
J_{\zeta}(x) :=
\begin{cases}
-\frac{1}{2}\log (1-x^2),\quad &x\in[0,1),\\
+\infty,\quad & \text{otherwise}.
\end{cases}
\end{equation}
Due to the independence of $\zeta^{(n)}$ and $X^{(n)}$, the result follows by  Lemma~\ref{ldp-prod}.

Now, suppose Assumption~\ref{ass-sublinear} holds with sequence $\{s_n\}_{n\in\N}$, $r\in[0,\infty]$ and GRF $J_X^{(r)}$. We will make repeated use of the following simple observation: by Remark~\ref{lem-chi-2} and the contraction principle applied to the map $x \mapsto \sqrt{x}$, $\{\|\zeta^{(k_n)}\|_2/\sqrt{k_n}\}_{n\in\N}$ and $\{\|\zeta^{(n)}\|_2/\sqrt{n}\}_{n\in\N}$ satisfy LDPs at speed $k_n$ and $n$, respectively, with the same GRF $J_{\zeta^2}(x^2)$, 
and hence, $\{\sqrt{n}/\|\zeta^{(n)}\|_2\}_{n\in\N}$ also satisfies an LDP at speed $n$, 
 and all three sequences converge almost surely  to $1$ by the strong law of large numbers. We now consider three cases.
\begin{case}
Suppose $r=0$. Then $s_n\ll k_n\ll n$. By the case assumption, $\{\sqrt{k_n}\|X^{(n)}\|_2/n\}_{n\in\N}$
satisfies an LDP at speed $s_n $ with GRF $J_X^{(0)}$. 
Hence, the result in this case follows by~\eqref{n-scaling}, the observation above and {a dual application of Lemma~\ref{ldp-prod}  with $(V_n,W_n)$ therein first being $(\sqrt{k_n}\|X^{(n)}\|_2/n,\|\zeta^{(k_n)}\|_2/\sqrt{k_n})$ and then $(\|\zeta^{(k_n)}\|_2\|X^{(n)}\|_2/n,\sqrt{n}/\|\zeta^{(n)}\|_2)$.}
\end{case}

\begin{case}
Suppose $ r\in(0,\infty)$. By Remark~\ref{req1c}, $\{\sqrt{k_n}\|X^{(n)}\|_2/n\}_{n\in\N}$ satisfies an LDP at speed $k_n$ and with GRF $rJ_X^{(r)}(\sqrt{r}x)$. Then~\eqref{n-scaling}, the independence of $X^{(n)}$ and $\zeta^{(n)}$,  the observation above and Lemma~\ref{ldp-prod} together show that $\{n^{-1/2}Y^n_{2,k_n}\}_{n\in\N}$ satisfies an LDP at speed $k_n$ with GRF 
\[
\mathbb{J}_2^{\mathsf{sub}}(x):= \inf_{c>0}\left\{\frac{c^2-1}{2}-\log c+rJ_X^{(r)}\left(\frac{\sqrt{r}x}{c}\right) \right\}.
\]
\end{case}
\begin{case}
Suppose $r=\infty$. Again from the reformulation~\eqref{n-scaling}, we have
\begin{equation}\label{scale-reform}
n^{-1/2}Y^n_{2,k_n} \buildrel (d) \over =\frac{\norm{\zeta^{(k_n)}}_2/\sqrt{s_n}}{\norm{\zeta^{(n)}}_2/\sqrt{n}}\frac{\sqrt{s_n}\norm{X^{(n)}}_2}{n}.
\end{equation}
Since
\[
\frac{\norm{\zeta^{(k_n)}}_2}{\sqrt{s_n}}=\left(\frac{1}{s_n}\sum_{i=1}^{k_n}\zeta_i^2 \right)^{1/2},
\]
and (because $r = \infty$) $k_n/s_n\to0$ as $n\to\infty$, by Lemma~\ref{LDP-Y2} (with $p=2$) and the contraction principle applied to $t\mapsto \sqrt{t}$, $\{\|\zeta^{(k_n)}\|_2/\sqrt{s_n}\}_{n\in\N}$ satisfies an LDP at speed $s_n$ with GRF $\mathcal{J}_{\zeta}(t)=t^2/2$ if $t\geq0$, and $\mathcal{J}_{\zeta}(t)= \infty$ otherwise.  By the case assumption, $\{\sqrt{s_n}\|X^{(n)}\|_2/n\}_{n\in\N}$ satisfies an LDP at speed $s_n$ with GRF $J_X^{(\infty)}$.  The independence of $\{\zeta_i\}_{i\in\mathbb{N}}$ and $X^{(n)}$, together with the contraction principle applied to the product mapping $(x,y) \mapsto xy$, then implies that $\{V_n:=\|\zeta^{(k_n)}\|_2\|X^{(n)}\|_2/n\}_{n\in\N}$ satisfies an LDP at speed $s_n$  with GRF $\inf_{c>0}\{c^2/2+J_X^{\infty}(x/c)\}$.
The lemma follows on applying  Lemma~\ref{ldp-prod} with $\{V_n\}_{n\in\N}$ defined above,  $W_n:=(\|\zeta^{(n)}\|_2/\sqrt{n})^{-1}$, $n \in \N$, and $m = 1$. 
\end{case}
\end{proof}

\begin{proof}[Proof of Theorem~\ref{sub-mdp}]
Fix $q\in[1,2]$ and suppose $k_n$ grows sublinearly. From Lemma~\ref{lem-toprow} and \eqref{eq-eqd}, we have the following reformulation
\begin{equation}\label{kn-scaling}
k_n^{-1/q}Y^n_{q,k_n} \buildrel (d) \over = \frac{\norm{\zeta^{(k_n)}}_q/k_n^{1/q}}{\norm{\zeta^{(n)}}_2/\sqrt{n}}\frac{\norm{X^{(n)}}_2}{\sqrt{n}},
\end{equation}
where $\zeta^{(n)}:=(\zeta_1,\ldots,\zeta_n)$ with $\{\zeta_i\}_{i\in\mathbb{N}}$ being i.i.d. $\mathcal{N}(0,1)$ random variables.
Consider the following sequence of random vectors
\[
R_n:=\left(\frac{\norm{\zeta^{(k_n)}}_q}{k_n^{1/q}},\frac{\norm{\zeta^{(n)}}_2}{\sqrt{n}},\frac{\norm{X^{(n)}}_2}{\sqrt{n}}\right),\quad n\in\N.
\]
By Assumption~\ref{ass-normldp}, $\{\|X^{(n)}\|_2/\sqrt{n}\}_{n\in\N}$ satisfies an LDP at speed $s_n$ with GRF $J_X$. By Cram\'er's theorem and the contraction principle, $\{\|\zeta^{(k_n)}\|_q/k_n^{1/q}\}_{n\in\N}$ satisfies an LDP at speed $k_n$ with
GRF $J_{\zeta,q}(x):= \Lambda^*_q(x^q)$ if $x\geq 0$, with $\Lambda_q$ defined in~\eqref{lambda_q},
and $J_{\zeta,q}(x) = \infty$ if $x < 0$. Note that $\Lambda^*_q$ is strictly convex with unique minimizer $\mathcal{M}_q^{1/q}$. Similarly, $\{\|\zeta^{(n)}\|_2/\sqrt{n}\}_{n\in\N}$ satisfies an LDP at speed $n$ with GRF $J_{\zeta,2}(x):= \Lambda^*_2(x^2)$ if $x\geq 0$ (and is equal to  infinity otherwise) with unique minimizer $1$.

To prove the theorem, we will use Lemma~\ref{ldp-prod} to show that in each of the three regimes, $\{R_n\}_{n\in\N}$ satisfies an LDP at speed $s_n\wedge k_n$ with a GRF {$J_R$ that we identify in each case.}  In view of \eqref{kn-scaling} {and the contraction principle}, this would imply that $\{k_n^{-1/q}Y^n_{q,k_n}\}_{n\in\mathbb{N}}$ satisfies an LDP at speed $s_n\wedge k_n$ with GRF 
\begin{equation}\label{rate-R}
\widehat{\mathbb{J}}_q^{\mathsf{sub}}(u):=\inf\left\{J_R(x,y,z):u = \frac{xz}{y},x,y,z>0 \right\}.
\end{equation}

\begin{case}
  Suppose $k_n\ll s_n$.  Then we  also have $k_n\ll n$, and by Lemma~\ref{ldp-prod}
  {and the additional assumption that $J_X$ has a unique minimizer $m$}, $\{R_n\}_{n\in\N}$ is exponentially equivalent at speed $k_n$ to 
  $(\|\zeta^{(k_n)}\|_q/k_n^{1/q},1,m)$,
which  satisfies an LDP at speed $k_n$ with GRF $J_R(x,y,z)=J_{\zeta,q}(x)$ 
when $y=1$ and $z=m$, and $J_R(x,y,z)=+\infty$ otherwise. {Then \eqref{rate-R} shows that
$\widehat{\mathbb{J}}_q^{\mathsf{sub}}(u) = J_{\zeta,q} (u/m) = \Lambda_q^*(u^q/m^q)$ for $u \geq 0$, and is equal to positive infinity otherwise.}  
\end{case}
\begin{case}
Suppose $s_n=k_n$. {Since $k_n \ll n$,} again invoking Lemma~\ref{ldp-prod} {and the additional assumption that $J_X$ has a unique minimizer $m$,} we have the exponential equivalence at speed $k_n$ between 
$\{R_n\}_{n\in\N}$ and $\{(\|\zeta^{(k_n)}\|_q/k_n^{1/q},1,\|X^{(n)}\|_2/\sqrt{n})\}_{n\in\N}$. Hence, $\{R_n\}_{n\in\N}$ satisfies an LDP at speed $k_n$ with GRF $J_R(x,y,z)=J_{\zeta,q}(x)+J_X(z)$ when $y=1$, and $J_R(x,y,z)=+\infty$ otherwise.
{Moreover, \eqref{rate-R} shows that $\widehat{\mathbb{J}}_q^{\mathsf{sub}}(u) = \inf_{x > 0} \{\Lambda^*_q(x^q) + J_X(u/x)\}$ if  $u > 0$, (and is equal to $+ \infty$ otherwise).} 
\end{case}
\begin{case}
  Suppose $s_n\ll k_n$. Then  {$k_n \ll n$ implies $s_n \ll n$, and} so Lemma~\ref{ldp-prod} {and
  the observation that the GRF $J_{\zeta,q}$ has $M_q^{1/q}$ as its unique minimizer} imply that $\{R_n\}_{n\in\N}$ is exponentially equivalent at speed $s_n$ to the sequence
$
\{(\mathcal{M}_q^{1/q},1,\|X^{(n)}\|_2/\sqrt{n})\}_{n\in\N},
$
and therefore satisfies an LDP at speed $s_n$ with GRF $J_R(x,y,z)=J_{X}(z)$ when $x=\mathcal{M}_q^{1/q}$ and $y=1$, and $J_R(x,y,z)=+\infty$ otherwise. {When combined with \eqref{rate-R} this implies that $\widehat{\mathbb{J}}_q^{\mathsf{sub}}(u) = J_X(u/\mathcal{M}_q^{1/q})$ if $u \geq 0$ (and is equal to $+\infty$ otherwise).} 
\end{case}
\end{proof}

\section{Proofs of $q$-norm LDPs  in the linear regime} 
\label{sec-norm}

The goal of this section is to prove Theorem \ref{th-normlinear}.
Throughout, fix $\lambda \in (0,1]$, and assume $k_n \sim \lambda n$.
  Also, for  $q \in [1,2]$ and  $n,k \in \N, k \leq n$, recall the definition 
  $n^{-1/q}Y^n_{q,k_n} = n^{- 1/q}\|\mathbf{A}_{n,k_n}^T X^{(n)}\|_{q}$ given  in \eqref{normY2}.
Section \ref{subs-qnorm} contains a simple proof that is valid when $q\in[1,2)$.
 Section \ref{subs-2norm} is devoted to 
 the more involved case of  $q=2$ in the linear regime, which also then provides
 an alternative proof and alternative form of the rate function in the case $q\in[1,2)$. 
  
\subsection{The case $q\in[1,2)$}
\label{subs-qnorm}

\begin{proof}[Proof of Theorem \ref{th-normlinear} when $q \in [1,2)$] 
    Fix  $q \in [1,2)$, and 
      observe that with  $M_q$,  $L^n$, $Y_{q,k_n}^n$ and $\mathbf{A}_{n,k_n}$  defined as in \eqref{secmom}, \eqref{eq-empproj}, \eqref{normY2}
      and Section  \ref{sec-mainres}, respectively, it follows that  
      $$\left(\frac{k_n}{n} M_q(L^n)\right)^{1/q} = n^{-1/q} \|\mathbf{A}_{n,k_n}^T X^{(n)}\|_q = n^{-1/q}Y^n_{q,k_n}.$$
      {Then, the LDP for $\{L^n\}_{n \in \N}$ in 
      Theorem \ref{th-simn} and the contraction principle applied to the continuous
        map $\cP_q(\R) \ni \nu \mapsto (\lambda M_q)^{1/q}(\nu) \in \R$ 
       imply that 
      $\{\lambda M_q(\ln)\}_{n\in\N}$
      satisfies an LDP with GRF   
\begin{eqnarray*}
  \mathbb{J}_{q,\lambda}^{\mathsf{lin}}(x) & := &  \inf_{\mu \in \cP(\R)} \{  \mathbb{I}_{L,\lambda}(\mu): (\lambda M_q(\mu))^{1/q} = x \}.
  \end{eqnarray*}
Since $\tfrac{k_n}{n}M_q(\ln)$ is exponentially equivalent to $\lambda M_q(\ln)$ at speed $n$ by Remark~\ref{rm-eq-conv},
this implies (by Remark \ref{rem-expeq}) that
$\{n^{-1/q}Y^n_{q,k_n}\}_{n\in\N}$ also satisfies an LDP at speed $n$ with GRF $\mathbb{J}_{q,\lambda}^{\mathsf{lin}}$.}

      If  $x < 0$, we infimize over an empty set, and so
     $\mathbb{J}_{q,\lambda}^{\mathsf{lin}}(x)=-\infty$.
Now, if $s_n=n$, {using the expression for $\mathbb{I}_{L, \lambda}$ given in Remark \ref{rem-linsnn}}, 
we see that for $x \geq 0$, 
  \begin{eqnarray*} 
 \mathbb{J}_{q,\lambda}^{\mathsf{lin}}(x)& = & 
  \inf_{c \in \R_+, \, \mu\in\cP(\R)}  \left\{ \mathcal{H}_\lambda(\mu (\cdot \times c)) + J_X(c) : \, [\lambda M_q(\mu)]^{1/q} = x \right\},\\
  & = & 
  \inf_{c \in \R_+, \, \nu\in\cP(\R)}  \left\{ \mathcal{H}_\lambda(\nu) + J_X(c) : \, [\lambda M_q(\nu(\cdot \times c^{-1}))]^{1/q} = x \right\},\\
  & = &  \inf_{c \in \R_+, \, \nu\in\cP(\R)}   \left\{ \mathcal{H}_\lambda(\nu) + J_X(c) : \, [\lambda M_q(\nu)]^{1/q} = \frac{x}{c} \right\},
\end{eqnarray*}
  which coincides with \eqref{normgrf1}.

On the other hand, if $s_n \ll n$, using the identity $\mathbb{I}_{L, \lambda}(\mu) = J_X(c)$ if $\mu = \gamma_c$
  (and infinity, otherwise) established in Theorem \ref{th-simn},   we have for $x \geq 0$,
\begin{align*}
\mathbb{J}_{q,\lambda}^{\mathsf{lin}}(x) & = \inf_{c\in\R_+}\{J_X(c):(\lambda M_q(\gamma_c))^{1/q} = x \}\\
&= \inf_{c\in\R_+}\{J_X(c):(\lambda c^{q}\mathcal{M}_q)^{1/q} = x \}\\
&=J_X\left(\frac{x}{(\lambda \mathcal{M}_q)^{1/q}}  \right),
\end{align*}
where $\mathcal{M}_q$ is as defined in~\eqref{q-abs-moment}.  This proves   \eqref{normgrf1}. 
\end{proof}

\subsection{The case $q = 2$ and alternative proof for $q\in[1,2)$}
  \label{subs-2norm}

  We now provide an alternative proof of Theorem \ref{th-normlinear} for $q \in [1, 2)$, which also extends to the
  case $q = 2$.  
   This yields an alternative representation for the rate function 
  $\mathbb{J}_{q,\lambda}^{\mathsf{lin}}$ of \eqref{normgrf1}. 
   To introduce this representation, fix $q \in [1,2]$ and  define the following functions. 
Let \begin{equation}\label{amgf}
  \Lambda_{\mathsf{A},q}(t_1, t_2) := \log \int_{\R}\frac{1}{\sqrt{2\pi}} \exp\left( t_1|x|^q + (t_2 -\tfrac{1}{2}) x^2 \right)  dx, \quad (t_1,t_2) \in \R^2.  
  \end{equation}
Note that for $q \in [1, 2)$, we have $\Lambda_{\mathsf{A},q}(t_1, t_2) < \infty$ when $t_1\in \R$ and $t_2 < \frac{1}{2}$. On the other hand, for $q =2$, we have $\Lambda_{\mathsf{A},q}(t_1,t_2) < \infty$ when $t_1 + t_2 < \tfrac{1}{2}$. We also define\begin{equation}\label{bmgf}
 \Lambda_{\mathsf{B}}(t_3) := \log \int_{\R}\frac{1}{\sqrt{2\pi}} \exp\left(  (t_3 -\tfrac{1}{2}) x^2 \right)  dx, \quad t_3 \in \R.
\end{equation}
  Note that $\Lambda_{\mathsf{B}}(t_3) < \infty$ for $t_3 < \frac{1}{2}$.  Let $\Lambda_{\mathsf{A},q}^*$ and $\Lambda_{\mathsf{B}}^*$ denote the Legendre-Fenchel transforms  of $\Lambda_{\mathsf{A},q}$ and $\Lambda_{\mathsf{B}}$, respectively. For $q\in [1, 2)$ and
    $\lambda \in  (0,1)$, define
\begin{equation}\label{jq} \mathbf{J}_{q,\lambda}(z):=  \inf_{(x_1,x_2,x_3) \in \R^3} \left\{  \lambda \Lambda_{\mathsf{A},q}^*\left(\frac{(x_1,x_2)}{\lambda}\right) + (1-\lambda)\Lambda_{\mathsf{B}}^*\left(\frac{x_3}{1-\lambda}\right )  \,: \, z = \frac{x_1^{1/q}}{(x_2+ x_3)^{1/2}} \right\}, 
\end{equation}
{and define for $\lambda=1$, 
\begin{equation}\label{jq-1} \mathbf{J}_{q,1}(z):=  \inf_{(x_1,x_2) \in \R^3} \left\{   \Lambda_{\mathsf{A},q}^*\left((x_1,x_2)\right)   \,: \, z = \frac{x_1^{1/q}}{(x_2)^{1/2}} \right\}.  
\end{equation} 
Also, recall the definition of $\mathbf{J}_{2,\lambda}$, $\lambda \in (0,1],$ from \eqref{j2}. }

We then have the following result. 

\begin{proposition}\label{prop-norm}
  Fix $q\in [1,2]$. Suppose $\{k_n\}_{n\in\N}$ grows linearly with rate $\lambda \in (0,1]$,
and that Assumption \ref{ass-normldp} holds with associated GRF $J_X$.
Then, the sequence of scaled $\ell_q^n$ norms of random projections, $\{n^{-1/q}Y^n_{q,k_n} = n^{-1/q}\|\mathbf{A}_{n,k_n}^T X^{(n)}\|_q\}_{n\in\N}$, satisfies an LDP at speed $n$ with GRF 
\begin{equation}\label{eq-iaq} 
\bar{\mathbb{J}}_{q,\lambda}^{\mathsf{lin}}(x) := \inf_{y,z\in \R }  \left\{\mathbf{J}_{q,\lambda}(z) + J_X(y) \, : \, x = yz \right\} , \quad x\in \R_+.
\end{equation}
\end{proposition}

The proof of Proposition \ref{prop-norm} relies on an auxiliary result stated in Lemma \ref{lem-zzs}, which concerns an LDP related to the top row of the matrix $\mathbf{A}_{n,k_n}$. As in Section \ref{sec-anpr}, let $\zeta_1,\zeta_2,\dots$ denote a sequence of i.i.d.\ standard Gaussian random  variables, independent of $X^{(n)}$, and let $\zeta^{(n)}:= (\zeta_1,\dots, \zeta_n) \in \R^n$. Due to Lemmas \ref{lem-toprow} and \ref{lem-row1}, we have 
\begin{equation*}
\mathbf{A}_{n,k_n}^T X^{(n)} = \mathbf{A}_{n,k_n}^T \tfrac{X^{(n)}}{\|X^{(n)}\|_2}\, \|X^{(n)}\|_2 \eqdist   \mathbf{A}_{n,k_n}^T e_1 \, \|X^{(n)}\|_2 \eqdist \frac{(\zeta_1,\dots,\zeta_{k_n})}{\|\zeta^{(n)}\|_2 } \|X^{(n)}\|_2 \in \R^{k_n}. 
\end{equation*}
Therefore, for $n\in \N$ and all $q\in [1,\infty]$, we have
\begin{equation}\label{eq-dist} n^{-1/q} Y_{q,k_n}^n =  n^{-1/q}\|\mathbf{A}_{n,k_n}^T X^{(n)}\|_q \eqdist n^{1/2-1/q}
\frac{\|\zeta^{(k_n)}\|_q}{\|\zeta^{(n)}\|_2}  \frac{\|X^{(n)}\|_2}{\sqrt{n}}.   
\end{equation}

Given \eqref{eq-dist} and Assumption \ref{ass-normldp}, a natural step to proving Proposition \ref{prop-norm} is to establish an LDP for the sequence $\{n^{1/2-1/q} \|\zeta^{(k_n)}\|_q/\|\zeta^{(n)}\|_2\}_{n \in \N}$.

\begin{lemma}\label{lem-zzs}
Suppose $\{k_n\}_{n\in\N}$ grows linearly with rate $\lambda\in[0,1]$. For $q \in [1,2]$, the sequence 
\begin{equation}\label{eq-zk}  n^{1/2-1/q} \frac{\|\zeta^{(k_n)}\|_q}{\|\zeta^{(n)}\|_2} = {n^{1/2-1/q}}  \frac{\|(\zeta_1,\dots,\zeta_{k_n})\|_q}{\|(\zeta_1, \ldots, \zeta_n)\|_2 }, \quad  n\in \N,
\end{equation}
satisfies an LDP in $\R$ at speed $n$ with GRF $\mathbf{J}_{q,\lambda}$ of \eqref{jq} and \eqref{jq-1} for
$q \in [1,2)$ and \eqref{j2} for $q = 2$. 
\end{lemma} 
\begin{proof}
   Fix $q \in [1,2)$. For $n\in \N$. We start with the observation that the quantity in \eqref{eq-zk} 
  can be represented in terms of a continuous mapping of a vector of scaled i.i.d. sums:  for each
  $n \in \N$, 
  \begin{equation}
    \label{rep-Tmap}
    n^{1/2 - 1/q} \frac{\|\zeta^{(k_n)}\|_q}{\|\zeta^{(n)}\|_2} = T(Z_n),
      \quad
    Z_n := \left(\frac{1}{n} \sum_{i=1}^{k_n} \abs{\zeta_i}^q, \frac{1}{n} \sum_{i=1}^{k_n}\zeta_i^2, \,\, \frac{1}{n} \sum_{j=k_n+1}^{n} \zeta_j^2\right),
  \end{equation}
  where $T$ is the continuous map 
\begin{equation*} T:\R^3 \ni (x_1,x_2,x_3) \mapsto \frac{x_1^{1/q}}{(x_2 + x_3)^{1/2}} \in \R. 
\end{equation*}

{We now establish an LDP for the sequence $\{Z_n\}_{n\in\N}$. To this end, for fixed $\lambda \in (0,1)$, we first establish  LDPs of the related  sequences 
\begin{equation*}   \mathsf{A}_{n,q} := \frac{\lambda}{k_n} \sum_{i=1}^{k_n} (|\zeta_i|^q, \zeta_i^2), \quad \quad  \mathsf{B}_{n} := \frac{1-\lambda}{n-k_n} \sum_{j=k_n+1}^{n} \zeta_j^2, \quad n \in \N. 
\end{equation*}
Note that for $q\in [1, 2]$, the origin $(0,0)$ lies in the interior of the domain of $\Lambda_{\mathsf{A},q}$,
which is the logarithmic moment generating function of $(|\zeta_1|^q, \zeta_1^2)$. 
Since the $\{\zeta_i\}_{i \in \N}$ are i.i.d., by  Cram\'er's theorem \cite[Theorem 2.2.1]{DemZeiBook} 
the sequence $\{\mathsf{A}_{n,q}\}_{n\in \N}$ satisfies an LDP in $\R^2$ at
  speed $k_n$ with GRF $\Lambda_{\mathsf{A},q}^*(\cdot/\lambda)$, and hence (by  Remark \ref{rem-ldpscale} and the fact that
  $k_n/n \rightarrow \lambda$) also satisfies
an LDP in $\R^2$ at speed $n$ with GRF $\lambda \Lambda_{\mathsf{A},q}^*(\cdot/\lambda)$.  
Similarly, since $0$ belongs to the interior of the domain of $\Lambda_{\mathsf{B}}$, the logarithmic
moment generating function of $\zeta_1^2$, the sequence $\{\mathsf{B}_n\}_{n\in \N}$ satisfies an LDP in $\R$ at speed $n$  with GRF $(1-\lambda)\Lambda_{\mathsf{B}}^* (\cdot/(1-\lambda))$. 
Since  $k_n/n \rightarrow \lambda$ and $(n-k_n)/n \rightarrow 1 - \lambda$ deterministically, by Remarks \ref{rem-expeq} and \ref{rm-eq-conv},  $\{\frac{1}{n}\sum_{i=1}^{k_n} (|\zeta_i|^q, \zeta_i^2)\}_{n\in \N}$ and
$\{\frac{1}{n}  \sum_{j=k_n+1}^{n}\zeta_j^2\}_{n\in \N}$ satisfy LDPs at speed $n$ with the same GRFs as  
$\{\mathsf{A}_{n,q}\}_{n\in \N}$ and  $\{\mathsf{B}_n\}_{n \in \N}$, respectively. 
  Combined with the independence of $\mathsf{A}_{n,q}$ and $\mathsf{B}_n$,  and Lemma~\ref{ldp-prod},  it follows that the sequence $\{Z_n\}_{n \in \N}$ 
satisfies an LDP in $\R^3$ at speed $n$ with GRF
\begin{equation*}
(x_1,x_2,x_3) \mapsto \lambda \, \Lambda_{\mathsf{A},q}^*\left(\frac{(x_1,x_2)}{\lambda}\right) + (1-\lambda)\,\Lambda_{\mathsf{B}}^*\left(\frac{x_3}{1-\lambda}\right), \qquad x \in \R^3_+, 
\end{equation*}
and with the GRF equal to $-\infty$ otherwise.}
Finally, use \eqref{rep-Tmap}, the LDP for $\{Z_n\}_{n \in \N}$ and  the contraction principle for the map $T$ defined above
to show that the sequence in \eqref{eq-zk} satisfies an LDP with the  GRF $\mathbf{J}_{q,\lambda}$ of \eqref{jq}.

{Next, suppose $\lambda=1$. We see that the sequence $\{\mathsf{A}_{n,q}\}_{n\in\N}$ is exponentially equivalent to $\{(\frac{1}{n}\sum_{i=1}^{k_n}|\zeta_i|^q,\frac{1}{n} \sum_{i=1}^{n}\zeta^2)\}_{n\in\N}$ by Remark \ref{rm-eq-conv}. An application of the contraction principle to the map ($x_1,x_2)\mapsto x_1^{1/q}/x_2^{1/2}$ yields the LDP for the sequence in~\eqref{eq-zk} with speed $n$ and GRF~\eqref{jq-1}.
} 

{The case $q = 2$ was also  considered in \cite[Lemma 4.2]{AloGutProTha18}), but   
  can   be obtained via (a simpler form of) the  argument given  above.   To be self-contained,
  we provide some details. In this case, the quantity of interest is $\tilde{T} (\tilde{Z}_n)$, where 
  $\tilde{Z}_n$, is the $\R^2$-valued random vector that coincides with the   second and third components of $Z_n$,
  and $\tilde{T}$ maps  $(x_2,x_3) \mapsto x_2/(x_2+x_3)^{1/2}$, and the arguments above show 
  that $\{\tilde{Z}_n\}_{n \in \N}$ satisfies 
    an LDP at speed $n$ with GRF 
   $\lambda \Lambda_{\mathsf{B}}^*\left(\cdot/\lambda\right) + (1-\lambda)\Lambda_{\mathsf{B}}^*\left(\cdot/(1-\lambda)\right)$
  and thus,  by the contraction principle, $\{\|\zeta^{(k_n)}\|/\|\zeta^{(n)}\|\}_{n \in \N}$ satisfies an LDP with  GRF 
\begin{equation*}
  \mathbf{J}_{2,\lambda}(z) :=  \inf_{(x_2,x_3) \in \R^2} \left\{  \lambda \Lambda_{\mathsf{B}}^*\left(\frac{x_2}{\lambda}\right) + (1-\lambda)\Lambda_{\mathsf{B}}^*\left(\frac{x_3}{1-\lambda}\right )  \,: \, z^2 = \frac{x_2}{x_2+ x_3}  \right\}. 
\end{equation*}
Noting that $\Lambda_B$ is the log moment generating function of the chi-squared distribution, it follows from
Remark \ref{lem-chi-2} that $\Lambda_B^* (x) = J_{\zeta^2}(x) = \frac{1}{2} (x-1 - \log x)$, for $x>0$ and infinity, otherwise.  Substituting this into the expression
for $\mathbf{J}_{2,\lambda}(z)$ in the last display and solving the optimization problem  yields \eqref{j2}.} 
\end{proof}

\begin{proof}[Proof of Proposition \ref{prop-norm}] 
  Recall from Lemma \ref{lem-zzs} that the sequence $\{ n^{1/2 - 1/q} \|\zeta^{(k_n)}\|_q/\|\zeta^{(n)}\|_2\}_{n \in \N}$  
  satisfies an LDP with GRF $\mathbf{J}_{q,\lambda}$. In addition, Assumption \ref{ass-normldp} states that the sequence $\{\|X^{(n)}\|_2/\sqrt{n}\}_{n\in \N}$ satisfies an LDP with GRF $J_X$. Given the equality in distribution of \eqref{eq-dist}, the contraction principle applied to the continuous function $(y,z) \mapsto yz$ yields that  the sequence of scaled $\ell_q^n$ norms of random projections, $\{Y^n_{q,k_n}\}_{n\in\N}$, satisfies an LDP  at speed $n$ with GRF $\bar{\mathbb{J}}_{q,\lambda}^{\mathsf{lin}}$ of \eqref{eq-iaq}.
\end{proof}

\begin{remark}
Note that it follows from  Proposition \ref{prop-norm} and the proof of Theorem \ref{th-normlinear} for $q\in[1,2)$ in
Section \ref{subs-qnorm}  that
$\bar{\mathbb{J}}_{q,\lambda}^{\mathsf{lin}} = \mathbb{J}_{q,\lambda}^{\mathsf{lin}}$ for all $q\in[1,2)$.
\end{remark}

To complete the proof of Theorem \ref{th-normlinear} we show that this relation also holds for $q = 2$.

\begin{proof}[Proof of Theorem \ref{th-normlinear} in the case $q=2$]
First, consider the case  $s_n=n$.  {Fix $\lambda \in (0,1)$. The result for $\lambda=1$ follows on using the convention that $0\log 0 = 0 \log 0/0 =0$ everywhere in the derivation below.}  By  Proposition \ref{prop-norm}, it suffices to show that 
  $\bar{\mathbb{J}}_{2,\lambda}^{\mathsf{lin}} = \mathbb{J}_{2,\lambda}^{\mathsf{lin}}$. 
{In view of \eqref{normgrf1},  \eqref{eq-iaq},  and \eqref{j2},  
  it clearly suffices to show that for $z^2 \in (0, 1)$, 
\begin{equation}\label{eq-jform}
\mathbf{J}_{2,\lambda} (z) =   \frac{\lambda}{2} \log \left( \frac{\lambda}{z^2}\right) + \frac{1-\lambda}{2} \log \left( \frac{1-\lambda}{1-z^2}\right)
  = \inf_{\nu\in\cP(\R) : z^2 = \lambda M_2(\nu)} \mathcal{H}_\lambda(\nu).
\end{equation}
since when  $z^2 \notin (0,1)$ both
$\mathbf{J}_{2,\lambda} (z)$  and  the right-hand side above are infinite.  
For $z^2 \in (0,1)$, substituting the expression
 for $\mathcal{H}_\lambda$ from \eqref{eq-hrf}, we obtain 
\begin{align*} 
\inf_{\nu \in \cP(\R): z^2 = \lambda M_2(\nu)} \mathcal{H}_\lambda(\nu)   &= \inf_{\nu \in \cP(\R) : z^2 = \lambda M_2(\nu)} \left\{ -\lambda \, h(\nu) + \tfrac{\lambda}{2} \log (2\pi e) + \tfrac{1-\lambda}{2}\log\left(\tfrac{1-\lambda}{1-\lambda\, M_2(\nu)} \right) \right\} \\  &= -\sup_{\nu \in \cP(\R): z^2 = \lambda M_2(\nu)} \lambda \, h(\nu) + \tfrac{\lambda}{2} \log (2\pi e) + \tfrac{1-\lambda}{2}\log\left(\tfrac{1-\lambda}{1- z^2}\right) .
\end{align*}
Recall that the maximum entropy probability measure constrained to have second moment $z^2/\lambda$ is the Gaussian measure with mean zero and variance $z^2/\lambda $ (see, e.g., Ex. 12.2.1 of \cite{CovTho01}). Since the entropy of such a Gaussian equals $\frac{1}{2}\log (2\pi e z^2/\lambda)$, we have
\begin{align*}
\inf_{\nu \in \cP(\R): z^2 = \lambda M_2(\nu)} \mathcal{H}_\lambda(\nu) &= - \tfrac{\lambda}{2} \log(2\pi e z^2 / \lambda )	+\tfrac{\lambda}{2} \log (2\pi e)  +  \tfrac{1-\lambda}{2}\log\left(\tfrac{1-\lambda}{1- z^2}\right) = \mathbf{J}_{2,\lambda} (z),
\end{align*}
which proves \eqref{eq-jform}.
This completes the proof of  the equality $\bar{\mathbb{J}}_{2,\lambda}^{\mathsf{lin}} = \mathbb{J}_{2,\lambda}^{\mathsf{lin}} $ in this case.}

Now, suppose $s_n\ll n$. 
By~\eqref{eq-dist}, it follows that $n^{-1/2} Y^n_{2,k_n} \eqdist V_n W_n W'_n $, where $V_n = \|X^{(n)}\|_2/\sqrt{n}$, 
 $W_n := \sqrt{k_n/n}\|\zeta^{(k_n)}\|_2/\sqrt{k_n}$ and  $W'_n =  \sqrt{n}/\|\zeta^{(n)}\|_2$. 
Noting that $k_n/n \to \lambda$ as $n\to\infty$, 
we see that $\{W_n\}_{n \in \N}$ and $\{W'_n\}_{n \in \N}$ converge almost surely to $(\lambda {\mathcal M}_2)^{1/2}$ and $1$, respectively
(by the strong law of large numbers), and satisfy LDPs at speed $n$ ({by Remark~\ref{lem-chi-2}} and the contraction principle).   
Since by assumption, $\{V_n\}_{n\in\N}$ satisfies an LDP at speed $s_n\ll n$ with GRF $J_X$, by Lemma \ref{ldp-prod} the sequence $\{V'_n := V_nW_n'\}_{n\in\N}$ satisfies an LDP at speed $s_n$ with GRF $J_X$.
{Since the GRF of $\{W'_n\}_{n \in \N}$ has a unique minimum at $m := (\lambda {\mathcal M}_2)^{1/2}$, 
by  Lemma \ref{ldp-prod}  $\{V_n' W_n\}_{n\in\N}$ satisfies an LDP at speed $s_n$ with GRF $J_X (\cdot/m)$, which
coincides with the expression in \eqref{normgrf1}. }  
\end{proof}

\begin{remark}\label{rmk-technical}
  Recall that in the sublinear/linear regime, when $q =2$ the contraction principle cannot be applied because  
  the LDP of Theorem \ref{th-simn} holds with respect to the $q$-Wasserstein topology only for $q\in[1,2)$.
  Moreover, the LDP for $\{\hat{\mu}^n_{\mathbf{A}}\}_{n\in\N}$ of \eqref{eq-muna} as stated in Proposition \ref{lem-aux} also holds with respect to the  $q$-Wasserstein topology only for $q<2$. However, in the case $q=2$, \eqref{eq-jform} still establishes a variational problem that explicitly relates the rate function $\mathbf{J}_{2,\lambda}$ of \eqref{j2} to the rate function $\mathcal{H}_\lambda$ of \eqref{eq-hrf},  in the same manner as if the contraction principle were applicable. 
 We claim that this inconvenient gap is related  to a more fundamental obstacle. To illustrate this concretely, consider the case where $X_1,X_2,\dots$ are i.i.d.\ exponential random variables with mean 1. It is known that: 
\begin{enumerate}
\item Due to Sanov's theorem, the sequence of empirical measures $\{\frac{1}{n}\sum_{i=1}^n \delta_{X_i}\}_{n\in \N}$ satisfies an LDP in $\cP(\R)$ with GRF $H(\cdot\| \, \textnormal{Exp}(1))$, where we write $\text{Exp}(1)$ to denote the exponential distribution with mean 1.

\smallskip 

\item Due to Cram\'er's theorem \cite[Theorem 2.2.1]{DemZeiBook}, the sequence of empirical means $\{\frac{1}{n}\sum_{i=1}^n X_i\}_{n\in \N}$ satisfies an LDP in $\R$ with GRF	
\begin{equation*} L(\beta):= \beta - \log \beta  -1 .
\end{equation*}

\item An explicit calculation establishes the expression 
\begin{equation*} L(\beta) = \inf_{\mu  \in \cP(\R)} \left\{ H(\mu  \| \, \textnormal{Exp}(1)) : \int_{\R} x \, d\mu = \beta \right\}. 
\end{equation*}
\end{enumerate}
Note that if the map $\mu \mapsto \int_{\R} x\,d\mu$ were continuous, then the LDP of point 2.\ above would follow from point 1., point 3., and an application of the contraction principle. However, the map $\mu \mapsto \int_{\R} x\,d\mu$ is \emph{not} continuous with respect to the weak topology on probability measures. Moreover, the result of \cite[Theorem 1.1]{WanWanWu10} applied to the exponential distribution indicates that the LDP of point 1.\ \emph{does not} hold with respect to the $1$-Wasserstein topology. This suggests that the apparently cryptic  transition at $q = 2$ in the nature of the proof  of Theorem \ref{th-normlinear}
  and the result of Theorem \ref{th-simn} is in a sense a manifestation of a more common sticking point in large deviations theory.  In other words, even in the simple setting of i.i.d.\ random variables, the continuity required by the contraction principle fails to hold, but the consequences (i.e., a large deviation principle and a variational formula for the rate function) \emph{do} still hold in many instances.
\end{remark}

\vspace{20pt}
\noindent{\sc Acknowledgements:} 
The authors would like to thank Keith Ball for a question that led to Remark
\ref{rem-Keith}, and  thank Joscha Prochno for feedback on the first version
of this paper, which helped improve the exposition. This paper is comprised of some results of the PhD thesis of the first author~\cite{skim-thesis} under the supervision of the third author, namely    
the results on LDPs for the projection (or their empirical measures) in all regimes, and norms of the projections in the constant and linear regimes under Assumption A*.   All other results contained in the paper, 
including the formulation of Assumptions A, B and C, proofs of related theorems, as well as the verification of
these assumptions for $\ell_p^n$ balls, $p \in [1,2)$ and Orlicz balls,  will be a part of the second author's PhD thesis,
  under the supervision of the third author.
  We would also like to thank anonymous referees for many suggestions that greatly improved the exposition of the paper.  
\appendix
\section{An approximate contraction principle}
\label{app-a}

In this appendix, we recall the approximate contraction principle established in Section 6.2 of \cite{BenDemGui01},
and establish a corollary of it that we use in the proof of the LDP in the linear regime in Section
  \ref{sec-simn}. 
 Recall that given $a, b \in \R$, we will use $a \vee b$ and $a \wedge b$ to denote $\max(a,b)$ and $\min(a,b)$, respectively. 

Let $\Sigma$ be a Polish space. Let $\mathcal{X}$ be a separable Banach space with topological dual space $\mathcal{X}^*$, and let $\langle \cdot , \cdot \rangle:\mathcal{X}^*\times\mathcal{X} \ra \R$ denote the associated dual pairing.  Fix a continuous map $\mathbf{c}:\Sigma \ra \mathcal{X}$, and let $\{\mathscr{L}_n\}_{n\in\N}$ be a sequence of $\mathcal{P}(\Sigma)$-valued random elements. 
For $r \in  (0,\infty]$
   and a continuous function $W:\Sigma \ra \R$ such that $\P$-a.s., $\int_{\Sigma} (W (x)\vee 0) \mathscr{L}_n(dx) < \infty$ 
  for all $n \in \N$,   let 
\begin{equation}
  \label{def-lambdar}
  \Lambda_r(W) := \limsup_{n\ra\infty} \frac{1}{n} \log \E\left[  e^{n \left(\int_{\Sigma} W (x)\mathscr{L}_n (dx)\wedge r\right)} \right] ,
\end{equation}
and also let \begin{equation}
  \label{def-barlambda}
  \bar{\Lambda}(W) := \sup_{r > 0} \Lambda_r(W).
\end{equation}
We introduce the ``domain" of $\bar{\Lambda}$, defined in the following manner: let 
\begin{align}
\mathcal{D} &:= \{\alpha \in \mathcal{X}^* : \bar{\Lambda}(\langle \alpha, \mathbf{c}(\cdot)\rangle) <\infty\} \label{eq-dom},\\
\mathcal{D}_\circ &:= \{\alpha \in \mathcal{X}^* : \exists p > 1, p\alpha \in \mathcal{D} \}. \nonumber
\end{align}
Then, for $x \in \mathcal{X}$, let  
\begin{equation}\label{eq-f} F(x) := \sup_{\alpha \in \mathcal{D}_\circ} \langle \alpha, x \rangle.
\end{equation}
Lastly, for $n\in\N$, we define the $\mathcal{X}$-valued random variable $\mathscr{C}_n := \int_{\Sigma} \mathbf{c}(x
)\, \mathscr{L}_n(dx)$.

\begin{proposition}[Proposition 6.4 of \cite{BenDemGui01}] \label{prop-bdgrestate}
Suppose that:
\begin{enumerate}
	\item $\{\mathscr{L}_n\}_{n\in\N}$ satisfies an LDP in $\mathcal{P}(\Sigma)$ at speed $n$ with GRF $I_0$;
	\item $\{\mathscr{L}_n, \mathscr{C}_n\}_{n\in\N}$ satisfies an LDP in $\mathcal{P}(\Sigma)\times\mathcal{X}$ at speed $n$ with some convex GRF $\mathbb{I}$;
	\item for any sequence $\{W_n\}_{n\in\N}$, in the set
	\begin{equation} \label{wnset} \{ V + \langle \alpha, \mathbf{c}(\cdot)\rangle : V:\Sigma\ra\R \text{ continuous and bounded}, \alpha \in \mathcal{D}_\circ\}
\end{equation}
	 such that $W_n \downarrow W_\infty$  to a limit $W_\infty:\Sigma\ra \R$  that is continuous and bounded above, we have
\begin{equation}\label{eq-wcond} \limsup_{n\ra\infty} \bar\Lambda(W_n) \le \bar\Lambda(W_\infty). 
\end{equation}  
\end{enumerate}
Then,  we have the following representation for the GRF $\mathbb{I}$, for all $\mu\in\cP(\Sigma)$ and $s\in \mathcal{X}$:
\begin{equation}\label{eqdef-irf}   \mathbb{I}(\mu,s) :=  \left\{\begin{array}{ll}
I_0(\mu) + F\left( s - \int_\Sigma \mathbf{c} \,d\mu\right) & \textnormal{ if } I_0(\mu) < \infty,\\
+\infty & \textnormal{ else,}
 \end{array}\right.
\end{equation}
 with $F$ as defined in  \eqref{eq-f}. 
\end{proposition}

The following corollary considers a special case where the conditions of Proposition \ref{prop-bdgrestate} can be easily verified.

\begin{corollary}\label{cor-iid}
Let $\Sigma$ be a Polish space and ${\mathcal X}$ be a separable Banach space. 
Suppose that $\{k_n\}_{n\in\N}$ grows sublinearly with rate $\lambda\in(0,1]$,   each $\mathscr{L}_n$ is the empirical measure of 	$k_n$ i.i.d.\ $\Sigma$-valued random variables $\eta_1,\dots,\eta_{k_n}$ with common distribution $\mu$ (that does not depend on $n$), and for continuous $W:\Sigma \ra \R$, define 
\begin{equation}
\label{rel-tillam}
\widehat{\Lambda}(W):= \log \mathbb{E}[e^{\lambda^{-1} W(\eta_1)}]
\end{equation}
Also, let $\mathbf{c}: \Sigma \to {\mathcal X}$ be a continuous map such
that   $0$ lies in the interior $\mathscr{D}^\circ$ of the set 
 \begin{equation}
 \label{rel-calD}
  \mathscr{D} := \left\{\alpha\in \mathcal{X}^* : \widehat{\Lambda}(\langle \alpha, \mathbf{c}(\cdot) \rangle) < \infty
  \right\},  
 \end{equation}
and let $\mathscr{C}_n := \int_{\Sigma} \mathbf{c}(x)\mathscr{L}_n(dx)$.  Then $\{\mathscr{L}_n, \mathscr{C}_n\}$ satisfies an LDP with GRF
given by \eqref{eqdef-irf}, where ${I_0(\nu)} :=  \lambda H(\nu|\mu)$ and $F(x) = \sup_{ \alpha \in \mathscr{D}^\circ} \langle \alpha, x \rangle$.
\end{corollary}
\begin{proof}
We start by verifying the conditions of Proposition \ref{prop-bdgrestate}.
    
The fact that $\{\mathscr{L}_n\}_{n \in \N}$ satisfies an LDP in $\cP(\Sigma)$ at speed $k_n$ with GRF $H(\cdot|\mu)$ follows from Sanov's theorem on the  Polish space ${\mathcal P}(\Sigma)$ (see, e.g., Theorem 6.6.9 of \cite{KanLakbook01}). Since $k_n /n \ra \lambda$, this immediately implies that $\{\mathscr{L}_n\}_{n \in \N}$ satisfies an LDP at speed $n$ with GRF $I_0 (\cdot) := \lambda H(\cdot|\mu)$, so condition 1.\ of Proposition \ref{prop-bdgrestate} is satisfied.

Next, note that $({\mathscr L}_n, {\mathscr C}_n) = \frac{1}{k_n} \sum_{j=1}^{k_n} (\delta_{\eta_j}, \mathbf{c}(\eta_j)) \in {\mathcal P}(\Sigma) \times {\mathcal X}$.  Therefore, it follows from Cram\'er's theorem on any locally convex topological space (see, e.g., Theorem 6.1.3 and Corollary 6.16 of \cite{DemZeiBook}), with an appeal to the assumption that $0$ lies in  $\mathscr{D}^\circ$, the interior 
 of the set $\mathscr{D}$ of \eqref{rel-calD}, that
$\{({\mathscr L}_n, {\mathscr C}_n)\}_{n \in \N}$ satisfies an LDP in $\cP(\Sigma) \times \mathcal{X}$ at speed $k_n$ with a convex GRF. Since $k_n/n \rightarrow \lambda \in (0,1]$,  condition 2.\ of Proposition \ref{prop-bdgrestate} is  satisfied.

  As for  condition 3., first 
  consider any sequence $\{W_n\}_{n\in\N}$ such that  for each $n \in \N$, $W_n = V_n + \langle \alpha_n, \mathbf{c}(\cdot)\rangle$ for $V_n$ bounded and continuous,   
  and $\alpha_n \in \mathscr{D}$. 
   Due to the assumption that $\alpha_1 \in \mathscr{D}_\circ$ and the boundedness of $V_1$, we have $\widehat{\Lambda}(W_1) < \infty$, so  if $W_n \downarrow W_\infty$, then by the dominated convergence theorem,
 \begin{equation}\label{tillim}
\lim_{n\ra\infty} \widehat{\Lambda}(W_n) = \widehat{\Lambda}(W_\infty).	
 \end{equation}
     To complete the proof, it clearly suffices to show  that $\mathscr{D} = \mathcal{D}$, for the domains 
$\mathscr{D}$ and $\mathcal{D}$ defined in  \eqref{rel-calD} and \eqref{eq-dom}, respectively, 
      and that  relation \eqref{tillim} holds
      when $\widehat{\Lambda}$ is replaced with $\bar{\Lambda}$.  In turn, to show the latter, it suffices to prove that 
      $\bar{\Lambda} = \lambda \widehat{\Lambda}$.  Indeed, 
     note that for any continuous $W:\Sigma \ra \R$ such that $\P$-a.s., $\int_{\Sigma} (W (x)\vee 0) {\mathscr{L}_n(dx)} < \infty$ for all $n \in \N$,  
 by the definitions of $\Lambda_\infty$ and $\bar{\Lambda}$ from \eqref{def-lambdar} and \eqref{def-barlambda}, respectively, the i.i.d.\ assumption on  $\{\eta_i\}_{i=1}^{k_n}$ and the fact that $k_n/n \rightarrow \lambda$, we have 
\begin{align*}
  \widehat{\Lambda}(W) &= \lambda^{-1}\Lambda_\infty(W) & 
  \\
  &\ge   \lambda^{-1} \bar{\Lambda}(W) & 
    \\
 &\ge  \lambda^{-1}\lim_{ R\ra\infty} \bar{\Lambda}(W \wedge R)\\
 &=   \lambda^{-1}\lim_{R \ra\infty} \Lambda_\infty(W \wedge R)\\
 &=  \lim_{R\ra\infty} \widehat{\Lambda}(W \wedge R)\\
 &= \widehat{\Lambda}(W),
\end{align*}
 where the first inequality uses the elementary observation that $\Lambda_{r}\leq \Lambda_{\infty}$ for every $r$ implies 
$\bar{\Lambda} \leq \Lambda_\infty$, and the second equality uses this observation along with the fact that
the converse inequality also holds when both functions are evaluated at $W \wedge R$, since then   $\int_{\Sigma}(W(x) \wedge R){\mathscr{L}_n(dx)}  \leq R$ implies 
$\Lambda_{\infty} (W \wedge R) = \Lambda_R (W \wedge R) \leq \bar{\Lambda} (W \wedge R)$.  
Thus, we have shown $\widehat{\Lambda} =  \lambda^{-1}\Lambda_\infty$, which completes the verification of
condition 3.

The  corollary then  follows from Proposition \ref{prop-bdgrestate} and the identity $\mathscr{D}^\circ = \mathcal{D}_0$. 

 \end{proof}
 
\section{LDP for generalized normal random variables} \label{pf-LDP-Y2}

The  aim of this section is to prove Lemma~\ref{LDP-Y2}, which concerns 
  an LDP for  weighted sums of stretched exponential random variables.

\begin{proof}
Fix $p\in[1,2)$, $t>0$, and denote $\xi:=\xi_i^{(p)}$. Define $m:=\mathbb{E}[\xi_1^2]$ and for $n\in\N$ define
$S_n := \sum_{i=1}^{k_n}(\xi_i^2-m)$. Since $k_n/b_n\to 0$ as $n\to\infty$, by Theorem 3 in~\cite{Nagaev69}, with $n$, $x_n$ therein and $\varepsilon$ replaced by $k_n$, $b_nt$ and $1-p/2$, we see that
\[
\mathbb{P}\left(S_n>b_nt  \right)= k_n\mathbb{P}\left(\xi^2_1-m>b_nt \right)(1+o(1)).
\]
Hence, by~\eqref{est-Y2} and the convergence  $k_n/b_n\to 0$ as $n\to\infty$,
\begin{align}\label{est-p}
\lim_{n\to\infty}\frac{1}{b_n^{p/2}}\log\mathbb{P}\left(\frac{1}{b_n}S_n>t \right) 
= -\mathcal{J}_{\xi,p}(t).
\end{align}

For $p=2$, the sum $\sum_{i=1}^{k_n}\xi_i^2$ is distributed as a chi-squared distribution with $k_n$ {degrees} of freedom. Hence for $t>0$, we have the following tail probability estimate:
\begin{align*}
\mathbb{P}\left(S_n>b_nt\right) = \frac{1}{2^{k_n/2}\Gamma(k_n/2)}(b_nt+m)^{k_n/2-1}e^{-(b_nt+m)/2}.
\end{align*}
Since $k_n/b_n\to 0$ as $n\to\infty$, we again have
\begin{align}
&\lim_{n\to\infty}\frac{1}{b_n}\log \mathbb{P}\left(\frac{1}{b_n}S_n>t \right) \nonumber\\
&\quad=-\frac{t}{2}+\lim_{n\to\infty}\left[
\frac{k_n/2-1}{b_n}\log(b_nt+m) - \frac{k_n}{2b_n}\log 2-\frac{k_n/2+1/2}{b_n}\log (k_n/2) -\frac{k_n/2}{b_n}\right]\nonumber\\
&\quad=-\frac{t}{2}+\lim_{n\to\infty}\left[\frac{k_n}{2b_n}\log\frac{b_nt+m}{k_n}-\frac{\log(b_nt+m)}{b_n}-\frac{k_n}{2b_n}-\frac{\log k_n}{2b_n}\right]\nonumber\\
&\quad = -\frac{t}{2}\label{est-p2},
\end{align}
where the second equality follows from Stirling's approximation.

Let $T_n:=b_n^{-1}\sum_{i=1}^{k_n}\xi_i^2=S_n+mk_n/b_n$. We now combine the two estimates~\eqref{est-p} and~\eqref{est-p2} to show that 
$\{T_n\}_{n\in\N}$ satisfies an LDP at speed $b_n^{p/2}$ with GRF $\mathcal{J}_{\xi,p}$.   Fix a closed set $F\subset\R$. If $0\in F$, then $\inf_{x\in F}\mathcal{J}_{\xi,p}(x)=0$, and  the large deviation upper bound is automatic since $\mathbb{P}(T_n\in F)\leq 1$. Suppose $0\notin F$. Let $b :=\inf\{\beta>0:\beta\in F\}$. Then for $b>\tau>0$,
by the positivity of $T_n$,~\eqref{est-p} or~\eqref{est-p2} and the monotonicity of $\mathcal{J}_{\xi,p}$,
\begin{align*}
\limsup_{n\to\infty}\frac{1}{b_n^{p/2}}\log\mathbb{P}(T_n\in F)&\leq \limsup_{n\to\infty}\frac{1}{b_n^{p/2}}\log\mathbb{P}(T_n\in [b,\infty))\\
&\leq \limsup_{n\to\infty}\frac{1}{b_n^{p/2}}\log\mathbb{P}(T_n\in [b-\tau,\infty))\\
&=-\mathcal{J}_{\xi,p}(b-\tau).
\end{align*}
Letting $\tau\to0$ and appealing to the continuity and monotonicity of $\mathcal{J}_{\xi,p}$ we obtain $$\limsup_{n\to\infty}\frac{1}{b_n^{p/2}}\log\mathbb{P}(T_n\in F)=-\inf_{x\in F}\mathcal{J}_{\xi,p}(x),$$
which completes the proof of the upper bound. 

Next, fix an open set $U\subset\R$. If $0\in U$, then $\inf_{x\in U}\mathcal{J}_{\xi,p}(x)=0$. Since $U$ is open, there exists $\varepsilon >0$ such that $(-\varepsilon,\varepsilon)\subset U$. Since $k_n/b_n\to0$ as $n\to\infty$, the strong law of large numbers implies $\lim_{n\to\infty}\mathbb{P}(T_n\in(-\varepsilon,\varepsilon))=1$. Hence, the large deviation lower bound follows. Next suppose $0\notin U$.
 For $\delta>0$, there exists $\beta\in U$ such that $\mathcal{J}_{\xi,p}(\beta)< \inf_{x\in U}\mathcal{J}_{\xi,p}(x)+\delta$. Pick $\varepsilon >0$ such that $(\beta-\varepsilon,\beta+\varepsilon)\subset U$. If $\beta <0$, then the large deviation lower bound is trivial. Suppose $\beta>0$. Then for $\varepsilon>\tau>0$,
\begin{align*}
\liminf_{n\to\infty}\frac{1}{b_n^{p/2}}\log\mathbb{P}(T_n\in U)&\geq \liminf_{n\to\infty}\frac{1}{b_n^{p/2}}\log\mathbb{P}(S_n\in (\beta-\varepsilon,\beta+\varepsilon-\tau))\\
&= \liminf_{n\to\infty}\frac{1}{b_n^{p/2}}\log\left[\mathbb{P}(S_n\in (\beta-\varepsilon,\infty))-\mathbb{P}(S_n\in [\beta+\varepsilon-\tau,\infty))\right]\\
&=\liminf_{n\to\infty}\frac{1}{b_n^{p/2}}\log\mathbb{P}(S_n\in (\beta-\varepsilon,\infty))\\
&=-\mathcal{J}_{\xi,p}(\beta-\varepsilon)\\
&\geq-\inf_{x\in U}\mathcal{J}_{\xi,p}(x)-\delta,
\end{align*}
where the third equality follows by the monotonicity of $\mathcal{J}_{\xi,p}$ and the last inequality by the continuity of $\mathcal{J}_{\xi,p}$ on choosing $\varepsilon$ sufficiently small.
The large deviation lower bound then follows on sending $\delta$ to $0$.  This completes the proof of the lemma. 
\end{proof}

\section{Properties of the function $\mathcal{J}$ of~\eqref{J_uv}} 
\label{ap-orlicz}

{
We recall here the definition of the subdifferential of a convex function. Let $f:\R\to\R\cup \{\infty\}$ be a convex function. Then the subdifferential at a point $x\in\R^n$ is defined to be
\[
\partial f(x):=\{s\in\R^n:f(y)\geq f(x)+\langle s,y-x\rangle\quad\text{fo all}\quad y\in\R^n\}.
\]
Note that when $f$ is differentiable at $x$, the subdifferential $\partial f(x)$ has only one element, which
is  the gradient $\nabla f(x)$. For further discussion of subdifferentials, the reader is referred to Chapter X in~\cite{HirLem13}.
}

\begin{proof}[Proof of Lemma~\ref{lem-prop-J}]
Before proving the individual properties, we make a general observation.
 Recall that $D_V$ is the domain of $V$. For $u$, $v\in\R_+$, define $H:\R_-\times\R\to\R$ to be
 \begin{equation}\label{H_st}
 H(s,t):= su+tv-\log\int_{D_V} e^{sV(x)+tx^2}dx.
 \end{equation}
 By definition, $H$ is differentiable. We first show that $H$ is concave. For this, by \eqref{H_st} it suffices to show that $(s,t)\mapsto \log\int_{D_V} e^{sV(x)+tx^2}dx$ is convex. Indeed, by H\"older's inequality, for $\lambda\in(0,1)$ and $(s_1,t_1)$, $(s_2,t_2)\in\R_-\times\R$, we see that
 \begin{align*}
& \log\int_{D_V} e^{(s_1\lambda+s_2(1-\lambda))V(x)+(t_1\lambda+t_2(1-\lambda))x^2}dx\\
 &\qquad\leq \log\left(\left(\int_{D_V} e^{s_1\lambda V(x)+t_1\lambda x^2}dx\right)^{1/\lambda}\left(\int_{D_V} e^{s_2
 (1-\lambda) V(x)+t_2(1-\lambda) x^2}dx\right)^{1/(1-\lambda)}\right)\\
 &\qquad = \frac{1}{\lambda}\log\int_{D_V} e^{s_1\lambda V(x)+t_1\lambda x^2}dx+\frac{1}{1-\lambda}\log\int_{D_V} e^{s_2(1-\lambda) V(x)+t_2(1-\lambda) x^2}dx.
 \end{align*}
 In particular, $H$ is strictly concave if and only if $V\not\equiv 0$ in {the domain of $V$}.

We now turn to the proof of property 1.
Let {$(u,v)\in\R^2_+$ be} in the interior of the domain of $\mathcal{J}$.
Then by~\cite[Theorem 1.4.2]{HirLem13}, since $\mathcal{J}$ is convex, the subdifferential at $(u,v)$ is nonempty, $\partial\mathcal{J}(u,v)\neq\emptyset$. Therefore there exists $s\in\R_-$, $t\in\R$ such that $(s,t)\in\partial\mathcal{J}(u,v)$. Since $\mathcal{J}$ and $(s,t)\mapsto \log\int_{D_V} e^{sV(x)+tx^2}dx$ are both convex, both functions satisfy the condition in~\cite[Theorem 1.4.1]{HirLem13}. Hence, we see that there exist $s\in\R_-$ and $t\in\R$ such that $\mathcal{J}(u,v)=H(s,t).$
 If $H$ is strictly concave, then automatically $(s,t)$ is the unique maximizer in the supremum of $\mathcal{J}(u,v)$ in \eqref{J_uv}. If $H$ is not strictly concave, then 
 $V\equiv 0$ in its domain, and so  by \eqref{J_uv},
 ${\mathcal J}(v) = \sup_{s < 0, t \in \R_+} \left\{ s u  + t v - \log ( \int_{D_V} e^{tx^2} dx ) \right\}$,
from which it clearly follows that $(0,t)$ is the unique maximizer (since $v \geq 0$). 
Next, again by the last display and \cite[Theorem 1.4.1]{HirLem13}, we see that $(u,v)\in \partial \log\int_{D_V} e^{sV(x)+tx^2}dx$. By the definition of a subdifferential, we conclude that 
\begin{align}
&u \geq \frac{d}{ds} \log\left( \int_{D_V} e^{sV(x)+tx^2}dx\right)= \int_{D_V} V(x)\nu_{s,t}(dx) =M_V(\nu_{s,t}),\nonumber\\
&v = \frac{d}{dt} \log\left( \int_{D_V} e^{sV(x)+tx^2}dx\right)=\int_{D_V} x^2\nu_{s,t}(dx) =M_2(\nu_{s,t}),\label{inf-attain}
\end{align}
where \eqref{inf-attain}  holds due to the differentiability of $t\mapsto \log\int_{D_V} e^{sV(x)+tx^2}dx$. This proves property 1.

To see why property 2 holds, note that
by the duality of the Legendre transform~\cite[Equation (12)]{Zia09}, the minimizer of $\mathcal{J}(1,\cdot)$ is obtained at $m$ such that
\begin{align}\label{smin}
  \left.\frac{d}{dv}\mathcal{J}(1,v)\right|_{v=m}=t(m) = 0.
\end{align}
Substituting this relation back into~\eqref{inf-attain}, we obtain
$1 = M_V(\nu_{s(m),0})$ and $m=M_2(\nu_{s(m),0})$.  
Setting $b^* = -s(m)>0$
and observing that $\nu_{s,0} = \mu_{V,-s}$ for any $s$, 
we conclude that $1=M_V(\nu_{s(m),0})=M_V(\mu_{V,b^*})$ and 
$m=M_2(\nu_{s(m),0})=M_2(\mu_{V,b^*})$. 
Now,~\eqref{JV-min} follows since the supremum on the right-hand side of~\eqref{JV-min} is attained at $1 = M_V(\nu_{s(m),0})$ and~\eqref{inf-attain} is uniquely solvable.
This proves property 2. 

For the remaining properties, first note that
by the duality of the Legendre transform, the supremum in the definition~\eqref{J_uv} of $\mathcal{J}(u,v)$, when finite, is attained at $(\partial_u\mathcal{J}(u,v),\partial_v\mathcal{J}(u,v))\in\R_-\times\R$. This proves property 4. By the convexity of $v\mapsto\mathcal{J}(1,v)$ and the fact that
it uniquely attains its minimum at the value $m$ defined above, we see that $\partial_v\mathcal{J}(1,v)>0$ if $v>m$ and $\partial_v\mathcal{J}(1,v)<0$ if $0<v<m$.  This proves property 3, and completes the proof of the lemma.
\end{proof}

\bibliographystyle{siam}
\bibliography{multidim}

\begin{bibdiv}
\begin{biblist}

\bib{AloGutBas15}{book}{
      author={Alonso-Guti{\'e}rrez, D.},
      author={Bastero, J.},
       title={Approaching the {K}annan-{L}ov{\'a}sz-{S}imonovits and variance
  conjectures},
   publisher={Springer},
        date={2015},
      volume={2131},
}

\bib{AloPro20}{article}{
      author={{Alonso-Guti{\'e}rrez}, D.},
      author={{Prochno}, J.},
       title={{Thin-shell concentration for random vectors in {O}rlicz balls
  via moderate deviations and {G}ibbs measures}},
        date={2020},
     journal={arXiv e-prints},
       pages={arXiv:2011.07523},
      eprint={2011.07523},
}

\bib{AloGutProTha18}{article}{
      author={Alonso-Guti\'{e}rrez, D.},
      author={Prochno, J.},
      author={Th\"{a}le, C.},
       title={Large deviations for high-dimensional random projections of
  $\ell_p^n$-balls},
        date={2018},
     journal={Advances in Applied Mathematics},
      volume={99},
       pages={1\ndash 35},
}

\bib{AloGutProTha20}{article}{
      author={Alonso-Guti\'{e}rrez, D.},
      author={Prochno, J.},
      author={Th\"{a}le, C.},
       title={{Large deviations, moderate deviations, and the {KLS}
  conjecture}},
        date={2021},
     journal={Journal of Functional Analysis},
      volume={280},
      number={1},
       pages={108779},
}

\bib{AndJacSit98}{article}{
      author={Anderson, B.},
      author={Jackson, J.},
      author={Sitharam, M.},
       title={Descartes' rule of signs revisited},
        date={1998},
     journal={American Mathematical Monthly},
      volume={105},
      number={5},
       pages={447\ndash 451},
}

\bib{AntBalPer03}{article}{
      author={Anttila, M.},
      author={Ball, K.},
      author={Perissinaki, I.},
       title={The central limit problem for convex bodies},
        date={2003},
     journal={Transactions of the American Mathematical Society},
      volume={355},
      number={12},
       pages={4723\ndash 4735},
}

\bib{BahRao60}{article}{
      author={Bahadur, R.},
      author={Rao, R.s},
       title={On deviations of the sample mean},
        date={1960},
     journal={Annals of Mathematical Statistics},
      volume={31},
      number={4},
       pages={1015\ndash 1027},
}

\bib{BarGamLozRou10}{article}{
      author={Barthe, F.},
      author={Gamboa, F.},
      author={Lozada-Chang, L.},
      author={Rouault, A.},
       title={Generalized {D}irichlet distributions on the ball and moments},
        date={2010},
        ISSN={1980-0436},
     journal={ALEA. Latin American Journ fal of Probability and Mathematical
  Statistics},
      volume={7},
       pages={319\ndash 340},
}

\bib{BarGueMenNao05}{article}{
      author={Barthe, F.},
      author={Gu{\'e}don, O.},
      author={Mendelson, S.},
      author={Naor, A.},
       title={A probabilistic approach to the geometry of the $\ell^n_p$-ball},
        date={2005},
     journal={Annals of Probability},
      volume={33},
      number={2},
       pages={480\ndash 513},
}

\bib{BenDemGui01}{article}{
      author={Ben~Arous, G.},
      author={Dembo, A.},
      author={Guionnet, A.},
       title={Aging of spherical spin glasses},
        date={2001},
     journal={Probability Theory and Related Fields},
      volume={120},
      number={1},
       pages={1\ndash 67},
}

\bib{BreVoi00}{article}{
      author={Brehm, U.},
      author={Voigt, J.},
       title={Asymptotics of cross sections for convex bodies},
        date={2000},
     journal={Beitr\"age zur Algebra und Geometrie},
      volume={41},
      number={2},
       pages={437\ndash 454},
}

\bib{CovTho01}{book}{
      author={Cover, T.~M.},
      author={Thomas, J.~A.},
       title={Elements of information theory},
     edition={2},
   publisher={Wiley},
        date={2001},
}

\bib{Csiszar84}{article}{
      author={Csisz{\'a}r, I.},
       title={Sanov property, generalized {I}-projection and a conditional
  limit theorem},
        date={1984},
     journal={Annals of Probability},
       pages={768\ndash 793},
}

\bib{DemZeiBook}{book}{
      author={Dembo, A.},
      author={Zeitouni, O.},
       title={Large deviations techniques and applications},
     edition={2},
   publisher={Springer Science \& Business Media},
        date={2009},
      volume={38},
}

\bib{DiaFre84}{article}{
      author={Diaconis, P.},
      author={Freedman, D.},
       title={Asymptotics of graphical projection pursuit},
        date={1984},
     journal={Annals of Statistics},
      volume={12},
      number={3},
       pages={793\ndash 815},
}

\bib{DinZab92}{article}{
      author={Dinwoodie, I.~H.},
      author={Zabell, S.~L.},
       title={Large deviations for exchangeable random vectors},
        date={1992},
     journal={Annals of Probability},
      volume={20},
      number={3},
       pages={1147\ndash 1166},
}

\bib{DupLasRam15}{article}{
      author={Dupuis, P.},
      author={Laschos, V.},
      author={Ramanan, K.},
       title={Large deviations for configurations generated by {G}ibbs
  distributions with energy functionals consisting of singular interaction and
  weakly confining potentials},
        date={2020},
     journal={Electronic Journal of Probability},
      volume={25},
}

\bib{FreTuk74}{article}{
      author={Freedman, J.},
      author={Tukey, J.W.T.},
       title={A projection pursuit algorithm for exploratory data analysis},
        date={1974},
     journal={IEEE Transactions on Computers},
      volume={9},
       pages={881\ndash 890},
}

\bib{Fre19}{article}{
      author={Fresen, D.~J.},
       title={A simplified proof of {CLT} for convex bodies},
        date={2021},
     journal={Proceedings of the American Mathematical Society},
      volume={149},
      number={1},
       pages={345\ndash 354},
}

\bib{GanKimRam16}{incollection}{
      author={Gantert, N.},
      author={Kim, S.~S.},
      author={Ramanan, K.},
       title={Cram\'er's theorem is atypical},
        date={2016},
   booktitle={Advances in the {M}athematical {S}ciences: Research from the 2015
  {A}ssociation for {W}omen in {M}athematics {S}ymposium},
      editor={Letzter, Gail},
      editor={Lauter, Kristin},
      editor={Chambers, Erin},
      editor={Flournoy, Nancy},
      editor={Grigsby, Julia~Elisenda},
      editor={Martin, Carla},
      editor={Ryan, Kathleen},
      editor={Trivisa, Konstantina},
   publisher={Springer International Publishing},
     address={Cham},
       pages={253\ndash 270},
         url={http://dx.doi.org/10.1007/978-3-319-34139-2_11},
}

\bib{GanKimRam18}{article}{
      author={Gantert, N.},
      author={Kim, S.~S.},
      author={Ramanan, K.},
       title={Large deviations for random projections of $\ell^p$ balls},
        date={2017},
     journal={Annals of Probability},
      volume={45},
       pages={4419\ndash 4476},
}

\bib{GueMil11}{article}{
      author={Gu{\'e}don, O.},
      author={Milman, E.},
       title={Interpolating thin-shell and sharp large-deviation estimates for
  lsotropic log-concave measures},
        date={2011},
     journal={Geometric and Functional Analysis},
      volume={21},
      number={5},
       pages={1043},
}

\bib{HirLem13}{book}{
      author={Hiriart-Urruty, Jean-Baptiste},
      author={Lemar{\'e}chal, Claude},
       title={Convex analysis and minimization algorithms {II}},
   publisher={Springer science \& business media},
        date={2013},
      volume={305},
}

\bib{IndMot99}{incollection}{
      author={Indyk, P.},
      author={Motwani, R.},
       title={Approximate nearest neighbors: towards removing the curse of
  dimensionality},
        date={1999},
   booktitle={S{TOC} '98 ({D}allas, {TX})},
   publisher={ACM, New York},
       pages={604\ndash 613},
}

\bib{JohLin84}{incollection}{
      author={Johnson, W.~B.},
      author={Lindenstrauss, J.},
       title={Extensions of {L}ipschitz mappings into a {H}ilbert space},
        date={1984},
   booktitle={Conference in modern analysis and probability ({N}ew {H}aven,
  {C}onn., 1982)},
      series={Contemp. Math.},
      volume={26},
   publisher={Amer. Math. Soc., Providence, RI},
       pages={189\ndash 206},
         url={http://dx.doi.org/10.1090/conm/026/737400},
}

\bib{JohPro20}{article}{
      author={{Johnston}, S.},
      author={{Prochno}, J.},
       title={{A {M}axwell principle for generalized {O}rlicz balls}},
        date={2020},
     journal={arXiv e-prints},
       pages={arXiv:2012.11568},
      eprint={2012.11568},
}

\bib{KabPro21}{article}{
      author={Kabluchko, Z.},
      author={Prochno, J.},
       title={The maximum entropy principle and volumetric properties of
  {O}rlicz balls},
        date={2021},
     journal={Journal of Mathematical Analysis and Applications},
      volume={495},
      number={1},
       pages={124687},
}

\bib{Kabluchko17}{article}{
      author={Kabluchko, Z.},
      author={Prochno, J.},
      author={Th{\"a}le, C.},
       title={High-dimensional limit theorems for random vectors in
  $\ell_p^n$-balls},
        date={2019},
     journal={Communications in Contemporary Mathematics},
      volume={21},
      number={1},
       pages={1750092},
}

\bib{ThaKabPro19}{article}{
      author={Kabluchko, Z.},
      author={Prochno, J.},
      author={Th\"{a}le, C.},
       title={High-dimensional limit theorems for random vectors in
  $\ell_p^n$-balls. {II}},
        date={2021},
     journal={Communications in Contemporary Mathematics},
      volume={23},
      number={3},
       pages={1950073},
}

\bib{KanLakbook01}{book}{
      author={Kannan, D.},
      author={Lakshmikanthan, V.},
       title={Handbook of stochastic analysis and applications},
     edition={First},
      series={Statistics: A Series of Textbooks and Monographs},
   publisher={CRC Press},
        date={2001},
}

\bib{KLSconj95}{article}{
      author={Kannan, R.},
      author={Lov{\'a}sz, L.},
      author={Simonovits, M.},
       title={Isoperimetric problems for convex bodies and a localization
  lemma},
        date={1995},
     journal={Discrete and Computational Geometry},
      volume={13},
      number={3-4},
       pages={541\ndash 559},
}

\bib{skim-thesis}{thesis}{
      author={Kim, S.~S.},
       title={Problems at the interface of probability and convex geometry:
  Random projections and constrained processes},
        type={Ph.D. Thesis},
        date={2017},
}

\bib{KimRam18}{article}{
      author={Kim, S.~S.},
      author={Ramanan, K.},
       title={A conditional limit theorem for high-dimensional
  $\ell^p$-spheres},
        date={2018},
     journal={Journal of Applied Probability},
      volume={55},
       pages={1060\ndash 1077},
}

\bib{Kla07a}{article}{
      author={Klartag, B.},
       title={A central limit theorem for convex sets},
        date={2007},
     journal={Inventiones mathematicae},
      volume={168},
       pages={91\ndash 131},
}

\bib{Kla07}{article}{
      author={Klartag, B.},
       title={Power-law estimates for the central limit theorem for convex
  sets},
        date={2007},
        ISSN={0022-1236},
     journal={Journal of Functional Analysis},
      volume={245},
      number={1},
       pages={284\ndash 310},
         url={http://dx.doi.org/10.1016/j.jfa.2006.12.005},
}

\bib{Kle13}{book}{
      author={Klenke, A.},
       title={{P}robability {T}heory: {A} {C}omprehensive {C}ourse},
   publisher={Springer Science \& Business Media},
        date={2013},
}

\bib{KolMil18}{article}{
      author={Kolesnikov, A.~V.},
      author={Milman, E.},
      author={others},
       title={The {KLS} isoperimetric conjecture for generalized {O}rlicz
  balls},
        date={2018},
     journal={Annals of Probability},
      volume={46},
      number={6},
       pages={3578\ndash 3615},
}

\bib{LiaRam21b}{article}{
      author={{Liao}, Y.-T.},
      author={{Ramanan}, K.},
       title={{Quenched large deviations for random projections of
  high-dimensional measures}},
        note={In preparation},
}

\bib{LiaRam21a}{article}{
      author={{Liao}, Y.-T.},
      author={{Ramanan}, K.},
       title={{Sharp large deviation estimate and its application to asymptotic
  convex geometry}},
        note={Preprint},
}

\bib{LiaRam20a}{article}{
      author={{Liao}, Y.-T.},
      author={{Ramanan}, K.},
       title={{Geometric sharp large deviations for random projections of
  $\ell_p^n$ spheres and balls}},
        date={2020},
     journal={arXiv e-prints},
       pages={arXiv:2001.04053},
      eprint={2001.04053},
}

\bib{LiuWu20}{article}{
      author={Liu, W.},
      author={Wu, L.},
       title={Large deviations for empirical measures of mean-field {G}ibbs
  measures},
        date={2020},
     journal={Stochastic Processes and their Applications},
      volume={130},
      number={2},
       pages={503\ndash 520},
}

\bib{MaiMun09}{incollection}{
      author={Maillard, O.},
      author={Munos, R.},
       title={Compressed least-squares regression},
        date={2009},
   booktitle={Advances in {N}eural {I}nformation {P}rocessing {S}ystems 22},
      editor={Bengio, Y.},
      editor={Schuurmans, D.},
      editor={Lafferty, J.~D.},
      editor={Williams, C. K.~I.},
      editor={Culotta, A.},
   publisher={Curran Associates, Inc.},
       pages={1213\ndash 1221},
  url={http://papers.nips.cc/paper/3698-compressed-least-squares-regression.pdf},
}

\bib{Mec12b}{incollection}{
      author={Meckes, E.},
       title={Projections of probability distributions: A measure-theoretic
  {D}voretzky theorem},
        date={2012},
   booktitle={Geometric {A}spects of {F}unctional {A}nalysis},
      series={Lecture Notes in Mathematics},
      volume={2050},
   publisher={Springer},
       pages={317\ndash 326},
}

\bib{Mul59}{article}{
      author={Muller, M.~E.},
       title={A note on a method for generating points uniformly on
  n-dimensional spheres},
        date={1959},
        ISSN={0001-0782},
     journal={Communications of the ACM},
      volume={2},
      number={4},
       pages={19\ndash 20},
         url={http://doi.acm.org/10.1145/377939.377946},
}

\bib{Mus83}{book}{
      author={Musielak, J.},
       title={Orlicz {S}paces and {M}odular {S}paces},
      series={Lecture Notes in Mathematics},
   publisher={Springer},
        date={1983},
      volume={1034},
}

\bib{Nagaev69}{article}{
      author={Nagaev, A~V},
       title={Integral limit theorems taking large deviations into account when
  {C}ram{\'e}r's condition does not hold. {I}},
        date={1969},
     journal={Theory of Probability and its Applications},
      volume={14},
      number={1},
       pages={51\ndash 64},
}

\bib{Pao06}{article}{
      author={Paouris, G.},
       title={Concentration of mass on convex bodies},
        date={2006},
     journal={Geometric and Functional Analysis},
      volume={16},
      number={5},
       pages={1021\ndash 1049},
}

\bib{PapRagTamVem00}{article}{
      author={Papadimitriou, C.~H.},
      author={Raghavan, P.},
      author={Tamaki, H.},
      author={Vempala, S.},
       title={Latent semantic indexing: a probabilistic analysis},
        date={2000},
        ISSN={0022-0000},
     journal={J. Comput. System Sci.},
      volume={61},
      number={2},
       pages={217\ndash 235},
         url={http://dx.doi.org/10.1006/jcss.2000.1711},
        note={Special issue on the Seventeenth ACM SIGACT-SIGMOD-SIGART
  Symposium on Principles of Database Systems (Seattle, WA, 1998)},
      review={\MR{1802556}},
}

\bib{PilWoj08}{article}{
      author={Pilipczuk, M.},
      author={Wojtaszczyk, J.~O.},
       title={The negative association property for the absolute values of
  random variables equidistributed on a generalized {O}rlicz ball},
        date={2008},
        ISSN={1572-9281},
     journal={Positivity},
      volume={12},
      number={3},
       pages={421\ndash 474},
         url={http://dx.doi.org/10.1007/s11117-007-2116-4},
}

\bib{RacRus91}{article}{
      author={Rachev, S.~T.},
      author={R\"{u}schendorf, L.},
       title={Approximate independence of distributions on spheres and their
  stability properties},
        date={1991},
     journal={Annals of Probability},
      volume={19},
      number={3},
       pages={1311\ndash 1337},
}

\bib{SchZin90}{article}{
      author={Schechtman, G.},
      author={Zinn, J.},
       title={On the volume of the intersection of two ${L}_p^n$ balls},
        date={1990},
     journal={Proceedings of the American Mathematical Society},
      volume={110},
      number={1},
       pages={217\ndash 224},
}

\bib{Sud78}{article}{
      author={Sudakov, V.~N.},
       title={Typical distributions of linear functionals in finite-dimensional
  spaces of high dimension},
        date={1978},
     journal={Doklady Akademii Nauk SSSR},
      volume={243},
      number={6},
       pages={1402\ndash 1405},
}

\bib{Vil08}{book}{
      author={Villani, C.},
       title={{O}ptimal {T}ransport: {O}ld and {N}ew},
   publisher={Springer Science \& Business Media},
        date={2008},
      volume={338},
}

\bib{vonWei97}{article}{
      author={von Weizs{\"a}cker, H.},
       title={Sudakov's typical marginals, random linear functionals and a
  conditional central limit theorem},
        date={1997},
     journal={Probability Theory and Related Fields},
      volume={107},
      number={3},
       pages={313\ndash 324},
}

\bib{WanWanWu10}{article}{
      author={Wang, R.},
      author={Wang, X.},
      author={Wu, L.},
       title={{S}anov's theorem in the {W}asserstein distance: a necessary and
  sufficient condition},
        date={2010},
     journal={Statistics and Probability Letters},
      volume={80},
      number={5},
       pages={505\ndash 512},
}

\bib{Zia09}{article}{
      author={Zia, R.},
      author={Redish, E.},
      author={McKay, S.},
       title={Making sense of the {L}egendre transform},
        date={2009},
     journal={American Journal of Physics},
      volume={77},
      number={7},
       pages={614\ndash 622},
}

\end{biblist}
\end{bibdiv}

\vskip 8mm
\font\tenrm =cmr10  {\tenrm 
\parskip 0mm
 
\noindent Division of Applied Mathematics 

\noindent Brown University

\noindent 182 George St.

\noindent Providence, RI 02912

\bigskip 

\noindent steven\_kim@alumni.brown.edu

\noindent yin\_ting-liao@brown.edu

\noindent kavita\_ramanan@brown.edu

}

\end{document}